\documentclass[11pt]{amsart}
\usepackage[utf8]{inputenc}
\usepackage{notes-style}
\usepackage{amsaddr}
\usepackage[backend=biber,maxbibnames=6]{biblatex}
\bibliography{./ref}

\title{Bow varieties: Stable envelopes and their 3d mirror symmetry} 

\author{T. M. Botta$^{\diamond \bullet}$ and R. Rim\'anyi$^\star$}

\address{$^\diamond$ Department of Mathematics, ETH Z\"urich, Switzerland \\ 
$^\bullet$ Department of Mathematics, Columbia University, USA \\
$^\star$ Department of Mathematics, University of North Carolina at Chapel Hill, USA}
\email{tommaso.botta@columbia.edu}
\email{rimanyi@email.unc.edu}

\subjclass[2020]{55N34, 81T30, 14J33, 20G42, 14C17}

\begin{document} 

\maketitle

\begin{abstract}
    We study the elliptic characteristic classes known as ``stable envelopes'', which were introduced by M.~Aganagic and A.~Okounkov. We prove that for a rich class of holomorphic symplectic varieties---called Cherkis bow varieties---their elliptic stable envelopes exhibit a duality property inspired by mirror symmetry in $d=3$, $\mathcal N=4$ quantum field theories. A crucial step of our proof involves the process of ``resolving'' large charge branes into multiple smaller charge branes. For D5 branes this phenomenon turns out to be the geometric counterpart of the algebraic fusion procedure, and for NS5 branes this is a new mirror construction.
\end{abstract}

\tableofcontents

\section{Introduction}

Three dimensional quantum field theories predict unexpected relations between seemingly unrelated fields. One of these relations that can be phrased in concrete mathematical terms ($\vartheta$-function identities) concerns a class of equivariant characteristic classes, namely the {\em elliptic stable envelopes} introduced by  Aganagic and Okounkov \cite{aganagic2016elliptic}. In this paper we prove the 3d mirror symmetry of elliptic stable envelopes on type A bow~varieties. These varieties include cotangent bundles of partial flag varieties, type A Nakajima quiver varieties, and much more. Hardly any non-Abelian 3d mirror symmetry statements were know before---our result is new even for 2-step flag varieties. 

Another achievement---discussed below---is a geometric construction underpinning the so-called algebraic fusion procedure of R-matrices. In our language the geometry arises from the ``resolution of D5 branes''. Moreover, we discover a 3d mirror dual geometric fusion, the ``resolution of NS5 branes''. The two kinds of 5-brane resolutions are promising new geometric techniques in a variety of areas including curve counting, vertex functions, and quantum bispectral Schur-Weyl duality. 

\subsection{Stable envelopes}
A key concept of this paper, the “stable envelope”, is due to A.~Okounkov and his co-authors  \cite{maulik2012quantum, Okounkov_lectures, aganagic2016elliptic}. It is that rare notion in mathematics that not only organizes and generalizes several earlier mathematical objects (in this case, from enumerative geometry and representation theory), but also opens up a plethora of new connections, found in places like geometry, quantum integrable systems, KZ-type PDE’s, difference equations, quantum groups, and quantum cohomology. 

Stable envelopes are equivariant characteristic classes assigned to the fixed point components of certain smooth holomorphic symplectic varieties with a torus action. They come in three flavors: cohomological, K theoretic, and elliptic. They depend on some choices: a ``chamber'' and an ``attracting bundle'' (in the elliptic case) or ``alcove'' (in K theory). 


Instead of a precise definition of stable envelopes, let us get familiar with them through an example: The natural $\At=(\Cs)^2$ action on $\C^2$ induces $\At$ actions on $\mathbb P^1$ and on $X=T^*\mathbb P^1$. An additional $\Cs_{\h}$ acts on $X$ scaling the fibers.  The set of $\Tt=\At\times \Cs_{\h}$ fixed points on $X$ is two points, $f_1=(1:0)$ and $f_2=(0:1)$, and there are two possible choices of chambers $\mC_1$  and $\mC_2$. For an appropriate choice of attracting bundle, the {\em cohomological} stable envelopes and their fixed point restrictions are

\smallskip


\begin{tabular}{lllll}
$\Stab{\mC_1}{f_1}=t-a_2$, & & & $(\ )|_{f_1}=a_1-a_2$ & $(\ )|_{f_2}=0$, \\
$\Stab{\mC_1}{f_2}=a_1-t+\h$, & & & $(\ )|_{f_1}=\h$ & $(\ )|_{f_2}=a_1-a_2+\h$, \\
$\Stab{\mC_2}{f_1}=a_2-t+\h$, & & & $(\ )|_{f_1}=a_2-a_1+\h$ & $(\ )|_{f_2}=\h$, \\
$\Stab{\mC_2}{f_2}=t-a_1$, & & & $(\ )|_{f_1}=0$ & $(\ )|_{f_2}=a_2-a_1$,
\end{tabular}

\smallskip

\noindent while the {\em elliptic} stable envelopes and their fixed point restrictions are
\begin{eqnarray*}
 \Stab{\mC_1}{f_1}=\frac{
 \te(t/a_1\cdot \h^{-1}\cdot z_1/z_2)\te(t/a_2)}
 {\te(z_1/z_2\cdot h^{-1})},
 & &
 \begin{array}{l}
 (\ )|_{f_1}=\te(a_1/a_2) \\
 (\ )|_{f_2}=0,
 \end{array}
\\
 \Stab{\mC_1}{f_2}=\frac{
 \te(t/a_2\cdot z_1/z_2)\te(a_1/t\cdot \h)}
 {\te(z_1/z_2)},
 & &
\begin{array}{l}
 (\ )|_{f_1}=\te(a_1/a_2\cdot z_1/z_2)\te(\h)/\te(z_1/z_2) \\
 (\ )|_{f_2}=\te(a_1/a_2\cdot \h),
 \end{array}
 \\
\Stab{\mC_2}{f_1}=\frac{
 \te(t/a_1\cdot z_1/z_2)\te(a_2/t \cdot \h)}
 {\te(z_1/z_2)},
 & &
 \begin{array}{l}
 (\ )|_{f_1}=\te(a_2/a_1\cdot \h) \\
 (\ )|_{f_2}=\te(a_2/a_1\cdot z_1/z_2)\te(\h)/\te(z_1/z_2),
 \end{array}
 \\
 \Stab{\mC_2}{f_2}=\frac{
 \te(t/a_2\cdot \h^{-1} \cdot z_1/z_2)\te(t/a_1)}
 {\te(z_1/z_2\cdot  \h^{-1})},
 & &
\begin{array}{l}
 (\ )|_{f_1}=0 \\
 (\ )|_{f_2}=\te(a_2/a_1).
 \end{array}
\end{eqnarray*}
The variable $t$ is the Chern class of the tautological line bundle over $\mathbb P^1$, and $a_1,a_2, \h$ are the Chern roots of $\At \times \Cs_{\h}$. The variables $z_1,z_2$ (in the elliptic case) are called K\"ahler (or dynamical) variables, $\te$ is the odd Jacobi theta function written multiplicatively, $\te(1)=0$. Restrictions to the fixed points $f_1$ and $f_2$ are achieved by the substitutions $t\mapsto a_1, t\mapsto a_2$. 
 
\medskip

There are multiple roles stable envelopes play in geometry and related fields. Now we sketch a few of these, and refer the reader to the introductions of \cite{maulik2012quantum, aganagic2016elliptic} for more. 


First, stable envelopes ``geometrize'' quantum group representations \cite{maulik2012quantum}. Namely, tensor products of fundamental representations of quantum groups are identified with cohomology rings of moduli spaces by matching ``eigenvectors'' with the classes of fixed points. Under such an identification, the coordinate-basis of the quantum group representation is matched with the stable envelopes. We invite the reader to calculate the primary feature of this identification: {\em the matrix 
$(\Stab{\mC_1}{f_j}|_{f_i})^{-1}  (\Stab{\mC_2}{f_j}|_{f_i})$ (in cohomology) is Yang's R-matrix}. 

Second, when $X$ is a cotangent bundle of a homogeneous space $M$, the $\h=1$ substitution in a cohomological stable envelope gives ($\pm1$ times) the so-called equivariant Chern-Schwartz-McPherson (CSM) class of a Schubert cell in $M$ \cite{RRAV, AMSS_csm}. CSM classes are designed to generalize the total Chern class of the tangent bundle for locally closed varieties with singularities, in a ``motivic'' way. The $M=\mathbb P^1$ example above gives $1+t-a_2$ and $t-a_1$ as CSM classes of the two Schubert cells $\C, \{ pt\} \subset \mathbb P^1$, whose sum is $1+2t-a_1-a_2=c(T\mathbb P^1)$. The K theoretic and elliptic versions of CSM classes  match the appropriate versions of stable envelopes, see \cite{Feher_2020, AMSS_mc, RW_elliptic, kumar2020}.
 
Third, to a torus fixed point of an appropriate holomorphic symplectic variety (let us assume the fixed points are isolated) one associates a curve counting ``vertex function''. Vertex functions are power series in K\"ahler variables $z$, with coefficients in $K_{\Tt}(X)$ 
that count ``quasi-maps'' in an appropriate sense. These curve counting functions satisfy difference equations (of the type q-KZ) in the equivariant $a$ variables and so-called dynamical difference equations in the $z$ variables. The matrix $(\Stab{}{f}|_g)_{f,g}$ of fixed point restrictions of elliptic stable envelopes serves as the {\em monodromy} matrix for the difference equations \cite{aganagic2016elliptic, okounkov2020inductiveII, okounkovsmirnov2016Kthy}.

\subsection{3d mirror symmetry} \label{sec:intro_3d}
The phenomenon called ``mirror symmetry'' has been a remarkable motivation for algebraic geometers for decades. Its physical origin is the fact that certain $d=2$, $\mathcal N=(2,2)$ quantum field theories can be ``twisted'' two different ways (topological A and B twists) to obtain topological quantum field theories, whose mathematical incarnations are then expected to be related. Three dimensional theories have rich mathematical counterparts too, see eg. \cite{BKK, PSZ, KPSZ}. It turns out that appropriate $d=3$, $\mathcal N=4$ supersymmetric $\sigma$-models also have two topological twists \cite{IntSei, HW, BHOOY, GaiottoWitten, bullimore}. The occurring duality for mathematical objects \cite{NakajimaTowards,BFN,KZ,kamnitzer} derived from the two TQFTs is called {\em 3d mirror symmetry} or {\em symplectic duality}. For a recent survey on 3d mirror symmetry see \cite{WebsterYoo}.

One of the key predictions of 3d mirror symmetry concerns the curve counting functions mentioned above. Namely, the Higgs branches $X, X^!$ of 3d mirror dual theories (or equivalently, the Higgs and Coulomb branches of the same theory) are expected to have the same curve counting functions with $a$ and $z$ variables interchanged and $\h$ inverted. As a consequence, or essentially equivalently, the monodromies of their governing difference equations are expected to be equal---after transposition, switching $a \leftrightarrow z$ variables, and inverting $\h$. This latter property is the topic of our paper:
\begin{equation}\label{eq:Intro3d}
\frac{\Stab{}{f}|_{g} (a,z,\h)}{\Stab{}{g}|_{g} (a,z,\h)} = \pm \frac{\Stab{}{g^!}|_{f^!} (z,a,\h^{-1})}{\Stab{}{f^!}|_{f^!} (z,a,\h^{-1})}
\end{equation}
for elliptic stable envelopes, and the appropriate choices of chamber and attractive bundle. The essence is the numerators; the ``diagonal restrictions'' in the denominators are explicit formulas with geometric meaning, they could be omitted by a different normalization. Note that, at the very least, 3d mirror symmetry requires a bijection $f\leftrightarrow f^!$ between the torus fixed points of $X$ and $X^!$. 

Equation \eqref{eq:Intro3d} as the 3d mirror symmetry property for stable envelopes appeared in a lecture of A.~Okounkov \cite{Okounkov_video}, and then in the literature \cite{RSVZ_Gr3d}.

For example, the variety $X=T^*\mathbb P^1$ is self dual with $f_1^!=f_2, f_2^!=f_1$. Consider the chamber $\mC_1$ and the elliptic Stab formulas above. Then \eqref{eq:Intro3d} reduces to verifying that $\te(a_1/a_2\cdot z_1/z_2)\te(h) / ( \te(a_1/a_2) \te(z_1/z_2) )$ gets multiplied by $-1$ after substituting $a_i \leftrightarrow z_i$, $\h \leftrightarrow \h^{-1}$.

Other examples are much less obvious: the Nakajima varieties corresponding to the quivers 
\begin{equation}\label{random_quiver_pair}
\begin{tikzpicture}[baseline=1, scale=.4]
\draw[thick] (0,1) -- (7.5,1);
\draw[fill] (0,1) circle (3pt);     \node at (0,1.6) {$\scriptstyle 1$} ;
\draw[fill] (1.5,1) circle (3pt);  \node at (1.5,1.6) {$\scriptstyle 2$} ;
\draw[fill] (3,1) circle (3pt);     \node at (3,1.6) {$\scriptstyle 2$} ;
\draw[fill] (4.5,1) circle (3pt);  \node at (4.5,1.6) {$\scriptstyle 3$} ;
\draw[fill] (6,1) circle (3pt);  \node at (6,1.6) {$\scriptstyle 2$} ;
\draw[fill] (7.5,1) circle (3pt);  \node at (7.5,1.6) {$\scriptstyle 1$} ;
\draw[thick] (1.5,1) -- (1.5,0.1); \draw[thick] (1.4,-0.1) rectangle (1.6,0.1); \node at (1.5,-.6) {$\scriptstyle 1$};
\draw[thick] (4.5,1) -- (4.5,0.1); \draw[thick] (4.4,-0.1) rectangle (4.6,0.1); \node at (4.5,-.6) {$\scriptstyle 2$};
\draw[thick] (6,1) -- (6,0.1); \draw[thick] (5.9,-0.1) rectangle (6.1,0.1); \node at (6,-.6) {$\scriptstyle 1$};
\draw[thick] (7.5,1) -- (7.5,0.1); \draw[thick] (7.4,-0.1) rectangle (7.6,0.1); \node at (7.5,-.6) {$\scriptstyle 1$};
\end{tikzpicture}
\qquad \text{and} \qquad 
\begin{tikzpicture}[baseline=1, scale=.4]
\draw[thick] (0,1) -- (4.5,1);
\draw[fill] (0,1) circle (3pt);     \node at (0,1.6) {$\scriptstyle 2$} ;
\draw[fill] (1.5,1) circle (3pt);  \node at (1.5,1.6) {$\scriptstyle 3$} ;
\draw[fill] (3,1) circle (3pt);     \node at (3,1.6) {$\scriptstyle 2$} ;
\draw[fill] (4.5,1) circle (3pt);  \node at (4.5,1.6) {$\scriptstyle 1$} ;
\draw[thick] (0,1) -- (0,0.1); \draw[thick] (-0.1,-0.1) rectangle (0.1,0.1); \node at (0,-.6) {$\scriptstyle 3$};
\draw[thick] (1.5,1) -- (1.5,0.1); \draw[thick] (1.4,-0.1) rectangle (1.6,0.1); \node at (1.5,-.6) {$\scriptstyle 2$};
\draw[thick] (3,1) -- (3,0.1); \draw[thick] (2.9,-0.1) rectangle (3.1,0.1); \node at (3,-.6) {$\scriptstyle 1$};
\draw[thick] (4.5,1) -- (4.5,0.1); \draw[thick] (4.4,-0.1) rectangle (4.6,0.1); \node at (4.5,-.6) {$\scriptstyle 1$};
\end{tikzpicture}
\end{equation}
are also 3d mirror duals. They have complex dimensions 16 and 22, respectively. They have a 6- and an 8-torus acting on them, respectively. Yet, both have exactly 1055 fixed points and with the right choice of bijection between the fixed points \eqref{eq:Intro3d} holds. (We will revisit the combinatorics of this example in Section~\ref{sec:IntBow}.)
The pair of quivers we just discussed had to be chosen carefully: not all quivers have a 3d mirror quiver---this phenomenon will be illuminated soon. 

There are two infinite families of pairs of quiver varieties for which \eqref{eq:Intro3d} is proved already. The first class is
\[
\begin{tikzpicture}[baseline=1, scale=.4]
\draw[thick] (1.5,1) -- (7.5,1);
\draw[fill] (1.5,1) circle (3pt);  \node at (1.5,1.6) {$\scriptstyle 1$} ;
\draw[fill] (3,1) circle (3pt);     \node at (3,1.6) {$\scriptstyle 2$} ;
       \node at (4.2,1.6) {$\cdots$} ;
\draw[fill] (6,1) circle (3pt);  \node at (5.7,1.6) {$\scriptstyle n-2$} ;
\draw[fill] (7.5,1) circle (3pt);  \node at (7.6,1.6) {$\scriptstyle n-1$} ;
\draw[thick] (7.5,1) -- (7.5,0.1); \draw[thick] (7.4,-0.1) rectangle (7.6,0.1); \node at (7.5,-.6) {$\scriptstyle n$};
\end{tikzpicture}
\qquad\text{versus} \qquad 
\begin{tikzpicture}[baseline=1, scale=.4]
\draw[thick] (0,1) -- (6,1);
\draw[fill] (0,1) circle (3pt);     \node at (-.3,1.6) {$\scriptstyle n-1$} ;
\draw[fill] (1.5,1) circle (3pt);  \node at (1.5,1.6) {$\scriptstyle n-2$} ;
       \node at (3.5,1.6) {$\cdots$} ;
\draw[fill] (4.5,1) circle (3pt);  \node at (4.5,1.6) {$\scriptstyle 2$} ;
\draw[fill] (6,1) circle (3pt);  \node at (6,1.6) {$\scriptstyle 1$} ;
\draw[thick] (0,1) -- (0,0.1); \draw[thick] (-.1,-0.1) rectangle (0.1,0.1); \node at (0,-.6) {$\scriptstyle n$};
\end{tikzpicture},
\]
both isomorphic to the cotangent bundle of the full flag variety on $\C^n$ \cite{rimanyi20193d, RRAW_Langlands}. The second class is \cite{RSVZ_Gr3d} 
\[
T^*\Gr{k}{n}=
\begin{tikzpicture}[baseline=3, scale=.4]
\draw[fill] (0,1) circle (3pt);     
\node at (0,1.6) {$\scriptstyle k$} ;
\draw[thick] (0,1) -- (0,0.1); 
\draw[thick] (-.1,-0.1) rectangle (0.1,0.1); 
\node at (0,-.6) {$\scriptstyle n$};
\end{tikzpicture}
\qquad\text{versus} \qquad 
\begin{tikzpicture}[baseline=3, scale=.4]
\draw[thick] (0,1) -- (15,1);
\draw[fill] (0,1) circle (3pt);     \node at (0,1.6) {$\scriptstyle 1$} ;
\draw[fill] (1.5,1) circle (3pt);  \node at (1.5,1.6) {$\scriptstyle 2$} ;
 \node at (3,1.6) {$\cdots$} ;
\draw[fill] (4.5,1) circle (3pt);  \node at (4.5,1.6) {$\scriptstyle k$} ;
\draw[fill] (6,1) circle (3pt);  \node at (6,1.6) {$\scriptstyle k$} ;
\node at (7.5,1.6) {$\cdots$} ;
\draw[fill] (9,1) circle (3pt);  \node at (9,1.6) {$\scriptstyle k$} ;
\draw[fill] (10.5,1) circle (3pt);  \node at (10.5,1.6) {$\scriptstyle k$} ;
\draw[fill] (12,1) circle (3pt);  \node at (12,1.6) {$\scriptstyle k-1$} ;
\node at (13.6,1.6) {$\cdots$} ;
\draw[fill] (15,1) circle (3pt);  \node at (15,1.6) {$\scriptstyle 1$} ;
\draw[thick] (4.5,1) -- (4.5,0.1); \draw[thick] (4.4,-0.1) rectangle (4.6,0.1); \node at (4.5,-.6) {$\scriptstyle 1$};
\draw[thick] (10.5,1) -- (10.5,0.1); \draw[thick] (10.4,-0.1) rectangle (10.6,0.1); \node at (10.5,-.6) {$\scriptstyle 1$};
\end{tikzpicture}
\]
for $k\leq n/2$. On the right there are $n-2k+1$ copies of $k$ on the top (if $n=2k$ then the two bottom $1$s collide to be a $2$).  

Why exactly these quivers are mirrors of each other? How to find the mirror of other quivers?
These questions will be answered next.

\subsection{Bow varieties and the Main Theorem} \label{sec:IntBow}

Bow varieties $\Ch(\DD)$ form a remarkable class of smooth, holomorphic complex varieties that generalize type A Nakajima quiver varieties. They were introduced by Cherkis as moduli spaces of unitary instantons on multi-Taub-NUT spaces \cite{cherkis1, cherkis2, cherkis3}; and were given a quiver-like presentation by Nakajima and Takayama \cite{Nakajima_Takayama, takayama_2016}, see also \cite{RimRoz}. We will use \cite{rimanyi2020bow} as our general reference. 

The combinatorial code of a bow variety is a {\em brane diagram} 
$\DD=\ttt{0\fs 1\fs 3\fs 3\bs 1\bs 0}$,
that is, a finite sequence of NS5 branes ($\ttt{\fs}$) and D5 branes ($\ttt{\bs}$), separated by integers (the dimension vector, starting and ending with 0, often omitted). Extra flexibility is provided by a local surgery, called Hanany-Witten (HW) transition $\ttt{$d_1$\fs $d_2$\bs $d_3$} \stackrel{HW}{\leftrightarrow} \ttt{$d_1$\bs $d_1+d_3-d_2+1$\fs $d_3$}$, under which the associated bow variety is unchanged. A brane diagram is balanced if the integers on the two sides of any D5 brane are equal. Nakajima quiver varieties correspond to brane diagrams HW equivalent to balanced ones, for example the $\DD$ above: $\ttt{\fs 1\fs 3\fs 3\bs 1\bs} \stackrel{HW}{\leftrightarrow}  \ttt{\fs 1\bs 1\fs 2\bs 2\fs}$ 
$=
\begin{tikzpicture}[baseline=1, scale=.4]
\draw[thick] (0,1) -- (1.5,1);
\draw[fill] (0,1) circle (3pt);     \node at (0,1.6) {$\scriptstyle 1$} ;
\draw[fill] (1.5,1) circle (3pt);  \node at (1.5,1.6) {$\scriptstyle 2$} ;
\draw[thick] (0,1) -- (0,0.1); \draw[thick] (-0.1,-0.1) rectangle (0.1,0.1); \node at (0,-.6) {$\scriptstyle 1$};
\draw[thick] (1.5,1) -- (1.5,0.1); \draw[thick] (1.4,-0.1) rectangle (1.6,0.1); \node at (1.5,-.6) {$\scriptstyle 1$};
\end{tikzpicture}
$.

Bow varieties come equipped with all the ``usual suspects'' of structures one wishes for in enumerative geometry and geometric representation theory: tautological bundles, torus action with finitely many fixed points,  combinatorial codes of fixed points (called {\em tie diagrams}), and combinatorial descriptions of fixed point restriction maps.  

Another remarkable feature of brane diagrams is the existence of an involution $\DD \leftrightarrow \DD^!$, a ``combinatorial 3d mirror symmetry'', obtained by simply swapping NS5 and D5 branes. For example $\ttt{\fs 1\fs 2\bs 2\fs 2\fs 3\bs 3\bs 3\fs 2\bs 2\fs 1\bs 1\fs}$ and $\ttt{\bs 1\bs 2\fs 2\bs 2\bs 3\fs 3\fs 3\bs 2\fs 2\bs 1\fs 1\bs}$ are 3d mirror dual brane diagrams, and in fact, they are Hanany-Witten equivalent to the pair of quivers in \eqref{random_quiver_pair}! 

There is a natural bijection between the combinatorial codes of the fixed points for $\DD$ and $\DD^!$ (simply looking at tie diagrams upside down), and in turn, a natural bijection between the torus fixed points of $\Ch(\DD)$ and $\Ch(\DD^!)$. The reader can now---correctly---guess our

\medskip
\noindent {\bf Main Theorem.} (Theorem \ref{theorem mirror symmetry stabs} below.)
{\em The varieties $\Ch(\DD)$ and $\Ch(\DD^!)$ satisfy {3d mirror symmetry for elliptic stable envelopes} in the sense of~\eqref{eq:Intro3d}.  }
\medskip

\noindent 
The theorem answers the questions in the last paragraph of Section~\ref{sec:intro_3d}. Moreover, it answers why some quivers have no mirror quivers: all quiver varieties are bow varieties (this larger class is closed for 3d mirror symmetry), but the dual of a quiver variety is not necessarily a quiver variety.

\subsection{Our proof}
Ever since quantum field theories and computational evidence predicted 3d mirror symmetry for stable envelopes, experts envisioned two strategies to prove it. 
\begin{itemize}
\item A conceptual proof would be the geometric construction of a {\em duality interface}, a kind of ``two-sided stable envelope'' on $\Ch(\DD)\times \Ch(\DD^!)$---as suggested by A. Okounkov \cite{Okounkov_video}. 
\item A computational proof would be achieved by developing formulas or recursions for the stable envelopes, and proving \eqref{eq:Intro3d} algebraically with the appropriate residue calculus and/or trisecant identities. This 
strategy is used in the known special cases \cite{RSVZ_Gr3d, rimanyi20193d, RRAW_Langlands}. 
\end{itemize}
In a subsequent paper \cite{RRTBshuffle} we will prove a shuffle property, and in turn, an Elliptic Cohomological Hall algebra structure of stable envelopes on bow varieties. As a result, explicit formulas for stable envelopes are obtained. 
Nevertheless, even if we could organize the needed residue calculus to prove \eqref{eq:Intro3d} for the formulas, such a proof would not be illuminating or satisfactory.   

Instead, in this paper we choose a different strategy: we analyze the geometric changes of a bow variety as we replace a 5-brane of charge$=\!\!\w$ with $\w$ copies of 5-branes of charge$=\!\!1$. After such ``resolutions'' of the 5-branes we arrive at a brane diagram with only charge$=\!\!1$ 5-branes, and the corresponding bow variety is the cotangent bundle of the full flag variety. For the latter 3d mirror symmetry has been proven \cite{rimanyi20193d, RRAW_Langlands}. 

Resolutions of 5-branes are remarkable operations, and we believe that their study is worthwhile independently of our proof of 3d mirror symmetry; let us explain why. 

The resolution of a D5 brane leads to a closed embedding between bow varieties. Moreover, this embedding respects a poset structure that is relevant in the definition of stable envelopes. As a consequence we obtain a ``one-term relation'' (Corollary~\ref{corollary nicest one-term D5 resolution}) between the stable envelopes before and after the D5 resolution. This relation is the basis for the known ``algebraic fusion of R-matrices'' (Section~\ref{subsection fusion of R-matrices}), hence we call our D5 brane resolution {\em geometric fusion}. Moreover, the net effect of maximally resolving all D5 branes is a remarkable embedding of the bow variety into a cotangent bundle of a full flag variety, cf. \cite{ji2024bow}.

The resolution of an NS5 brane is a procedure analogous to a forgetful map from a many-step flag variety to a fewer-step flag variety---but in the symplectic settings where the map is replaced with a convolution diagram. Such a map/convolution yields a push-forward formula for characteristic classes, and we obtain ``many-term relations'' (Proposition \ref{proposition separated NS5 fusion ratios stabs} and \ref{proposition co-separated NS5 fusion ratios stabs}) that compare stable envelopes before and after an NS5 resolution. Since an NS5 resolution is the 3d mirror of a D5 resolution, we name the former a {\em mirror geometric fusion}, and we expect that its study will lead to fusion-like statements for the dynamical Weyl group (3d mirror symmetry is expected to swap the R-matrix action and the dynamical Weyl group action in the style of ``quantum bispectral Schur-Weyl duality" \cite{smirnov2021qDiffHilb, kononov2022pursuingI, kononov2020pursuingII, dalipi2022howe}). 

The remaining ingredient of our proof is an R-matrix argument (Section~\ref{sec:Rmatrix}) which proves the consistency of the stable envelope relations coming from geometric vs mirror geometric fusions. 

\smallskip

Let us emphasize that our proof is geometric: we prove the 3d mirror symmetry of elliptic stable envelopes without providing formulas for them. Therefore the formulas obtained by the shuffle structure in the future paper \cite{RRTBshuffle} (cf.~\cite{botta2023framed, botta2021shuffle}) will automatically satisfy \eqref{eq:Intro3d}, even though that equality is not obvious algebraically. As a consequence, our geometric proof yields an abundance of theta-function identities: one for every pair of 0-1-matrices that have the same row and column sums. 

\smallskip

While the main result of the paper is the 3d mirror symmetry of elliptic stable envelopes for bow varieties, we decided not to write a paper that only contains that proof in its most compact form. Instead, we decided to
provide a detailed and rigorous treatment of elliptic stable envelopes on bow varieties, and to explore various aspects of the arising rich geometry. Examples for these explorations include the tensor structure of bow varieties, relations among their tautological bundles, a relevant quadratic form calculus, the study of geometric R-matrices, the $S_n$ action obtained by permuting the same kinds of 5-branes, as well as a discussion on what modification of the powerful theory of equivariant localization holds for bow varieties.

\bigskip

\noindent{\bf Acknowledgments.} The first author was supported as a part of NCCR SwissMAP, a National Centre of Competence in Research, funded by the Swiss National Science Foundation (grant number 205607) and grant 200021\_\-196892 of the Swiss National Science Foundation. The second author was supported by the Simons Foundation grant 5107838 as well as the NSF grants 2152309 and 2200867. We are grateful for useful discussions on the subject with A.~Buch, 
G.~Felder, L.~Rozansky, Y.~Shou, A.~Smirnov, A. Varchenko, and T. Wehrhan.

\section{Brane diagrams, bow varieties, fixed points}
\label{Section: Brane diagrams, bow varieties, fixed points}

In this section we recall the content of Sections 2 and 3 of \cite{rimanyi2020bow} in  a nutshell---yet, the reader is advised to consult that paper as well as \cite{Nakajima_Takayama} for more details. 

\subsection{Brane diagrams} 
Combinatorial objects like $\DD=\ttt{\fs 2\bs 2\fs 2\bs 4\fs 3\fs 3\fs 4\bs 3\fs 2\bs 2\bs}$ will be called (type A) brane diagrams. The red forward-leaning lines are called NS5 branes, denoted by $\Zb$. The blue backward-leaning lines are called D5 branes, denoted by $\Ab$. The positions between 5-branes are called D3 branes, denoted by $\Xb$, and the integer sitting there is called its multiplicity or dimension $d_{\Xb}$. 

If all NS5 branes are to the left of all D5 branes, we call the diagram {\em separated}. If all NS5 branes are to the right of all D5 branes, we call the diagram {\em co-separated}.

The {\em charge} of an NS5  brane $\Zb$ or a D5 brane $\Ab$ is defined by
\begin{align*}
\ch(\Zb)= & (d_{\Zb^+}-d_{\Zb^-})+|\{\text{D5 branes left of $\Zb$}\}|,
\\ 
\ch(\Ab)= & (d_{\Ab^-}-d_{\Ab^+})+|\{\text{NS5 branes right of $\Ab$}\}|.
\end{align*}
Here, the superscripts $+, -$ refer to the D3 branes directly to the right and left. 
We define the {\em local charge} (or ``{\em weight}'') of 5-branes by  $\w(\Zb)=|d_{\Zb^+}-d_{\Zb^-}|$, $\w(\Ab)=|d_{\Ab^+}-d_{\Ab^-}|$. For an NS5 brane let $\ell(\Zb)$ denote the number of D5 branes left of $\Zb$. 

\subsection{The bow variety} \label{sec:def of bow variety}
To a D3 brane we associate a complex vector space $W$ of the given dimension. To a D5 brane $\Ab$ we associate a one-dimensional space $\C_{\Ab}$ with the standard $\GL(\C_{\Ab})$ action and the ``three-way part''
\begin{align*}
\MM_{\Ab}= &\Hom(W_{{\Ab}^+},W_{{\Ab}^-}) \oplus
\h\Hom(W_{{\Ab}^+},\C_{\Ab}) \oplus \Hom(\C_{\Ab},W_{{\Ab}^-}) \\
&  \oplus\h\End(W_{{\Ab}^-}) \oplus \h\End(W_{{\Ab}^+}),
\end{align*} 
with elements denoted $(A_{\Ab}, b_{\Ab}, a_{\Ab}, B_{\Ab}, B'_{\Ab})$, and $\NN_{\Ab}=\h\Hom(W_{{\Ab}^+},W_{{\Ab}^-})$. To an NS5 brane $\Zb$ we associate the ``two-way part''
\[
\MM_{\Zb}= \h\Hom(W_{{\Zb}^+},W_{{\Zb}^-}) \oplus
\Hom(W_{{\Zb}^-},W_{{\Zb}^+}),
\]
whose elements will  be denoted by $(C_{\Zb}, D_{\Zb})$. To a D3 brane $\Xb$ we associate $\NN_{\Xb}=\h\End(W_{\Xb})$. In these formulas, the $\h$ factor means an action of an extra $\Cs$ factor called $\Cs_{\h}$.
Let 
\[
\MM=\bigoplus_{\Ab} \MM_{\Ab} \oplus \bigoplus_{\Zb} \MM_{\Zb}, \qquad \NN=\bigoplus_{\Ab} \NN_{\Ab} \oplus \bigoplus_{\Xb} \NN_{\Xb}.
\]
We define a map $\mu:\MM \to \NN$ componentwise as follows.
\begin{itemize}
\item 
The $\NN_{\Ab}$-component of $\mu$ is 
$B_{\Ab} A_{\Ab} -A_{\Ab} B'_{\Ab}+a_{\Ab} b_{\Ab}$.
\item
The $\NN_{\Xb}$-components of $\mu$ depend on the diagram:
\begin{itemize}
\item[\ttt{\bs -\bs}] If $\Xb$ is in between two D5 branes then it is $B'_{\Xb^-}-B_{\Xb^+}$. 
\item[\ttt{{\fs}-{\fs}}] If $\Xb$ is in between two NS5 branes then it is $C_{\Xb^+}D_{\Xb^+}$ $-D_{\Xb^-}C_{\Xb^-}$.
\item[\ttt{{\fs}-\bs}] If $\Xb^-$ is an NS5 brane and $\Xb^+$ is a D5 brane then it is $-D_{\Xb^-}C_{\Xb^-}$ $-B_{\Xb^+}$.
\item[\ttt{\bs -{\fs}}] If $\Xb^-$ is a  D5 brane and $\Xb^-$ is an NS5 brane then it is $C_{\Xb^+}D_{\Xb^+}+B'_{\Xb^-}$.
\end{itemize}
\end{itemize}
Let $\widetilde{\mathcal M}$ consist of points of $\mu^{-1}(0)\subset \MM$ for which the stability conditions  
\begin{itemize}
\item[(S1)]  if $S\leq W_{{\Ab}^+}$ is a subspace with $B'_{\Ab}(S)\subset S$, $A_{\Ab}(S)=0$, $b_{\Ab}(S)=0$ then $S=0$,
\item[(S2)]  if $T \leq W_{{\Ab}^-}$ is a subspace with $B_{\Ab}(T)\subset T$, $Im(A_{\Ab})+Im(a_{\Ab})\subset T$ then $T=W_{\Ab^-}$
\end{itemize}
hold for all D5 branes $\Ab$. Let $G=\prod_{\Xb} \GL(W_\Xb)$ and consider the character 
\begin{equation}\label{eq:character}
\chi: G\ \to \C^{\times}, 
\qquad\qquad
(g_\Xb)_{\Xb} \mapsto \prod_{\Xb'} \det(g_{\Xb'}),
\end{equation}
where the product runs for D3 branes $\Xb'$ such that $(\Xb')^-$ is an NS5 brane (in picture \ttt{{\fs}}$\!\Xb'$). Let $G$ act on $\widetilde{\mathcal M} \times \C$ by $g.(m,x)=(gm,\chi^{-1}(g)x)$. We say that $m\in \widetilde{\mathcal M}$ is stable (notation $m\in \widetilde{\mathcal M}^s$) if the orbit $G(m,x)$ is closed and the stabilizer of $(m,x)$ is finite for $x\not=0$. Define the bow variety $\Ch(\DD)$ associated with the brane diagram  $\DD$ to be $\widetilde{\mathcal M}^s/G$. By this definition $\Ch(\DD)$ is an orbifold, but in fact the stabilizers of stable points turn out to be trivial.

The obtained variety $\Ch(\DD)$ is smooth, holomorphic symplectic. It comes with the action of the torus $\Tt=\At\times \C_{\h}^\times$, where $\At=\times_{\Ab} \GL(\C_{\Ab})$. The vector spaces $W_{\Xb}$ associated with the D3 branes induce ``tautological'' bundles $\xi$ (of the given rank). 

\subsection{Affinization and handsaw variety} \label{sec:Affinization}

Although the stability conditions (S1) and (S2) are open, the variety $\widetilde{\mathcal M}$ is affine  \cite[Section 2]{takayama_2016}. As a consequence, bow varieties come with a projective morphism
\[
\pi:X(\D)\to X_0(\D):=\Spec(\C[\widetilde{\mathcal M}]^{G})
\]
to an affine variety. 
\begin{remark}
    The chosen stability condition implies that if $\DD$ is separated then the maps $A_{\Ab}$ are injective and the maps $C_{\Zb}$ are surjective for all $\Ab$ and $\Zb$. If instead the brane diagram $\DD$ is co-separated, then $A_{\Ab}$ must be surjective and the maps $D_{\Zb}$ are injective. As a consequence, the bow variety $C(\DD)$ is empty unless $ d_{\Ab^-}-d_{\Ab^+}\geq 0$ and $d_{\Zb^+}-d_{\Zb^-}\geq 0$ in the separated case and $ d_{\Ab^-}-d_{\Ab^+}\leq 0$ and $d_{\Zb^+}-d_{\Zb^-}\leq 0$ in the co-separated case. 
\end{remark}

Assume now that the bow diagram $\D$ is separated or co-separated. By neglecting the two-way parts of the diagram, we get a map $  \mu_{HS}:(\oplus_{\Ab} \MM_{\Ab})\to (\oplus_{\Ab} \NN_{\Ab}) \oplus (\oplus_{\Xb} \NN_{\Xb}) $.
The subscript refers to the shape of the associated quiver, which resembles a handsaw.  Let $\widetilde{ \mathcal{M}}_{HS}$ denote the subvariety of $\mu_{HS}^{-1}(0)$ satisfying the conditions (S1) and (S2). It is also affine and its quotient 
\[
HS(\D):=\Spec(\C[\widetilde{\mathcal{M}}_{HS}]^{\prod_{\Ab}\GL(W_{\Ab})})
\]
is called a handsaw variety. As shown in \cite[Prop. 2.9 and Cor. 2.21]{takayama_2016}, the $G_{\D}$ action on $\widetilde{\mathcal M}$ is free and all its orbits are closed, so $HS(\D)$ is just the orbit space $\widetilde{\mathcal{M}}_{HS}/\prod_{\Ab}\GL(W_{\Ab})$.

The natural projection $\widetilde{\mathcal{M}}\to \widetilde{\mathcal{M}}_{HS}$ descends to an affine morphism $\rho: X_0(\D)\to HS(\D)$. Altogether, if $\D$ is separated or co-separated, we have morphisms 
\begin{equation}
    \label{maps from bow to affine bow and handsaw}
    X(\D)\xrightarrow[]{\pi} X_0(\D) \xrightarrow[]{\rho} HS(\D).
\end{equation}

\subsection{Hanany-Witten transition} \label{sec:HW}
The local surgery 
\[
\ttt{$d_1$\fs $d_2$\bs $d_3$} \leftrightarrow \ttt{$d_1$\bs $d_1+d_3-d_2+1$\fs $d_3$}\] 
on diagrams is called Hanany-Witten transition. It is a fact that under such transition the associated bow varieties are isomorphic, with the change of K-theory classes of bundles, and parametrization
\begin{equation}\label{HWfig}
\begin{tikzpicture}[baseline=(current  bounding  box.center), scale=.5]
\draw[thick] (-1,1)--(6,1);
\draw[thick,red] (0.6,0)--(1.4,2);
\draw[thick,blue] (4.4,0)--(3.6,2);
\node at (0.2,2) {$\xi_1$};
\node at (2.5,2) {$\xi_2$};
\node at (5.2,2) {$\xi_3$};
\node at (4.6,-.7) {$\C_{\Ab}$};
\draw[ultra thick, <->] (8,1)--(9.5,1) node[above]{HW} -- (11,1);
\draw[thick] (13,1)--(24,1);
\draw[thick,blue] (14.6,2)--(15.4,0);
\draw[thick,red] (22.4,2)--(21.6,0);
\node at (13.8,2) {$\xi_1$};
\node at (18.5,2) {$\xi_1+\xi_3-\xi_2+\C_{\Ab}$}; 
\node at (23.2,2) {$\xi_3$};
\node at (15.3,-.7) {$\C_{\Ab}\h^{-1}$};
\end{tikzpicture}.
\end{equation}
The charges of 5-branes are invariant under HW transition, and in fact, the vectors of NS5 and D5 charges $r$ and $c$ are a complete invariant of a HW equivalence class. Any brane diagram can be made separated of co-separated by a sequence of HW moves. 

\subsection{Cotangent bundles of partial flag varieties and Nakajima quiver varieties are bow varieties \cite[\S 6.2-6.3]{rimanyi2020bow}} \label{sec:quivers}
The bow varieties for which the charge vector $c$ is non-increasing are naturally isomorphic to Nakajima quiver varieties. Among these, the bow variety with $c=(1,1,\ldots,1)$ is naturally isomorphic to the total space of the cotangent bundle of a partial flag variety, namely one with {\em dimension steps} $r_i$. In particular, the bow variety with $r=c=(1,\ldots,1)\in \N^n$ is the cotangent bundle of the full flag variety on $\C^n$.

\subsection{Torus fixed points} \label{sec:fixedpoints}
The $\Tt$ (or $\At$) action on $\Ch(\DD)$ has finitely many fixed points. The fixed points can be encoded combinatorially by {\em tie diagrams}, like this one
\begin{equation}
\begin{tikzpicture}[baseline=0,scale=.4]
\draw [thick,red] (0.5,0) --(1.5,2); 
\draw[thick] (1,1)--(2.5,1) node [above] {$2$} -- (31,1);
\draw [thick,blue](4.5,0) --(3.5,2);  
\draw [thick](4.5,1)--(5.5,1) node [above] {$2$} -- (6.5,1);
\draw [thick,red](6.5,0) -- (7.5,2);  
\draw [thick](7.5,1) --(8.5,1) node [above] {$2$} -- (9.5,1); 
\draw[thick,blue] (10.5,0) -- (9.5,2);  
\draw[thick] (10.5,1) --(11.5,1) node [above] {$4$} -- (12.5,1); 
\draw [thick,red](12.5,0) -- (13.5,2);   
\draw [thick](13.5,1) --(14.5,1) node [above] {$3$} -- (15.5,1);
\draw[thick,red] (15.5,0) -- (16.5,2);  
\draw [thick](16.5,1) --(17.5,1) node [above] {$3$} -- (18.5,1);  
\draw [thick,red](18.5,0) -- (19.5,2);  
\draw [thick](19.5,1) --(20.5,1) node [above] {$4$} -- (21.5,1);
\draw [thick,blue](22.5,0) -- (21.5,2);
\draw [thick](22.5,1) --(23.5,1) node [above] {$3$} -- (24.5,1);  
\draw[thick,red] (24.5,0) -- (25.5,2); 
\draw[thick] (25.5,1) --(26.5,1) node [above] {$2$} -- (27.5,1);
\draw [thick,blue](28.5,0) -- (27.5,2);  
\draw [thick](28.5,1) --(29.5,1) node [above] {$2$} -- (30.5,1);   
\draw [thick,blue](31.5,0) -- (30.5,2);   

\draw [dashed, black](4.5,-.25) to [out=-45,in=225] (12.5,-.25);
\draw [dashed, black](10.5,-.25) to [out=-45,in=225] (12.5,-.25);
\draw [dashed, black](10.5,-.25) to [out=-45,in=225] (15.5,-.25);
\draw [dashed, black](10.5,-.25) to [out=-45,in=225] (24.5,-.25);
\draw [dashed, black](22.5,-.25) to [out=-45,in=225] (24.5,-.25);

\draw [dashed, black](1.5,2.25) to [out=45,in=-225] (3.5,2.25);
\draw [dashed, black](1.5,2.25) to [out=45,in=-225] (9.5,2.25);
\draw [dashed, black](13.5,2.25) to [out=45,in=-225] (21.5,2.25);
\draw [dashed, black](16.5,2.25) to [out=45,in=-225] (21.5,2.25);
\draw [dashed, black](19.5,2.25) to [out=45,in=-225] (30.5,2.25);
\draw [dashed, black](25.5,2.25) to [out=45,in=-225] (30.5,2.25);
\node at (.7,-1) {\small $\Zb_1$};
\node at (6.8,-1) {\small $\Zb_2$};
\node at (4,-1) {\small $\Ab_1$};
\node at (31.7,-1) {\small $\Ab_5$};
\node at (28.5,-1) {\small $\Ab_4$};
\node at (25,-1) {\small $\Zb_6$};

\end{tikzpicture},
\end{equation}
see details in \cite{rimanyi2020bow}. Equivalently, fixed points can be encoded by {\em binary contingency tables} (BCTs), which are 01 matrices with row and column sums the charge vectors. The BCT corresponding to the tie diagram above, as well as the charge vectors $r=(2,1,1,2,3,2)$, $c=(5,2,2,0,2)$, are 
\[
\begin{tikzpicture}[baseline=1.4 cm, scale=.5]
\draw[ultra thin] (0,0) -- (5,0);
\draw[ultra thin]  (0,1) -- (5,1);
\draw[ultra thin]  (0,2) -- (5,2);
\draw[ultra thin]  (0,3) -- (5,3);
\draw[ultra thin]  (0,4) -- (5,4);
\draw[ultra thin]  (0,5) -- (5,5);
\draw[ultra thin]  (0,6) -- (5,6);
\draw[ultra thin]  (0,0) -- (0,6);
\draw[ultra thin]  (1,0) -- (1,6);
\draw[ultra thin]  (2,0) -- (2,6);
\draw[ultra thin]  (3,0) -- (3,6);
\draw[ultra thin]  (4,0) -- (4,6);
\draw[ultra thin]  (5,0) -- (5,6);
\node at (.5,6.4) {\small $5$}; \node at (1.5,6.4) {\small $2$}; \node at (2.5,6.4) {\small $2$}; \node at (3.5,6.4) {\small $0$}; \node at (4.5,6.4) {\small $2$};
\node at (.5,7.4) {\small $\Ab_1$}; \node at (1.5,7.4) {\small $\Ab_2$}; \node at (2.5,7.4) {\small $\Ab_3$}; \node at (3.5,7.4) {\small $\Ab_4$}; \node at (4.5,7.4) {\small $\Ab_5$};
\node at (-.5,0.5) {\small $2$}; \node at (-.5,1.5) {\small $3$}; \node at (-.5,2.5) {\small $2$}; \node at (-.5,3.5) {\small $1$}; \node at (-.5,4.5) {\small $1$}; \node at (-.5,5.5) {\small $2$};  
\node at (-1.8,0.5) {\small $\Zb_6$}; \node at (-1.8,1.5) {\small $\Zb_5$}; \node at (-1.8,2.5) {\small $\Zb_4$}; \node at (-1.8,3.5) {\small $\Zb_3$}; \node at (-1.8,4.5) {\small $\Zb_2$}; \node at (-1.8,5.5) {\small $\Zb_1$};  
\node[violet] at (0.5,0.5) {\small $1$};\node[violet] at (1.5,0.5) {\small $0$};\node[violet] at (2.5,0.5) {\small $0$};\node[violet] at (3.5,0.5) {\small $0$};\node[violet] at (4.5,0.5) {\small $1$};
\node[violet] at (0.5,1.5) {\small $1$};\node[violet] at (1.5,1.5) {\small $1$};\node[violet] at (2.5,1.5) {\small $0$};\node[violet] at (3.5,1.5) {\small $0$};\node[violet] at (4.5,1.5) {\small $1$};
\node[violet] at (0.5,2.5) {\small $1$};\node[violet] at (1.5,2.5) {\small $0$};\node[violet] at (2.5,2.5) {\small $1$};\node[violet] at (3.5,2.5) {\small $0$};\node[violet] at (4.5,2.5) {\small $0$};
\node[violet] at (0.5,3.5) {\small $0$};\node[violet] at (1.5,3.5) {\small $0$};\node[violet] at (2.5,3.5) {\small $1$};\node[violet] at (3.5,3.5) {\small $0$};\node[violet] at (4.5,3.5) {\small $0$};
\node[violet] at (0.5,4.5) {\small $1$};\node[violet] at (1.5,4.5) {\small $0$};\node[violet] at (2.5,4.5) {\small $0$};\node[violet] at (3.5,4.5) {\small $0$};\node[violet] at (4.5,4.5) {\small $0$};
\node[violet] at (0.5,5.5) {\small $1$};\node[violet] at (1.5,5.5) {\small $1$};\node[violet] at (2.5,5.5) {\small $0$};\node[violet] at (3.5,5.5) {\small $0$};\node[violet] at (4.5,5.5) {\small $0$};
\end{tikzpicture}.
\]
In the whole paper, we assume that $\DD$ is such that $\Ch(\DD)$ has at least one fixed point. The Gale-Ryser theorem \cite[\S 6.2.4]{BrualdiRyser} is a combinatorial characterization of the charge vectors for which this condition holds.  

\subsection{Fixed point restrictions}
Throughout the paper we will be concerned with formulas in four sets of variables:
\begin{itemize}
\item{} a variable $a_i$ for each D5 brane (``equivariant variables'');
\item{} a variable $z_i$ for each NS5 brane (``K\"ahler variables'');
\item{} variables $t^\xi_{i}$ for $i=1,\ldots,\rk(\xi)$ (``topological variables'') for the bundles $\xi$;
\item{} an extra variable $\h$.
\end{itemize}
The geometric meaning of these variables depends on the context, namely whether we are in cohomology, K theory, elliptic cohomology, or if we are recording ``factors of automorphy'' of elliptic functions. Yet, in all these settings $a_i$ are the relevant Chern roots of the factors of $\At$. The $t^{\xi}_{j}$ are the Chern roots of the bundle $\xi$.  The variable $\h$ is the Chern root of the $\C_{\h}^{\times}$. In elliptic cohomology these are coordinates on factors of the elliptic cohomology scheme---cf. Section \ref{section: Elliptic cohomology}. The $z_i$ variables are coordinates on other factors of the same scheme.

At torus fixed points, the $t^\xi_{j}$ variables specialize to monomials in the $a_i$ and $\h$ variables. The combinatorics of this specialization is involved, see the ``butterfly diagrams'' of \cite[Sec.4.3]{rimanyi2020bow}. For the fixed point in the figure above the specialization (written multiplicatively) is
\begin{multline*}
t_{1j} \mapsto a_1, a_2\h^{-1}; \qquad  
t_{2j}\mapsto a_1, a_2\h^{-1}; \qquad
t_{3j} \mapsto a_1, a_2; \qquad
t_{4j} \mapsto a_1, a_2, a_2\h, a_2\h^2; 
\\
t_{5j} \mapsto a_2\h, a_2\h^2, a_3\h^{-2}; \quad
t_{6j} \mapsto a_2\h^2, a_3\h^{-2}, a_3\h^{-1}; \quad
t_{7j} \mapsto a_2\h^2, a_3\h^{-1}, a_3, a_5\h^{-1};
\\
t_{8j} \mapsto a_2\h^2, a_3\h^{-1}, a_5\h^{-1}; \qquad
t_{9j} \mapsto a_5, a_5\h^{-1}; \qquad
t_{10j} \mapsto a_5, a_5\h^{-1}.
\end{multline*}
Here we wrote $t_{ij}$ for $t^\xi_j$ for the $i$'th $\xi$ bundle from the left. In notation we will write $(\ )|_f$ for the restriction map to the fixed point $f\in \Ch(\D)^{\Tt}$. When working in cohomology instead of K theory or elliptic cohomology (or when we work with quadratic forms governing the quasi-periods of elliptic functions) we need to read these specializations additively, e.g. $a_2-\h$ instead of  $a_2\h^{-1}$. 

\subsection{Zero charge 5-branes are unnecessary} \label{Sec:0charge}

Some of the charges of the NS5 or D5 branes in a brane diagram $\mathcal D$ can be zero. However, in such a case the associated bow variety $X(\mathcal D)$ can also be presented with another brane diagram that has no 0-charge 5-branes, cf. \cite[Lemma 2.12]{ji2024bow}. The figure 
\begin{equation*}
\begin{tikzpicture}[baseline=5,scale=.3]
\draw [thick] (1,1) -- (11,1);
\draw [thick,red] (0.5,0) --(1.5,2); \node at ( .7,-1) {\small $\Zb_1$};
   \node at (2.5,1.8) {$\cdots$}; \node at (1.7,.2) {\small $\cdots$};
\draw [thick,red] (2.5,0) --(3.5,2); \node at (2.7,-1) {\small $\Zb_m$};
\draw [thick, blue] (6.5,2) --(7.5,0); \node at (7.2,-1) {\small $\Ab_1$};
   \node at (7.2,-2.2) {$w_1$};
\draw [thick, blue] (8.5,2) --(9.5,0); \node at (9.2,-1) {\small $\Ab_2$};
   \node at (9.2,-2.2) {$w_2$};
\draw [thick, blue] (10.5,2) --(11.5,0); \node at (11.2,-1) {\small $\Ab_3$};
   \node at (11.2,-2.2) {$w_3$};
 \draw [black](6,2) to [out=70,in=-90] (8.5,2.7) to [out=-90,in=100] (11,2);
\node at (8.5,3.2) {$n$};
\draw [black](1,2) to [out=70,in=-90] (2.5,2.7) to [out=-90,in=100] (4,2);
\node at (2.5,3.2) {$m$};
\end{tikzpicture}
\longrightarrow
\begin{tikzpicture}[baseline=5,scale=.3]
\draw [thick] (1,1) -- (14,1);
\draw [thick,red] (0.5,0) --(1.5,2); \node at ( .7,-1) {\small $\Zb_1$};
   \node at (2.5,1.8) {\small $\cdots$}; \node at (1.7,.2) {\small $\cdots$};
\draw [thick,red] (2.5,0) --(3.5,2); \node at (2.7,-1) {\small $\Zb_m$};
\draw [thick, blue] (6.5,2) --(7.5,0); \node at (7.2,-1) {\small $\Ab_1$};
   \node at (7.2,-2.2) {$w_1$};    \node at (7.1,-3.2) {\small $+1$};
\draw [thick, blue] (8.5,2) --(9.5,0); \node at (9.2,-1) {\small $\Ab_2$};
   \node at (9.2,-2.2) {$w_2$};  \node at (9.1,-3.2) {\small $+1$};
\draw [thick, blue] (10.5,2) --(11.5,0); \node at (11.2,-1) {\small $\Ab_3$};
   \node at (11.2,-2.2) { $w_3$};  \node at (11.1,-3.2) {\small $+1$};
\draw [thick,red] (13.5,0) --(14.5,2); \node at ( 13.7,-1) {\small $\Zb'$};
   \node at (13.5,-2.6) { $n$};
\node at (12.5,1.7) {$0$};
\draw [black](6,2) to [out=70,in=-90] (8.5,2.7) to [out=-90,in=100] (11,2);
\node at (8.5,3.2) {$n$};
\draw [black](12,2.2) to [out=70,in=-90] (13.5,2.7) to [out=-90,in=100] (15,2.2);
\node at (13.5,3.3) {ADD};
\end{tikzpicture}
=
\begin{tikzpicture}[baseline=5,scale=.3]
\draw [thick] (1,1) -- (11,1);
\draw [thick,red] (0.5,0) --(1.5,2); \node at ( .7,-1) {\small $\Zb_1$};
   \node at (2.5,1.8) {$\cdots$}; \node at (1.7,.2) {\small $\cdots$};
\draw [thick,red] (2.5,0) --(3.5,2); \node at (2.7,-1) {\small $\Zb_m$};
\draw [thick, blue] (6.5,2) --(7.5,0); \node at (7.2,-1) {\small $\Ab_1$};
   \node at (7.2,-2.2) {$w_1$}; \node at (7.1,-3.2) {\small $+1$};
\draw [thick, blue] (8.5,2) --(9.5,0); \node at (9.2,-1) {\small $\Ab_2$};
   \node at (9.2,-2.2) {$w_2$}; \node at (9.1,-3.2) {\small $+1$};
\draw [thick, blue] (10.5,2) --(11.5,0); \node at (11.2,-1) {\small $\Ab_3$};
   \node at (11.2,-2.2) {$w_3$}; \node at (11.1,-3.2) {\small $+1$};
 \draw [black](6,2) to [out=70,in=-90] (8.5,2.7) to [out=-90,in=100] (11,2);
\node at (8.5,3.2) {$n$};
\draw [thick,red] (4,0) --(5,2); \node at ( 4.4,-.85) {\small $\Zb'$};
   \node at (4.2,-2.6) { $n$};
\draw [black](1,2) to [out=70,in=-90] (3.25,2.7) to [out=-90,in=100] (5.5,2);
\node at (3.25,3.2) {$m+1$};
\end{tikzpicture}
\end{equation*}
illustrates how to get rid of possible 0-charge D5 branes. Namely, first we bring $\mathcal D$ to separated form (for convenience), see the first diagram. Then we add an NS5 brane at the right end of the diagram connected by a D3 brane of multiplicity 0. Its charge is then $n$, the number of D5 branes. This addition does not change the bow variety (by definition), but it adds 1 to all D5 brane charges, indicated by $w_i$ and $w_i+1$ in the figure. The separated HW-representative of this new diagram is on the right. 

The procedure of getting rid of possible 0-charge NS branes, by adding an extra D5 brane, is analogous. 

\section{More on bow varieties}

\noindent In this section we present three new results on bow varieties.

\subsection{Geometric tensor structure}
\label{section Actions of subtori}

Let us now consider more general torus actions. Namely, let $X(\D,r,c)$ be a bow variety with charge vectors $c=(\ch(\Ab))_{\Ab\in \D}$ and $r=( \ch(\Zb))_{\Zb\in \D}$ and consider a rank two torus $A=\Cs_{a'}\times \Cs_{a''}\subset \A$ acting with $a'$ on a subset of the D5 branes in $\D$, and with $a''$ on the remaining D5 branes.

Let $\D'$ (resp. $\D''$) be the brane diagram obtained from $\D$ by erasing those D5 branes that are not acted on by $a'$ (resp. $a''$). Accordingly, the charge vector $c$ induces a charge vector $c'$ (resp. $c''$ ) obtained from $c$ by erasing those entries corresponding to the D5 branes in $\D''$ (resp. $\D'$). By construction, the vector $c$ is recovered from $c'$ and $c''$ by merging the two vectors together, hence we write $c=c'\sqcup c''$. 

\begin{theorem}
\label{proposition tensor product decomposition fixed locus}
    We have 
    \[
    X(\D,r,c)^A=\bigsqcup_{r=r'+r''}X(\D',r',c')\times X(\D'',r'',c'').
    \]
\end{theorem}

\begin{proof}[Proof of Theorem \ref{proposition tensor product decomposition fixed locus}]
    The proof is an adaptation of the analogous statement for quiver varieties given in \cite[Prop. 2.3.1]{maulik2012quantum}, so we only sketch it. Recall that $X(\D,r,c)$ can be seen as the moduli space of representations $p\in \Rep_Q(\w)$ of a certain quiver $Q$ with dimension $\w$ satisfying $\mu(p)=0$, the open conditions $(S1), (S2)$, and GIT stability. 
    
    Fix some $x\in X(\D,r,c)^A$ and let $p=(a,b,A,B,C,D)\in \mu^{-1}(0)^{s}$ be a representation whose class is $x$. Since the action of $G=\prod_{\Xb} {\GL(W_{\Xb})}$ on $ \mu^{-1}(0)^{s}$ is free, we get a map
    \[
    \phi: A\to A\times G
    \]
    which is the identity on the first component and such that $p$ is fixed by $\phi(A)$. The composition $A\to A\times G\to G$ induces decompositions of the vector spaces $W_{\Xb}=W_{\Xb}'\oplus W_{\Xb}''$ in weight subspaces acted on by $a'$ and $a''$, respectively. The decomposition of the remaining vertices, i.e. the lower vertices in the three-way parts, is always trivial because by assumption they are either acted on by $a'$ or $a''$.
    
    Since $p$ is $\phi(A)$-fixed, it decomposes as a direct sum of sub-representations $p=p'\oplus p''$ respecting the weight decompositions. Since the lower vertices in the three-way parts are either acted on by $a'$ or $a''$, the point $p'$ (resp. $p''$) is effectively a representation of the quiver $Q'$ (resp. $Q''$) that can be obtained from $Q$ by erasing all the lower vertices with an action of $a''$(resp. $a'$) in the three-way parts together with all the arrows connected to them. The following is the graphical illustration, with the assumption that $\C_{\Ab}$ carries an action of $a'$: 
    \[
    \begin{tikzcd}[column sep=small, row sep=normal]
    W_{\Ab_-}\arrow[loop ,out=120,in=60,distance=2em, "B_{\Ab_-}"] & &  W_{\Ab_+}\arrow[ll, "A_{\Ab}"]\arrow[dl, "b_{\Ab}"] \arrow[loop ,out=120,in=60,distance=2em, "B_{\Ab_+}"]\\
    & \C_{\Ab}\arrow[ul, "a_{\Ab}"] &
    \end{tikzcd}
    =
    \begin{tikzcd}[column sep=small, row sep=normal]
    W'_{\Ab_-}\arrow[loop ,out=120,in=60,distance=2em, "B'_{\Ab_-}"] & &  W'_{\Ab_+}\arrow[ll, "A'_{\Ab}"]\arrow[dl, "b'_{\Ab}"] \arrow[loop ,out=120,in=60,distance=2em, "B'_{X_+}"]\\
    & \C_{\Ab}\arrow[ul, "a'"] &
    \end{tikzcd}\oplus 
    \begin{tikzcd}[column sep=small, row sep=normal]
    W''_{\Ab_-}\arrow[loop ,out=120,in=60,distance=2em, "B''_{\Ab_-}"] & &  W''_{\Ab_+}\arrow[ll, "A''_{\Ab}"]\arrow[dl, "0"] \arrow[loop ,out=120,in=60,distance=2em, "B''_{\Ab_+}"]\\
    & 0\arrow[ul, "0"] &
    \end{tikzcd}
    \]
    It is easy to check that $p=p'\oplus p''$ satisfies (S1),(S2) and the GIT stability iff $p'$ and $p''$ do. Similarly, $p$ satisfies $\mu(p)=0$ iff $\mu(p')=\mu(p'')=0$. Set $G'=\prod_{\Xb} {\GL(W'_{\Xb})}$ and $G''=\prod_{\Xb} {\GL(W''_{\Xb})}$. Then we have 
    \[
    X(\D,r,c)^A=\bigsqcup_{\w'+\w''=\w} (\Rep_{Q'}(\w')\cap\mu^{-1}(0)^s)/G'\times (\Rep_{Q''}(\w'')\cap \mu^{-1}(0)^s)/G''.
    \]
    It remains to identify the two factors of the tensor product with the bow varieties in the statement. 
    Since any $p''\in \Rep_{Q'}(\w'')\cap\mu^{-1}(0)^s$ satisfies (S1), (S2), and $\mu(p'')=0$, the maps of the form $A''_{\Ab}$ as in the figure above are isomorphisms and $B''_{\Ab_-}=A''_{\Ab}B''_{\Ab_+}(A''_{\Ab})^{-1}$. 
    
    Overall, this means that the quiver $Q''$ is obtained from $Q$ by replacing all the three-way parts with an $\Cs_{a'}$-action with a single edge and a loop, and the representations in $\Rep_{Q''}(\w'')\cap\mu^{-1}(0)^s$ assign isomorphisms to these edges. In other words, all these three-way parts can be replaced with a factor of the form $T^*\GL(W''_{\Ab_-})=T^*\GL(W''_{\Ab_+})$. Taking the symplectic reduction by $\GL(W''_{\Ab_-})=\GL(W''_{\Ab_+})$ is the same as neglecting such factors. This observation provides the identification $\Rep_{Q''}(\w'')\cap\mu^{-1}(0)^s=X(\D,r'',c'')$ for some $(r'',c'')$ determined by $\w$. A similar argument shows that $\Rep_{Q'}(\w')\cap\mu^{-1}(0)^s=X(\D,r',c')$  for some $(r',c')$. The check that $\w'+\w''=\w$ is equivalent to $r'+r''=r$ and $c'\sqcup c''=c$ is left as an exercise.
\end{proof}

By iteration of the previous proposition, one gets an explicit description of the fixed locus for larger tori acting on the D5 branes, and eventually recovers the description of the $\At$-fixed locus discussed in Section~\ref{sec:fixedpoints}.
\begin{example} \rm
\label{Example fixed point-fixed component interaction}
Consider the bow variety $X(\D,r,c)$ and the fixed point $f\in X(\D)^{\At}$ given by the following tie diagram: 
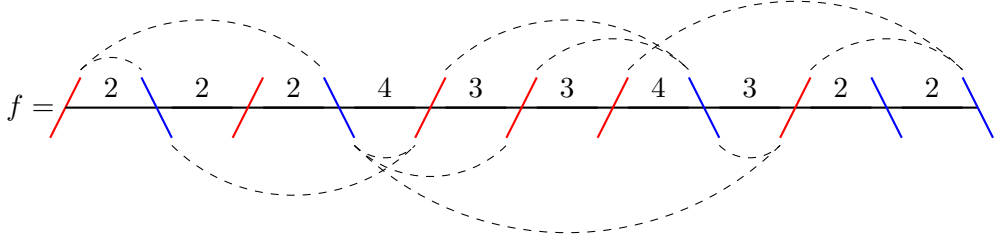
\[
\begin{tikzpicture}[baseline=0,scale=.4]
\begin{scope}[yshift=0cm]
\draw [thick,red] (0.5,0) --(1.5,2); 
\draw[thick] (1,1) node [left] {$f=$}--(2.5,1) node [above] {$2$} -- (31,1);
\draw [thick,blue](4.5,0) --(3.5,2);  
\draw [thick](4.5,1)--(5.5,1) node [above] {$2$} -- (6.5,1);
\draw [thick,red](6.5,0) -- (7.5,2);  
\draw [thick](7.5,1) --(8.5,1) node [above] {$2$} -- (9.5,1); 
\draw[thick,blue] (10.5,0) -- (9.5,2);  
\draw[thick] (10.5,1) --(11.5,1) node [above] {$4$} -- (12.5,1); 
\draw [thick,red](12.5,0) -- (13.5,2);   
\draw [thick](13.5,1) --(14.5,1) node [above] {$3$} -- (15.5,1);
\draw[thick,red] (15.5,0) -- (16.5,2);  
\draw [thick](16.5,1) --(17.5,1) node [above] {$3$} -- (18.5,1);  
\draw [thick,red](18.5,0) -- (19.5,2);  
\draw [thick](19.5,1) --(20.5,1) node [above] {$4$} -- (21.5,1);
\draw [thick,blue](22.5,0) -- (21.5,2);
\draw [thick](22.5,1) --(23.5,1) node [above] {$3$} -- (24.5,1);  
\draw[thick,red] (24.5,0) -- (25.5,2); 
\draw[thick] (25.5,1) --(26.5,1) node [above] {$2$} -- (27.5,1);
\draw [thick,blue](28.5,0) -- (27.5,2);  
\draw [thick](28.5,1) --(29.5,1) node [above] {$2$} -- (30.5,1);   
\draw [thick,blue](31.5,0) -- (30.5,2);   

\draw [dashed, black](4.5,-.25) to [out=-45,in=225] (12.5,-.25);
\draw [dashed, black](10.5,-.25) to [out=-45,in=225] (12.5,-.25);
\draw [dashed, black](10.5,-.25) to [out=-45,in=225] (15.5,-.25);
\draw [dashed, black](10.5,-.25) to [out=-45,in=225] (24.5,-.25);
\draw [dashed, black](22.5,-.25) to [out=-45,in=225] (24.5,-.25);

\draw [dashed, black](1.5,2.25) to [out=45,in=-225] (3.5,2.25);
\draw [dashed, black](1.5,2.25) to [out=45,in=-225] (9.5,2.25);
\draw [dashed, black](13.5,2.25) to [out=45,in=-225] (21.5,2.25);
\draw [dashed, black](16.5,2.25) to [out=45,in=-225] (21.5,2.25);
\draw [dashed, black](19.5,2.25) to [out=45,in=-225] (30.5,2.25);
\draw [dashed, black](25.5,2.25) to [out=45,in=-225] (30.5,2.25);

\end{scope}
\end{tikzpicture}
\]
Let $A:=\Cs_{a'}\times \Cs_{a''} \subset \At$ be the torus acting with $a'$ on the first, third and fourth D5 branes in $\D$ and with $a''$ with the remaining two D5 branes. By Theorem \ref{proposition tensor product decomposition fixed locus}, $f$ gets identified with a pair of fixed points $f'\times f'' \in X(\D',r',c')\times X(\D'',r'',c'')$ for the appropriate charges $r',r'',c'$ and $c''$. 

Following the argument in the proof above, one can check that the tie diagram of $f'$ (resp. $f''$) is obtained from the one of $f$ by removing the second and fifth  (resp. first, third, and fourth) branes together with all the ties attached to them:
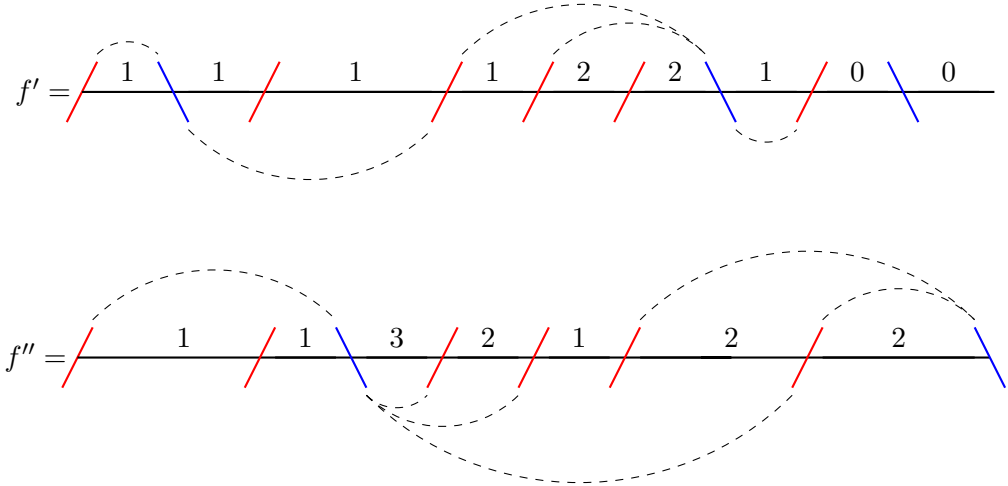
\[
\begin{tikzpicture}[baseline=0,scale=.4]
\begin{scope}[yshift=0cm]
\draw [thick,red] (0.5,0) --(1.5,2); 
\draw[thick] (1,1) node [left] {$f'=$} --(2.5,1) node [above] {$1$} -- (31,1);
\draw [thick,blue](4.5,0) --(3.5,2);  
\draw [thick](4.5,1)--(5.5,1) node [above] {$1$} -- (6.5,1);
\draw [thick,red](6.5,0) -- (7.5,2);  
\draw [thick](7.5,1) --(10,1) node [above] {$1$} -- (12.5,1); 
\draw [thick,red](12.5,0) -- (13.5,2);   
\draw [thick](13.5,1) --(14.5,1) node [above] {$1$} -- (15.5,1);
\draw[thick,red] (15.5,0) -- (16.5,2);  
\draw [thick](16.5,1) --(17.5,1) node [above] {$2$} -- (18.5,1);  
\draw [thick,red](18.5,0) -- (19.5,2);  
\draw [thick](19.5,1) --(20.5,1) node [above] {$2$} -- (21.5,1);
\draw [thick,blue](22.5,0) -- (21.5,2);
\draw [thick](22.5,1) --(23.5,1) node [above] {$1$} -- (24.5,1);  
\draw[thick,red] (24.5,0) -- (25.5,2); 
\draw[thick] (25.5,1) --(26.5,1) node [above] {$0$} -- (27.5,1);
\draw [thick,blue](28.5,0) -- (27.5,2);  
\draw [thick](28.5,1) --(29.5,1) node [above] {$0$} -- (30.5,1);   

\draw [dashed, black](4.5,-.25) to [out=-45,in=225] (12.5,-.25);
\draw [dashed, black](22.5,-.25) to [out=-45,in=225] (24.5,-.25);

\draw [dashed, black](1.5,2.25) to [out=45,in=-225] (3.5,2.25);
\draw [dashed, black](13.5,2.25) to [out=45,in=-225] (21.5,2.25);
\draw [dashed, black](16.5,2.25) to [out=45,in=-225] (21.5,2.25);
\end{scope}
\end{tikzpicture}
\]
\[
\begin{tikzpicture}[baseline=0,scale=.4]
\begin{scope}[yshift=0cm]
\draw [thick,red] (0.5,0) --(1.5,2); 
\draw[thick] (1,1) node [left] {$f''=$} --(4.5,1) node [above] {$1$} -- (31,1);
\draw [thick,red](6.5,0) -- (7.5,2);  
\draw [thick](7.5,1) --(8.5,1) node [above] {$1$} -- (9.5,1); 
\draw[thick,blue] (10.5,0) -- (9.5,2);  
\draw[thick] (10.5,1) --(11.5,1) node [above] {$3$} -- (12.5,1); 
\draw [thick,red](12.5,0) -- (13.5,2);   
\draw [thick](13.5,1) --(14.5,1) node [above] {$2$} -- (15.5,1);
\draw[thick,red] (15.5,0) -- (16.5,2);  
\draw [thick](16.5,1) --(17.5,1) node [above] {$1$} -- (18.5,1);  
\draw [thick,red](18.5,0) -- (19.5,2);  
\draw [thick](19.5,1) --(22.5,1) node [above] {$2$} -- (21.5,1);
\draw[thick,red] (24.5,0) -- (25.5,2); 
\draw[thick] (25.5,1) --(28,1) node [above] {$2$} -- (30.5,1);
\draw [thick,blue](31.5,0) -- (30.5,2);   

\draw [dashed, black](10.5,-.25) to [out=-45,in=225] (12.5,-.25);
\draw [dashed, black](10.5,-.25) to [out=-45,in=225] (15.5,-.25);
\draw [dashed, black](10.5,-.25) to [out=-45,in=225] (24.5,-.25);

\draw [dashed, black](1.5,2.25) to [out=45,in=-225] (9.5,2.25);
\draw [dashed, black](19.5,2.25) to [out=45,in=-225] (30.5,2.25);
\draw [dashed, black](25.5,2.25) to [out=45,in=-225] (30.5,2.25);
\end{scope}

\end{tikzpicture}
\]
It is also instructive to verify the consistency of the charges, as well as the identities $r=r'+r''$ and $c=c'\sqcup c''$.
\end{example}

\begin{remark}
\label{remark rep. thy. interpretation cohomology bow varieties}
This theorem can be seen as the bow variety version of the so-called tensor product structure of the fixed locus for quiver varieties \cite[Section 2.3]{maulik2012quantum}. It plays a key role in the interpretation of the cohomology of bow varieties as a module over the appropriate quantum group. Indeed, let us define 
\[
X(\D, c):=\bigsqcup_{r} X(\D, r, c).
\]
Then we have $X(\D, c')^A= X(\D, c')\times X(\D, c'')$ and hence, passing to cohomology with coefficients in a field, we get
\[
H^*_{\Tt}(X(\D,c)^A)=H_{\Tt'}(X(\D,c'))\otimes_{\Bbbk} H_{\Tt''}(X(\D,c'')),
\]
where $\Bbbk=H_{\Cs_{\h}}(\pt)$.
This description highlights the representation theoretic interpretation of the charge vectors $c$ and $r$. The vector $c$ defines a quantum group representation, namely $H^*_{\Tt}(X(\D,c))$, and $r$ prescribes a specific weight space $H^*_{\Tt}(X(\D,r,c))\subseteq H^*_{\Tt}(X(\D,c))$. The parallel with Nakajima varieties $\mathcal{M}_Q(\vi,\w)$ is thus the following: the D5 charge vector $c$ play similar roles as the framing vector $\w$, and the NS5 charge vector $r$ plays similar role to the dimension vector $\vi$.
\end{remark}

\subsection{Relations among tautological bundles} \label{sec:coincidence}
Consider the separated brane diagram
\begin{equation}\label{sepfig1}
\begin{tikzpicture}[baseline=10,scale=.44]
\draw [thick,red] (0.6,0) --(1.4,2); 
\draw [thick,red](3.6,0) --(4.4,2);  
\draw [thick,red](6.6,0) -- (7.4,2);  
\draw [thick,red](11.6,0) -- (12.4,2);   
\draw[thick,red] (14.6,0) -- (15.4,2);  
\draw [thick,blue](18.4,0) -- (17.6,2);  
\draw [thick,blue](21.4,0) -- (20.6,2);
\draw[thick,blue] (24.4,0) -- (23.5,2); 
\draw [thick,blue](28.5,0) -- (27.5,2);  
\draw [thick,blue](31.5,0) -- (30.5,2);   

\draw [thick] (1,1)--(31,1);

\node at (2.8,2) {$\xi_{1}$};
\node at (5.9,2) {$\xi_{2}$};
\node at (10,2) {$\cdots$};
\node at (14,2) {$\xi_{m-1}$};
\node at (16.5,2) {$\zeta_0$};
\node at (19.4,2) {$\zeta_{1}$};
\node at (22.4,2) {$\zeta_{2}$};
\node at (25.5,2) {$\cdots$};
\node at (29.4,2) {$\zeta_{n-1}$};

\node at (0.8,-.8) {$\Zb_{1}$};
\node at (3.8,-.8) {$\Zb_{2}$};
\node at (6.8,-.8) {$\Zb_{3}$};
\node at (11.8,-.8) {$\Zb_{m-1}$};
\node at (14.8,-.8) {$\Zb_{m}$};

\node at (18.7,-.8) {$\Ab_{1}$};
\node at (21.7,-.8) {$\Ab_{2}$};
\node at (24.7,-.8) {$\Ab_{3}$};
\node at (28.5,-.8) {$\Ab_{n-1}$};
\node at (31.6,-.8) {$\Ab_{n}$};
\end{tikzpicture}.
\end{equation}
\begin{proposition} \label{prop:rightbundles}
The $\zeta_k$ bundles 
are topologically trivial (with an action of $\Tt$). 
\end{proposition}

\begin{proof} For a point in $\Ch(\DD)$ let $v_k=a_{\Ab_{k}}(1)$ be a vector in its fiber of $\zeta_{k-1}$. The stability conditions (S1), (S2) imply that the sections $\left(B_{\Ab_{n}}\right)^i(v_n)$ for $i=0,\ldots,\ch(\Ab_{n})-1$ trivialize $\zeta_{n-1}$. The $A_{\Ab_{n-1}}$-images of these sections, together with the sections $\left(B_{\Ab_{n-1}}\right)^i(v_{n-1})$ for $i=0,\ldots,\ch(\Ab_{n})-1$ trivialize $\zeta_{n-2}$, again because of (S1), (S2), etc. 
\end{proof}

Let us consider a diagram HW-equivalent with the separated one above, and let the portion between $\Zb_{k}$ and $\Zb_{k+1}$ be (that is, $r=\ell(\Zb_{k+1})-\ell(\Zb_k)$)

\begin{equation}\label{etafig}
\begin{tikzpicture}[baseline=10,scale=.44]
\draw [thick,red] (3.6,0) --(4.4,2); 
\draw [thick,red](20.6,0) --(21.4,2);  
\draw [thick,blue](7.4,0) -- (6.6,2); 
\draw [thick,blue](10.4,0) -- (9.6,2);
\draw[thick,blue] (15.4,0) -- (14.5,2); 
\draw [thick,blue](18.5,0) -- (17.5,2); 
\draw [thick] (1,1)--(24,1);
\node at (3.8,-.8) {$\Zb_{k}$};
\node at (20.8,-.8) {$\Zb_{k+1}$};
\node at (12.5,1.5) {$\ldots$};
\node at (-0.5,1) {$\ldots$};
\node at (25.5,1) {$\ldots$};
\node at (5.5,2) {$\eta_{k0}$};
\node at (8.5,2) {$\eta_{k1}$};
\node at (16.2,2) {$\eta_{k,r-1}$};
\node at (19.5,2) {$\eta_{k,r}$};
\end{tikzpicture}.
\end{equation}

\begin{proposition}\label{prop:negligible}
Let us identify the varieties corresponding to diagrams \eqref{sepfig1} and \eqref{etafig} via HW isomorphism. Then the $\eta_{ki}$ bundles of Figure \eqref{etafig} can be written as $\xi_{k}+N_{ki}(a,\h) \in K_{\Tt}(\Ch(\DD))$, where $N_{ki}(a,\h)$ stands for a virtual bundle that is topologically trivial (with an action of $\Tt$). The bundles left of $\Zb_1$ and those right of $\Zb_m$ in Figure \eqref{etafig} are topologically trivial.
\end{proposition}

We used the letter $N$ to indicate that these contributions will be {\em negligible} for us.

\begin{proof}
Assuming the statement for the diagram on the left and applying the HW transition  
\begin{equation}\label{HWinprooffig}
\begin{tikzpicture}[baseline=(current  bounding  box.center), scale=.5]
\draw[thick] (-1,1)--(9.5,1);
\draw[thick,red] (0.6,0)--(1.4,2);
\draw[thick,blue] (6.9,0)--(6.1,2);
\node at (-0.2,2) {$\xi_{k}+N$};
\node at (4,2) {$\xi_{k+1}+N$};
\node at (8.5,2) {$\xi_{k+1}+N$};
\node at (0.6,-.7) {$\Zb_k$};
\node at (7.1,-.7) {$a_i\h^j$};
\draw[ultra thick, ->] (11,1)--(12.5,1);
\draw[thick] (16,1)--(24,1);
\draw[thick,blue] (17.6,2)--(18.4,0);
\draw[thick,red] (21.4,2)--(20.6,0);
\node at (15.8,2) {$\xi_{k}+N$};
\node at (19.5,2) {$\eta$}; 
\node at (23.7,2) {$\xi_{k+1}+N$};
\node at (18.3,-.7) {$a_i\h^{j-1}$};
\node at (21,-.7) {$\Zb_k$};
\end{tikzpicture},
\end{equation}
for $\eta$ we obtain 
$\eta=(\xi_{k}+N)+(\xi_{k+1}+N)-(\xi_{k+1}+N)+a_i\h^j
=
\xi_{k}+N \in K_{\Tt}(\Ch(\DD))$
what we wanted to prove (here we used the same letter $N$ for possibly different virtual bundles of type $N(a,h)$).
\end{proof}

\subsection{Quadratic forms} 
\label{subsection Quadratic forms}

When dealing with equivariant elliptic cohomology we will consider a certain quadratic form (up to an equivalence) associated with the space $\Ch(\DD)$. Part of this quadratic form (``the $\alpha$ part'') comes from a virtual bundle over the space, and another (``dynamical'') part is an explicit formula. In this section, we present these two parts and prove that their sum is HW invariant.

\subsubsection{The $alpha$ part}\label{The alpha part}
Let $\DD$ be a brane diagram. Consider the class $\alpha \in K_{\Tt}(\Ch(\DD))$ of the virtual bundle 
\begin{equation}
\label{class alpha for general bow variety}
\h \left( 
\bigoplus_{\Ab} \Hom( \xi^{\Ab^+}, \xi^{\Ab^-} ) \oplus \Hom( \C_{\Ab}, \xi^{\Ab^-} ) 
\oplus
\bigoplus_{\Zb} \Hom( \xi^{\Zb^-}, \xi^{\Zb^+} )
\ominus 
\bigoplus_{\xi} \Hom( \xi, \xi)
\right)^\vee
\end{equation}
where $\Ab$, $\Zb$, $\xi$ run over the D5, NS5 and D3 branes. Observe that the definition of $\alpha$ changes under HW transition, because the $\xi$ bundles change. Yet, the change is controlled. 

\begin{proposition} \label{prop:alphach} 
Consider the HW transition of \eqref{HWfig} {from left to right}, and denote $\xi'_2=\xi_1+\xi_3+\C_{\Ab}-\xi_2$.  In $K_{\Tt}(\Ch(\DD))$ we have
\[
\alpha^{\after}-\alpha^{\before}=\Hom(\xi_1,\C_{\Ab}) - \h\Hom(\C_{\Ab},\xi'_2).
\]
\end{proposition}

\begin{proof} Straightforward (long) calculation using \eqref{HWfig}. \end{proof}

Let us define an additive map $Q:K_{\Tt}(\Ch(\DD)) \to Q(H^2_{\Tt}(\Ch(\DD))$, where the codomain is the additive group of quadratic forms on $H^2_{\Tt}(\Ch(\DD))$, by sending a line bundle $L$ to $c_1(L)^2$. We say that two such quadratic forms are equivalent (in notation $\equiv$) if they differ by a quadratic form in $Q(H^2_{\Tt}(\pt))$ (or in the $z_i$ variables to be considered later). That is, quadratic forms in $a_i$ and $\h$ variables are equivalent to 0, and in turn, the quadratic forms of the ``negligible'' terms of Section \ref{sec:coincidence} are equivalent to 0. 

\begin{corollary} \label{cor:alphach}
Consider the situation of Proposition \ref{prop:alphach}, and let $t_{1i}$ be the Chern roots of the bundle $\xi_1$. Then 
\[
Q(\alpha^{\after})-Q(\alpha^{\before}) \equiv -2\h \sum_i t_{1i}.
\]
\end{corollary}

\begin{proof}
According to Proposition \ref{prop:alphach}, we have $Q(\alpha^{\after})-Q(\alpha^{\before})=
\sum_i (a-t_{1i})^2 - \sum_j (t'_{2j}-a+\h)^2$. Using Proposition \ref{prop:negligible}, this is equivalent to 
$\sum_i (a-t_{1i})^2 - \sum_i (t_{1i}-a+\h)^2 = -2\h\sum_i(t_{1i}-a) - \sum_i \h^2$, which is, by algebra, equivalent to the stated formula.
\end{proof}

\subsubsection{The dynamical part}\label{The dynamical part}

Let $\DD$ be a brane diagram. Let $\eta_k$ be {\em one} of the bundles in between the NS5 branes $\Zb_{k}$ and $\Zb_{k+1}$. 
Let us define the quadratic form
\begin{equation}
    \label{equadratic form U part}
    \QU=
    2 \sum_{k=1}^{m-1}  \left(\sum_i t^{\eta_k}_i\right) \left(z_{k}-z_{k+1}+(\ell(\Zb_{k+1})-\ch(\Zb_k))\h\right).
\end{equation}
We will only consider $\QU$ up to the equivalence $\equiv$, hence the choice of $\eta_k$ does not matter (cf. Proposition~\ref{prop:negligible}).

\begin{proposition}\label{prop:Qch}
Consider the HW transition of \eqref{HWfig} {\em from left to right}. We have 
\[
\QU^{\after}-\QU^{\before}\equiv 2 \h \sum_i t_{1i}.
\]
\end{proposition}

\begin{proof}
If the NS5 brane involved is $\Zb_k$ then the affected terms of $\QU$ are 
\begin{multline*}
2\sum_i t_{1i} (z_{k-1}-z_k+(\ell(\Zb_k)-\ch(\Zb_{k-1})) \h)+
\\
2\sum_j t_{3j} (z_k-z_{k+1}+(\ell(\Zb_{k+1})-\ch(\Zb_k)) \h).
\end{multline*}
By inspection, we see that 
$\ell(\Zb^{\after}_{k})=\ell(\Zb^{\before}_{k})+1$ and $\ell(\Zb^{\after}_{k+1})=\ell(\Zb^{\before}_{k+1})$, while the charges do not change under HW transition.
Simple algebra finishes the proof. (When $k=1$ or $m$ then one of the two terms above is missing, but the proof easily extends.)
\end{proof}

\noindent The net effect of Corollary \ref{cor:alphach} and Proposition \ref{prop:Qch} is 

\begin{theorem} 
\label{Theorem: quadratic form is invariant under HW}
The equivalence class of the quadratic form $Q(\alpha)+\QU$ associated with a brane diagram is invariant under Hanany-Witten transition. 
\end{theorem}


\section{Elliptic cohomology}
\label{section: Elliptic cohomology}

\subsection{Definition}
In this paper, we consider the $G$-equivariant elliptic cohomology associated with an elliptic curve $E=\Cs/q^{\Z}$, cf. \cite{Ganter_2014, ginzburg1995elliptic} for a comprehensive discussion, and \cite[App. A]{okounkov2020inductiveI} for a review in the context of stable envelopes.
We will be only interested in its $0$-th (covariant) functor
\[
E_G(-):G\text{-spaces}\to \text{Schemes}_{E_G},
\]
where
$E_G:=E_G(\pt)$ is the elliptic cohomology of the one-point space.

We only need the case when $G$ is either a rank $n$ torus $T$ or $\GL(n)$. 
Then $E_T=\text{Cochar}(T)\otimes_\Z E \cong  E^{n}$ 
and $E_{GL(n)}= E^{(n)}$, the symmetrized $r^{th}$ Cartesian power of $E$.

\subsection{Pullback}

For convenience, we will denote the functorial morphism $E_G(f):E_G(X)\to E_G(Y)$ associated with a map $f:X\to Y$ again by $f$. In particular, we have an adjoint pair
\[
\begin{tikzcd}
\text{Qcoh}(E_G(Y))\ar[r,bend left,"f^*",""{name=A, below}] & \text{Qcoh}(E_G(X))\ar[l,bend left,"f_*",""{name=B,above}] \ar[from=A, to=B, symbol=\dashv]
\end{tikzcd}
\]
which should not be confused with pushforward and pullback in elliptic cohomology.
Indeed, for any given $\Sh{F}\in \text{Qcoh}(E_G(Y))$ the pullback in elliptic cohomology is the canonical morphism 
\[
f^{\oast}:\Sh{F}\to f_*f^*\Sh{F} 
\]
coming from adjunction. Notice that it is not a functor but a morphism in the target category of the functor $f_{*}$. Pushing forward along the canonical maps $p_X: X\to \pt\leftarrow Y: p_Y$, we get a map of $\Sh{O}_{E_G(\pt)}$-modules $(p_Y)_*\Sh{F}\to (p_X)_*f^*\Sh{F}$,
which, with slight abuse of notation, we also denote by $f^{\oast}: \Sh{F}\to f^*\Sh{F}$.

\subsection{Change of group}
Elliptic cohomology is also functorial with respect to group homomorphisms. Namely, given a a $G$-space $X$ and a group homomorphism  $\varphi: H\to G$, we get a map $\varphi: E_H(X)\to E_G(X)$ and an adjoint pair 
\[
\begin{tikzcd}
\text{Qcoh}(E_G(X))\ar[r,bend left,"\varphi^*",""{name=A, below}] & \text{Qcoh}(E_H(X))\ar[l,bend left,"\varphi_*",""{name=B,above}] \ar[from=A, to=B, symbol=\dashv]
\end{tikzcd}
\]
As in the previous section, we denote use the symbol $\varphi^{\oast}$ to denote both the canonical morphism $\Sh{F}\to \varphi_*\varphi^*\Sh{F}$ and its pushforward to $E_G(\pt)$.

\subsection{Chern roots and Thom sheaves}

\label{Section: Chern roots and thom sheaves}
In order to define pushforwards, we first need to introduce Chern classes and Thom sheaves.
A $G$-equivariant complex vector bundle $V\to X$ of rank $r$ defines a map 
\[
c(V):E_G(X)\to E_{\GL(r)}(\pt)=E^{(r)},
\]
called elliptic characteristic class of $V$, see \cite[Section 1.8]{ginzburg1995elliptic} and \cite[Section 5]{Ganter_2014}. Its coordinates in the target are called elliptic Chern roots. Associated with $V$ is an invertible sheaf, the Thom sheaf of $V$
\[
\Theta(V):=c(V)^*(D).
\]
Here, $D$ is the divisor $D=\lbrace0\rbrace+E^{(r-1)}\subset E^{(r)}$ formed by those tuples $(t_1,\dots, t_r)\in E^{(r)}$ that contain $0$. Since the divisor $D$ is effective, the line bundle $\Theta(V)$ admits a canonical global section, which we denote by $\vartheta(V)$. This notation reflects the fact that $\vartheta(V)$ is the pullback by $c(V)$ of the canonical section
\[
\prod_{i=1}^r\vartheta(t_i)
\]
of the divisor $D$ on $E^{(r)}$. It is easy to check that the assignment $V\mapsto \Theta(V)$ extends to a functorial group homomorphism
\[
\Theta: K_G(X)\mapsto \Pic{}{E_G(X)}\qquad  [V]\mapsto \Theta(V).
\]

The Thom sheaf allows defining another important class of line bundles, which depends on an extra parameter $z$ living in a one-dimensional torus $\Cs_z$ acting trivially on $X$. Namely, for every $V$ as above,  we define 
\begin{equation*}
    \Sh{U}(V,z):=\Theta\left((V-1^{\oplus \rk(V)})(z-1)\right)\in \Pic{}{E_{G\times \Cs_z}(X)}.
\end{equation*}
It admits a canonical meromorphic section, which is the pullback by $c(V)$ of the canonical section
\[
\prod_{i=1}^{r}\delta(t_i,z)=\prod_{i=1}^r \frac{\vartheta(t_iz)}{\vartheta(t_i)\vartheta(z)}
\]
of the line bundle $\bigotimes_{i=1}^r\Theta((t_i-1)(z-1))$ on $E_{\GL(r)\times \Cs_{\hbar}}(\pt)$.
Note that the latter plays a central role in $\h$-deformed characteristic classes in elliptic cohomology \cite{rimanyi2019hbardeformed}. 
Next, we list some properties of these bundles.
\begin{lemma}
\label{Lemma properties U bundle}
    We have 
    \begin{align*}
        &\Sh{U}(V,1)\cong\Sh{O},
        \\
        &\Sh{U}(V,z_1z_2)\cong\Sh{U}(V,z_1)\otimes \Sh{U}(V,z_2),
        \\
        &\Sh{U}(V_1\oplus V_2,z)\cong \Sh{U}(V_1,z)\otimes \Sh{U}(V_2,z),
        \\
        &\Sh{U}(V_1\otimes V_2,z)\cong \Sh{U}(V_1,z)^{\otimes\rk(V_2)}\otimes \Sh{U}(V_2,z) ^{\otimes\rk(V_1)},
        \\
        &\Theta(\Hom(V_1,V_2)-\h\Hom(V_2,V_1))\cong \Sh{U} \left(V_1,\h^{-\rk(V_2)} \right) \otimes \Sh{U} \left( V_2,\h^{\rk(V_1)}\right)\otimes\Theta(\h)^{-\rk(V_1)\rk(V_2)}.
    \end{align*}
\end{lemma}
For the proof and more details about $\Sh{U}(V,z)$, see \cite[App. A]{okounkov2020inductiveI}. The choice of replacing $z$ with $\h$ in the last equation is in view of the applications.

\subsection{Pushforwards}
Let $f:X\to Y$ be a proper complex oriented morphism of smooth varieties and $\Sh{E}$ be any vector bundle on $E_G(Y)$. The pushforward in elliptic cohomology is a morphism
\begin{equation*}
    \label{def: pullback elliptic cohoomlogy}
    f_{\oast} :f_*(\Theta(-N_f)\otimes f^*\Sh{E})\to \Sh{E},
\end{equation*}
where $N_f=f^*TY-TX\in K_G(X)$ is the virtual normal bundle. Pushing forward to $E_G(\pt)$, we get a morphism of $\Sh{O}_{E_G(\pt)}$-modules $(\Theta(-N_f)\otimes f^*\Sh{E})\to \Sh{E}$, which, in analogy with the pullback discussed in the previous section, we also denote by $f_{\oast}$.
More generally, if $f$ is not proper, then pushforward may only be defined on some subsheaf of the domain. By definition, a section $s$ of some sheaf $\Sh{G}$ on $E_G(X)$ is supported on some closed $G$-invariant subset $W$ if its restriction to $E_G(X\setminus W)$ is zero. Its support $\text{supp}(s)$ is the intersection of all such $W$. Then, the pushforward of any complex oriented map $f$ can be defined as
\[
f_{\oast} :f_*(\Theta(-N_f)_c\otimes f^*\Sh{E})\to \Sh{E}
\]
where $\Theta(-N_f)_c$ is the subsheaf of $\Theta(-N_f)$ of sections $s$ such that $f|_{\text{supp}(s)}$ is proper \cite{okounkov2020inductiveI}.

For a closed embedding $i:X\hookrightarrow Y$, 
consider the composition
\[
    i_*\left(\Theta(-N_i)\otimes i^*\Sh{E}\right)\xrightarrow[]{ i_{\oast}} \Sh{E}\xrightarrow[]{i^{\oast}} i_*i^*\Sh{E}
\]
Given a global section section $\gamma\in \Gamma\left(\Theta(-N_i)\otimes i^*\Sh{E}\right)$, we have 
\begin{equation}
    \label{pushforward formula loalization}
i^{\oast}i_{\oast}(\gamma)= \vartheta(N_i)\cdot \gamma,
\end{equation}
cf. \cite[Proposition 2.9.1]{ginzburg1995elliptic}. In other words, the composition $i^{\oast}\circ i_{\oast}$ is equal to multiplication by the elliptic Euler class of the normal bundle $N_i$.

The pushforward in elliptic cohomology is functorial, in the sense that given two proper complex oriented morphisms of smooth varieties $f:X\to Y$ and $g:Y\to Z$ and a line bundle $\Sh{E}$ on $E_G(Z)$, the following diagram commutes
\begin{equation}
    \label{functoriality pushforward elliptic cohomology}
    \begin{tikzcd}
   (g\circ f)_*(\Theta(-N_{g\circ f})\otimes (g\circ f)^*\Sh{E})\arrow[d, equal]\arrow[r, "{(g\circ f)_{\oast}}"] & \Sh{E} \\
   g_*f_*(\Theta(-N_f)\otimes f^*(\Theta(-N_g)\otimes g^*\Sh{E}))\arrow[r, "g_*(f_{\oast})"] & g_*(\Theta(-N_g)\otimes g^*\Sh{E}).\arrow[u, "g_{\oast}"]
\end{tikzcd}
\end{equation}
The identity arrow in the diagram is induced by the isomorphism
\[
\Theta(-N_{g\circ f})=\Theta(-N_f-f^*N_{g})\cong \Theta(-N_f)\otimes f^*\Theta(-N_{g}).
\]
Throughout the paper, we will commonly pushforward this diagram to $E_G(\pt)$. Accordingly, with a slight abuse of notation, we will write $(g\circ f)_{\oast}=g_{\oast}f_{\oast}$.

The pushforward is defined like in ordinary equivariant cohomology, namely as the composition of the Thom isomorphism with the Pontryagin-Thom collapse map. From the naturality of the Thom isomorphism, it follows that if two proper complex-oriented morphisms $f$ and $F$ fit in a Cartesian diagram
\[
\begin{tikzcd}
    X_1\arrow[d, "F"]\arrow[r, "g"] & X_2\arrow[d, "f"]
    \\
    Y_1 \arrow[r, "G"] & Y_2 
\end{tikzcd}
\]
such that $\dim( X_1)-\dim(Y_1)=\dim( X_2)-\dim(Y_2)$, then $G^{\oast}f_{\oast}=G_*(F_{\oast})f_*(g^{\oast})$ as morphisms of $\Sh{O}_{E_G(Y_2)}$-modules. The reader is invited to compare the line bundles entering these maps' definitions. To ease the notation, we will simply write $G^{\oast}f_{\oast}=F_{\oast}g^{\oast}$. 

We conclude the section by stating the powerful 
localization theorem. 
To ease its formulation, we omit the sheaves entering in the definition of the pushforwards. In addition, for any given elliptic cohomology class $\gamma$ we denote by $\gamma|_F$ the restriction $(i_F)^{\oast}\gamma$, where $F\subset X^G$ and $i_F: F\hookrightarrow X$ is the inclusion. 
\begin{proposition}
\label{localization formula pushforward}
Let $p:X\to Y$ be a proper complex-oriented morphism of smooth $T$-varieties with finite fixed locus. If the class $\alpha$ is in the image of the pushforward $i_{\oast}$ associated with $i: X^T\hookrightarrow X$, then for every $g\in Y^T$ 
\[
(p_{\oast}(\alpha))|_{g}=\vartheta(T_gY)\sum_{f\in p^{-1}(g)^{T}}\frac{\alpha|_{f}}{\vartheta(T_fX)}
\]
holds in the localized theory (the ``localized theory'' is obtained by inverting the equivariant variables of $T$, i.e. by tensoring with the sheaf of meromorphic functions over $E_T(\pt)$).
\end{proposition}

\begin{proof}
    It suffices to assume that $\alpha=i_{\oast}(\beta)$ with $\beta$ supported on some $f\in X^T$. Then the proposition readily follows from the functoriality of the pushforward and \eqref{pushforward formula loalization}.
\end{proof}

\subsection{Action by correspondences}

\label{subsection action by correspondences}

For any line bundle $\Sh{L}$ on $E_G(X)$, its dual is defined as $\Sh{L}^\triangledown:= \Sh{L}^{-1}\otimes \Theta(TX)$ \cite{okounkov2020inductiveI}. Notice that if $p:X\to \pt $ is proper, then there is a well-defined pairing
\[
\Sh{L}^\triangledown\otimes \Sh{L}=\Theta(TX)\xrightarrow[]{p_{\oast}} \Sh{O}_{E_G(\pt)}.
\]
Let $\Sh{L}_i$ be a line bundle on $E_G(X_i)$ for $i=1,2$ and consider the maps 
\[
\begin{tikzcd}
    X_1 & X_1\times X_2 \arrow[l, swap, "p_1"]\arrow[r, "p_2"] & X_2.
\end{tikzcd}
\]
For us, a correspondence $\alpha$ in elliptic cohomology is a global section of 
\[
\Sh{L}_1^\triangledown\boxtimes \Sh{L}_2\in \Pic{}{E_G(X_1\times X_2)}
\]
that is proper over $X_2$, i.e. such that $p_2|_{\text{supp}(\alpha)}$ is proper. The upshot of the properness assumption is that we can push forward along $p_2$ to get an operator
\[
\alpha: \Sh{L}_1\to \Sh{L}_2 \qquad s\mapsto (p_2)_{\oast} (\alpha \cdot p_1^{\oast}(s)).
\]

\subsection{Convolution of correspondences}

Let $X_i$ for $i=1,2,3$ be three $G$-spaces and $\Sh{L}_i$ be line bundles on $E_G(X_i)$. Consider the diagram 
\[
\begin{tikzcd}
    X_1\times X_2\times X_2\times X_3 &  X_1\times X_2\times X_3\arrow[l, swap, "\Delta"] \arrow[r, "p_{13}"] & X_1\times X_3. 
\end{tikzcd}
\]
Let $\alpha_{12}\in H^0(\Sh{L}_1^\triangledown\boxtimes \Sh{L}_2)$ and $\alpha_{23}\in H^0(\Sh{L}_2^\triangledown\boxtimes \Sh{L}_3)$. 
\begin{lemma}
\label{lemma support convolution}
Assume that $\text{supp}(\alpha_{23})$ is proper over $X_3$. Then the convolution
\begin{equation}
    \label{equation definition convolution}
    \alpha_{23}\circ \alpha_{12} :=(p_{13})_{\oast}\left( \Delta^{\oast}(\alpha_{12}\boxtimes \alpha_{23})\right)
\end{equation}
is well-defined and is supported over
\begin{equation}
\label{support convolution}
    \Set{(x_1,x_3)\in X_1\times X_3}{\exists \,  x_2\in X_2 \text{ s.t. } (x_1,x_2)\in \text{supp}(\alpha_{12}) , \,   (x_2,x_3)\in\text{supp}(\alpha_{23}) }.
\end{equation}
\end{lemma}
\begin{proof}
The class $\Delta^{\oast}(\alpha_{12}\otimes \alpha_{23})$ is supported on
\[
\Delta^{-1}(\text{supp}(\alpha_{12})\times \text{supp}(\alpha_{23}))\subseteq X_1\times \text{supp}(\alpha_{23}).
\]
Since the map $\text{supp}(\alpha_{23})\to X_3$ is proper by assumption, the composition
\[
\Delta^{-1}(\text{supp}(\alpha_{12})\times \text{supp}(\alpha_{23}))\hookrightarrow X_1\times X_2\times X_3\to X_1\times X_3
\]
is also proper. Thus, the pushforward in the definition of $\alpha_{23}\circ \alpha_{12}$ is well defined. Let now $V_{13}$ be the set \eqref{support convolution} and notice that 
\[
V_{13}=p_{13} \left(\Delta^{-1}(\text{supp}(\alpha_{12})\times \text{supp}(\alpha_{23}))\right).
\]
We need to show that $\alpha_{12}\circ \alpha_{23}|_{X_1\times X_3\setminus V_{13}}=0$.
Consider the Cartesian diagram  
\[
\begin{tikzcd}
   p_{13}^{-1}(X_1\times X_3\setminus V_{13})\arrow[d, "p_{13}"]\arrow[r, hookrightarrow] & X_1\times X_2 \times X_3\arrow[d, "p_{13}"]
   \\
    X_1\times X_3\setminus V_{13}\arrow[r, hookrightarrow] & X_1\times X_3.
\end{tikzcd}
\]
By compatibility of pushforward and pullback in Cartesian squares, we have
\[
\alpha_{23}\circ \alpha_{12}\Big|_{X_1\times X_3\setminus V_{13}}=(p_{13})_{\oast} \left(\Delta^{\oast}(\alpha_{12}\boxtimes \alpha_{23})\Big|_{  p_{13}^{-1}(X_1\times X_3\setminus V_{13})}\right)=0.
\]
Here, the last step follows from the observation that
\begin{multline*}
p_{13}^{-1}(X_1\times X_3\setminus V_{13})=X_1\times X_2\times X_3\setminus p_{13}^{-1}(V_{13})
\\
\subseteq X_1\times X_2\times X_3\setminus \Delta^{-1}(\text{supp}(\alpha_{12})\times \text{supp}(\alpha_{23}))),
\end{multline*}
hence $\Delta^{\oast}(\alpha_{12}\boxtimes \alpha_{23})|_{  p_{13}^{-1}(X_1\times X_3\setminus V_{13})}=0$.
\end{proof}

\begin{remark}
A standard argument shows that if $\alpha_{12}\circ \alpha_{23}$ is proper over $X_1$, and hence its associated operator is well defined, then the diagram
\begin{tikzcd}
    \Sh{L}_1\arrow[r, "\alpha_{12}"]\arrow[rr, bend right, swap, "\alpha_{23}\circ \alpha_{12}"] & \Sh{L}_{2}\arrow[r, "\alpha_{23}"] & \Sh{L}_3
\end{tikzcd}
commutes. 
\end{remark}

\section{Elliptic stable envelopes}

In this section, we recall the general definition of elliptic stable envelopes. For more details, see \cite{aganagic2016elliptic, okounkov2020inductiveI}. 

\subsection{Attracting loci}
\label{subsection attracting loci}
 Let $X$ be a smooth quasi-projective variety equipped with the action of a torus $T$. Let $A\subset T$ be a subtorus. The $A$-weights in the normal bundle of $X^A$ in $X$ are finite and produce a wall arrangement $W$  in
$ 
\text{Lie}_\R(A)=\text{Cochar}(A)\otimes_\Z \R
$,
the real Lie algebra of $A$.
A chamber $\mathfrak{C}$ is a connected component in $\text{Lie}_\R(A)\setminus W$. A choice of chamber determines a collection of affine bundles
\begin{equation}
    \label{partial order definition}
    \Att{C}(F)=\Set{x\in X}{\lim_{z\to 0} \sigma (z)\cdot x\in F}\qquad \text{for any }\sigma \in \mathfrak{C}
\end{equation}
attached to fixed components in $X^A$. These are called attracting sets. Associated with them is a partial ordering (``generalized Bruhat order'') on the set of connected components, defined as the transitive closure of the relation
\[
\overline{\Att{C}(F)}\cap F'\neq 0 \implies F\geq F'.
\]

Consider the triple of $T$-invariant subspaces
\begin{multline}
    \label{equation spaces in the definition of stable envelopes}
    \Att{C}^{\leq}(F):= \bigcup_{F'\leq F} F\times \Att{C}(F'),
    \qquad X^{\geq F}=X\setminus \bigcup_{F'< F} \Att{C}(F'), 
    \\
    X^{> F} = X\setminus \bigcup_{F'\leq F} \Att{C}(F').
\end{multline}
Notice that $X^{> F}\subseteq X^{\geq F}$ and the inclusion $j: F\times \Att{C}(F)\hookrightarrow F\times X^{\geq F}$ is well defined and closed. 

The space $\Att{C}^{\leq}(F)$ is a singular closed subspace of $F\times X$. It contains the so-called full attracting set $\Attfull{C}(F)$, which is also $T$-invariant and is defined as the set of pairs $(f,x)\in F\times X$ connected via a chain of closures of attracting orbits. If all the fixed components are singletons, then $\Attfull{C}(F)= \Att{C}^{\leq}(F)$.

\subsection{Stable envelopes: basics} \label{section Stable envelopes: basics}


The $A$-equivariant elliptic stable envelopes \cite{aganagic2016elliptic, okounkov2020inductiveI} of $X$ are sections of certain line bundles depending on the choice of a chamber $\mathfrak{C}$ and of an attractive line bundle $\Sh{L}_X$. The latter is a line bundle on the scheme $E_T(X)$ satisfying 
\begin{equation}
    \label{defining equation attractive line bundle}
    \deg_A(i_F^*\Sh{L}_X)=\deg_A(\Theta(N^-_{F/X})),
\end{equation}
for all connected components $F\subseteq X^A$. Here $i_F: F\hookrightarrow X$ is the inclusion. For a general definition of degree in this context, we refer to \cite[App. B.3]{okounkov2020inductiveI}. Later, we will give a characterization for bow varieties.

A constructive way to produce an attractive line bundle is to first find a polarization, i.e. a class $T^{1/2}X\in K_T(X)$ that satisfies
\[
TX=T^{1/2}X+ (T^{1/2}X)^\vee.
\]
in $K_A(X)$.
\begin{proposition}[{\cite[Prop. 2.4, 2.6]{okounkov2020inductiveI}}]
\label{proposition polarization implies existence attractive line bunlde}
Let $T^{1/2}X$ be a polarization and let $\Sh{U}\in \Pic{}{E_T(X)}$ be any line bundle such that $\deg_A(\Sh{U})=0$. Then the line bundle 
$
\Sh{L}_X:=\Theta(T^{1/2}X)\otimes \Sh{U}
$
is attractive.
\end{proposition}
Let 
$ 
\Sh{L}_{A,F}=i_F^*\Sh{L}_X\otimes \Theta(-N_{F/X}^-).
$
The elliptic stable envelope of $F$ will be a section of the sheaf
\begin{equation}
\label{sheaf definition of stable envelopes}
    (\Sh{L}_{A,F})^{\triangledown}\boxtimes \Sh{L}_X(\infty\Delta)
\end{equation}
on $E_T(F\times X)$. The notation $\infty\Delta$ stands for a singular behavior of the stable envelope on the resonant locus $\Delta\subset E_T(\pt)$, see \cite[
Section 2.3]{okounkov2020inductiveI}. 
Explicitly, given a line bundle $\Sh{G}$ on $E_T(F\times X)$, the sheaf $\Sh{G}$ is defined as follows. Given a pair of connected components $F, F'\subseteq X^A$, consider the morphisms
\[
E_T(F\times F')\xrightarrow{\phi} E_{T/A}(F\times F')\xrightarrow{p} E_{T/A}.
\]
Given an attractive line bundle $\Sh{L}_X$, we define its resonant locus $\Delta\subset E_{T/A}$ as the union of 
\[
p\left(\text{supp}\left(R^\bullet \phi_*\left(\Sh{L}_{A,F'}\boxtimes (\Sh{L}_{A,F})^{-1}\right)\right)\right)
\]
for all fixed components $F'$ such that $F'< F$. Then, for any sheaf $\Sh{G}$ on $E_T(F\times X)$, we set
\[
\Sh{G}(\infty\Delta):=(i_{\Delta})_*(i_{\Delta})^*\Sh{G},
\]
where $i_{\Delta}$ fits in the Cartesian diagram 
\[
\begin{tikzcd}
    q^{-1}(E_{T/A}\setminus \Delta)\arrow[d]\arrow[r, hookrightarrow, "i_{\Delta}"]&  E_T(F\times X)\arrow[d, "q"]\\
    E_{T/A}\setminus \Delta \arrow[r, hookrightarrow]& E_{T/A}
\end{tikzcd}
\]
In conclusion, one can informally say that sections of $\Sh{G}(\infty\Delta)$ are sections of $\Sh{G}$ wish singularities on $\Delta$. The resonant locus is said to be  non-degenerate if its complement in $E_{T/A}$ is open and dense. In the case of bow varieties, the resonant locus will indeed be non-degenerate, and we will see how to bound it.

\begin{remark} 
\label{remark about twists of the attractive line bundle}
    We say that two line bundles in $\Pic{}{E_T(X)}$ are equivalent (in notation~$\equiv$) if they are isomorphic up to a twist by a line bundle $\Sh{G}$ pulled back from $E_T(\pt)$. This is the geometric counterpart of the equivalence relation of Section~\ref{The alpha part}.
    Because of the definition of $(\Sh{L}_{A,F})^{\triangledown}$, the isomorphism class of $\Sh{L}^{\triangledown}_{A,F}\boxtimes \Sh{L}_X$ only depends on the equivalence class of $\Sh{L}_X$.
    
    From now on, we say that $\Sh{L}_X$ is \emph{good} if it differs from an attractive line bundle by a twist of some $\Sh{G}$ as above. A natural situation where good line bundles show up is when $T^{1/2}X=\alpha+\beta$ for some $\beta\in K_T(\pt)$. Then Proposition \ref{proposition polarization implies existence attractive line bunlde} implies that
    \[
    \Sh{L}_X=\Theta(\alpha+\beta)\otimes \Sh{U}=\Theta(\alpha)\otimes \Theta(\beta)\otimes \Sh{U}
    \]
    is attractive, but it might be more convenient to use the good line bundle
    $\Sh{L}_X=\Theta(\alpha)\otimes \Sh{U}$. 
    
\end{remark}

\bigskip

For any good line bundle $\Sh{L}_X$, the restriction $\left((\Sh{L}_{A,F})^{\triangledown}\boxtimes \Sh{L}_X\right)|_{F\times \Att{C}(F)}$ is isomorphic to $\Theta(N_j)$, where $N_j$ is the normal bundle to the closed immersion $j: F\times \Att{C}(F)\hookrightarrow F\times X^{\geq F}$. Consequently, there is a well defined pushforward map 
\[
j_{\oast}: \Sh{O}_{E_T(F\times \Att{C}(F))} \to \left((\Sh{L}_{A,F})^{\triangledown}\boxtimes \Sh{L}_X\right)|_{F\times X^{\geq F}}.
\]
The image of the constant $1$ section is, by definition, the class $\left[\Att{C}(F) \right]$.

\begin{definition}
\label{Definition stable envelopes}
Let $\Sh{L}_{X}$ be a good line bundle for a given choice of chamber $\mathfrak{C}$.
The elliptic stable envelope $\Stab{C}{F}$ for $\Sh{L}_X$ at the fixed component $F\subset X^A$ is a section of $(\Sh{L}_{A,F})^{\triangledown}\boxtimes \Sh{L}_X(\infty \Delta)$ which is supported on $\Attfull{C}(F)$ and restricts to $\left[ \Att{C}(F)\right]$ on $F\times X^{\geq F}$.
\end{definition}
\begin{theorem}[Main result of \cite{okounkov2020inductiveI}]
\label{Theorem existence+uniqueness stable envelopes}
    For a fixed $\Sh{L}_{X}$, the stable envelopes exist and are unique. Moreover, the support condition on $\Attfull{C}(F)$ is uniquely determined by a weaker condition, namely support on $\Att{C}^\leq(F)$. 
\end{theorem}


\noindent We conclude the section with two important remarks.
\begin{remark}
    \label{Remark axiomatic stab}
    From Definition \ref{Definition stable envelopes} and Theorem \ref{Theorem existence+uniqueness stable envelopes}, we see that the stable envelopes are uniquely determined by the following two axioms:
\begin{enumerate}
    \item The support axiom: the restriction of $\Stab{C}{F}$ to $F\times X^{>F}=F\times X\setminus \Att{C}^\leq(F)$ is zero.
    \item The diagonal axiom: the restriction of $\Stab{C}{F}$ on $F\times X^{\geq F}$ is equal to $\left[ \Att{C}(F)\right]$.
\end{enumerate}
These axioms are often exploited to prove statements about the stable envelopes. 

Also notice that if we further restrict $\Stab{C}{F}$ to $F\times F\subset F\times X^{\geq F}$, we get $[\Att{C}(F)]|_{F\times F}=\vartheta(N_{F/X}^-)[\Delta]$, where $[\Delta]$ is the fundamental class of the diagonal.
\end{remark}

\begin{remark}
    Assuming that the composition $\Attfull{C}(F)\hookrightarrow F\times X \to X$ is proper, the stable envelope can be thought of as a map rather than a section. Indeed, the properness assumption implies that $\Stab{C}{F}$ defines a map\footnote{The invariance of the section $\Stab{C}{F}$ under twists of the attractive line bundle $\Sh{L}_X$ by some line bundle $\Sh{G}$ as in Remark~\ref{remark about twists of the attractive line bundle} translates in the freedom of twisting the associated map by the same $\Sh{G}$.} 
\[
\Stab{C}{F}:\Sh{L}_{A,F}\to \Sh{L}_X.
\]
via convolution, cf. Section~\ref{subsection action by correspondences}.
\end{remark}


\subsection{Extended elliptic cohomology of bow varieties}

Let $X$ be a bow variety and consider the action of a torus $T=A\times \Cs_{\h}\subseteq \Tt$ for some $A\subseteq \At$.

As discussed in the previous section, the existence of stable envelopes is implied
by the existence of a good line bundle $\Sh{L}_X$. Even if such a line bundle exists,
the stable envelopes are in general sections of $\Sh{L}^{\triangledown}_{A,F}\boxtimes \Sh{L}_X$ with an arbitrary singular behavior on the resonant locus $\infty\Delta\subset E_{\Tt}(X)$, which might be exceedingly big. Consequently, it is important to bound $\infty\Delta$ to the complement of a dense open set. 
Just like for Nakajima varieties, we will achieve this for bow varieties at the cost of enlarging the coefficient space $E_{\Tt}(\pt)$ of the cohomology theory by pulling back all bundles and sections to 
\begin{equation}
    \label{extended base}
    \Base{\Tt}{\Tt^!}:= E_{\At\times \Cs_{\h}\times \At^!}(\pt)=E_{\Tt}(\pt)\times E_{\At^!}(\pt).
\end{equation}

Here, $\At^!=\C^{m}$ is the torus of K\"ahler (or dynamical) variables $(z_1,\dots z_m)$ attached to the $m$ NS5 branes of the bow variety $X$. Thus, for a bow variety with $n$ D5 branes and $m$ NS5 branes, we have $\Base{\Tt}{\Tt^!}\cong E_\hbar\times E^n\times E^m$, and the equivariant and K\"ahler parameters are naturally attached to the appropriate 5-branes (we chose their numbering from left to right):
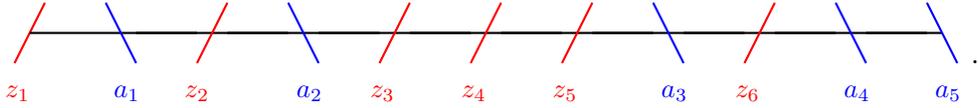
\begin{equation*}
\begin{tikzpicture}[baseline=0,scale=.4]
\draw [thick,red] (0.5,0) --(1.5,2); 
\draw[thick] (1,1)--(2.5,1) -- (31,1);
\draw [thick,blue](4.5,0) --(3.5,2);  
\draw [thick](4.5,1)--(5.5,1) -- (6.5,1);
\draw [thick,red](6.5,0) -- (7.5,2);  
\draw [thick](7.5,1) --(8.5,1) -- (9.5,1); 
\draw[thick,blue] (10.5,0) -- (9.5,2);  
\draw[thick] (10.5,1) --(11.5,1) -- (12.5,1); 
\draw [thick,red](12.5,0) -- (13.5,2);   
\draw [thick](13.5,1) --(14.5,1) -- (15.5,1);
\draw[thick,red] (15.5,0) -- (16.5,2);  
\draw [thick](16.5,1) --(17.5,1) -- (18.5,1);  
\draw [thick,red](18.5,0) -- (19.5,2);  
\draw [thick](19.5,1) --(20.5,1) -- (21.5,1);
\draw [thick,blue](22.5,0) -- (21.5,2);
\draw [thick](22.5,1) --(23.5,1) -- (24.5,1);  
\draw[thick,red] (24.5,0) -- (25.5,2); 
\draw[thick] (25.5,1) --(26.5,1) -- (27.5,1);
\draw [thick,blue](28.5,0) -- (27.5,2);  
\draw [thick](28.5,1) --(29.5,1) -- (30.5,1);   
\draw [thick,blue](31.5,0) -- (30.5,2);   


\node [blue] at (4.2,-1) {\small $a_1$};
\node [blue] at (10.2,-1) {\small $a_2$};
\node [blue] at (22.2,-1) {\small $a_3$};
\node [blue] at (28.2,-1) {\small $a_4$};
\node [blue] at (31.2,-1) {\small $a_5$};

\node [red] at (.6,-1) {\small $z_1$};
\node [red] at (6.5,-1) {\small $z_2$};
\node [red] at (12.6,-1) {\small $z_3$};
\node [red] at (15.6,-1) {\small $z_4$};
\node [red] at (18.6,-1) {\small $z_5$};
\node [red] at (24.6,-1) {\small $z_6$};

\end{tikzpicture}.
\end{equation*}
Like for Nakajima varieties \cite{aganagic2016elliptic, okounkov2020inductiveI}, the elliptic stable envelopes of bow varieties defined on the extended space \eqref{extended base} will only have poles jointly in the K\"ahler variables $z$ and $\hbar$.

\subsection{Chamber structure for a bow variety}
\label{subsection: hamber structure}
Let $X$ be a bow variety with $n\geq 2$ D5 branes and consider the action of the torus $A=\Cs_{a'}\times \Cs_{a''}\subset \At$ introduced in Section~\ref{section Actions of subtori}. Inspecting the proof of Proposition \ref{proposition tensor product decomposition fixed locus}, it is easy to check that the normal weights $\chi\in \Char(A)$ are of the form $\chi=(a'/a'')^{\pm}$. Consequently, the action of the torus $A$ admits two chambers:
\[
\mathfrak{C}_+=\lbrace a' <a''\rbrace \qquad \mathfrak{C}_-=\lbrace a'' <a'\rbrace.
\]
Similarly, the tangent weights for the $\At$ action on $X$ are of the form $a_i/a_j$, with $0 \leq i,j \leq n$ \cite[Section 4]{rimanyi2020bow}. 
Hence, a permutation of the $a_i$ variables determines a chamber. By slight abuse of notation we will write 
$\lbrace a_{\sigma(1)}< \dots <a_{\sigma(n)}\rbrace$ for the chamber it determines. The chamber determined by the choice
$\lbrace a_{1}< \dots <a_{n}\rbrace$ will be called the {\em standard chamber} for the action of $\At$. 

\subsection{Stable envelopes of bow varieties: definition}
\label{section: Stable envelopes of bow varieties: definition}

Let $X$ be an arbitrary bow variety with $m$ NS5 branes, ordered from left to right, and $\Tt=\At\times \Cs_{\h}$ be the usual torus acting on it. Let also $\eta_k$ be one (any) of the bundles in between the NS5 branes $\Zb_k$ and $\Zb_{k+1}$, cf. \eqref{etafig}.

Recall the line bundles introduced in Section~\ref{Section: Chern roots and thom sheaves}. We define
\begin{equation}
    \label{good line bundle bow variety}
    \Sh{L}_X:=\Theta(\alpha)\otimes \Sh{U} \in \Pic{}{E_{\Tt\times \At^!}(X)},
\end{equation}
where $\alpha\in K_{\Tt}(X)$ is the class \eqref{class alpha for general bow variety} and 
\[
\Sh{U}:=\bigotimes_{k=1}^{m-1} \Sh{U}\left(\eta_k,\frac{z_k}{z_{k+1}}\h^{\ell(\Zb_{k+1})-\ch(\Zb_k)}\right),
\]
\begin{definition}
Let $A\subseteq \At$ and $F$ a connected component of $X^A$. The stable envelope $\Stab{C}{F}$ is the unique section of $\Sh{L}_{A,F}^{\triangledown}\boxtimes \Sh{L}_X{(\infty\Delta)}$ 
that satisfies the axioms of Remark \ref{Remark axiomatic stab}.
\end{definition}

In the next proposition we show that this definition is independent of the choice of the $\eta_k$ bundle in the definition of $\Sh{U}$, and the following theorem claims the existence, uniqueness, and holomorphicity property of $\Stab{C}{F}$. 

Notice that $\Sh{L}_X$ is tautological, in the sense that it is pulled back by the equivariant Chern class morphism 
\[
c: E_{\Tt\times \At^!}(X)\to \Base{\Tt}{\Tt^!} \times \prod_{k=1}^m E^{(\rk(\eta_k))},
\]
cf. Section~\ref{Section: Chern roots and thom sheaves}. Recall that complex line bundles $\Sh{L}(Q,v)$ on an abelian variety of the form $\Sh{E}=\C^n/\Gamma$ are classified in terms of a quadratic form $Q\in \text{Mat}_{n\times n}(\Z)$ and a point $v$ in the dual variety $\Sh{E}^\vee=\Pic{0}{\Sh{E}}$, see \cite[Section 5.1]{Felder2018}. In this language, we have
\begin{equation}
    \label{good line bundle in terms of quadratic forms}
    \Sh{L}_X\cong c^* \Sh{L}(Q(\alpha)+QU,0),
\end{equation}
where $Q(\alpha)+QU$ is the quadratic form introduced in Section~\ref{subsection Quadratic forms}. Hence, checking whether a particular tautological formula is the section of the required bundle $\Sh{L}_{A,F}^{\triangledown}\boxtimes \Sh{L}_X$ reduces to a quadratic form calculus, see examples in Section \ref{sec:StabForPn}, cf. \cite{RRTBshuffle}.

Recall the equivalence relation $\equiv$ and the notion of good line bundle introduced in Remark~\ref{remark about twists of the attractive line bundle}.

\begin{proposition}\ 
\label{Proposition goodness chosen line bundle for bow variety}
\begin{enumerate}
    \item The equivalence class of $\Sh{L}_X$ is independent of the choices of $\eta_k$.
    \item The Hanany-Witten isomorphism $X_1\cong X_2$ induces an equivalence $\Sh{L}_{X_1}\equiv \Sh{L}_{X_2}$.
    \item The line bundle \eqref{good line bundle bow variety} is good for the action of any torus $A\subset \At$.
\end{enumerate}
\end{proposition}
\begin{proof}

The first two claims follow from \eqref{good line bundle in terms of quadratic forms}, Proposition \ref{prop:negligible}, and Theorem \ref{Theorem: quadratic form is invariant under HW}. We now prove the third one. Since any bow variety is isomorphic to a separated one via the Hanany-Witten transition, point (2) implies that we can assume that $X$ is separated.

As shown in \cite[Section 4.4.2]{Shou}, for $X$ separated there exists a class $\beta\in K_{\Cs_{\h}}(\pt)$ such that $\alpha+\beta$ satisfies 
\begin{equation}
   \label{equation separated alpha is a polarization}
    TX=\alpha+\beta +\h (\alpha+\beta)^\vee
\end{equation}
in $K_{\Tt}(X)$, and hence is a polarization for any $A\subset \At$. Hence, as observed in Remark \ref{remark about twists of the attractive line bundle}, the line bundle $\Theta(\alpha)$ is good. Moreover, it is easy to see that $\deg_A(i^*_F\Sh{U})=0$ for any fixed component $F\in X^A$. Indeed, the $A$-degree is equal to the part of the quadratic form associated with $i^*_F\Sh{U}$ with degree two in the equivariant parameters of $A$. Since these parameters can only appear by restricting the Chern roots $t_{i}^{\eta_k}$, which have degree one in \eqref{equadratic form U part}, the degree two part of $i^*_F\Sh{U}$ in the $A$-parameters must be zero. Hence, $\Sh{L}_X=\Theta(\alpha)\otimes \Sh{U} $
is also good. 
\end{proof}

Point (1) of Proposition~\ref{Proposition goodness chosen line bundle for bow variety} implies that the line bundles $\Sh{L}^{\triangledown}_{A,F}\boxtimes \Sh{L}_X$ are independent of the choices of the classes $\eta_k$, up to isomorphism.

\begin{lemma}
\label{lemma: poles stab}
    Fix $A\subset \At$. The line bundle $\Sh{L}_X$ is non-degenerate, i.e. the complement of resonant locus $\Delta$ in 
    \[
    E_{\Tt^!}(\pt).
    \]
    is open and dense. Moreover, $\Delta$ is contained in the complement of the union of codimension-one abelian varieties of the form $\{z_i/z_j\hbar^{\alpha_{ij}}=1\}$ for some $\alpha_{ij}\in \Z$.
\end{lemma}
\begin{proof}
    By the first two points of Proposition \ref{Proposition goodness chosen line bundle for bow variety}, it suffices to assume that $X$ is separated. In this case there is no ambiguity in the choice of $\eta_{k}$. Moreover, the product $\prod_{k}\det(\eta_k)$ is the canonical ample line bundle of the GIT quotient $X$. Hence, the first claim follows from \cite[Proposition 2.6]{okounkov2020inductiveI}. The second statement is a direct generalization of \cite[Section 2.3.7]{okounkov2020inductiveI}. An independent proof in the important case $A=\At$ follows from the known result in the case when $X$ is a partial flag variety and an iteration of Proposition~\ref{multiple terms D5 resolution stabs separated case}.
\end{proof}

\begin{theorem}
Stable envelopes for a bow variety exist for any $A\subseteq \At$ and are unique. Moreover, they are sections 
\[
\Stab{C}{F}\in \Gamma( \Sh{L}^{\triangledown}_{A,F}\boxtimes \Sh{L}_X)_{mer}
\]
meromorphic in the variables $z$ and $\h$ and holomorphic in all the remaining ones.
\end{theorem}

\begin{proof}
    The theorem now follows from Okounkov's general theory. The first claim follows from Theorem~\ref{Theorem existence+uniqueness stable envelopes} applied to our chosen good line bundle $\Sh{L}_X$. The second one follows from Lemma~\ref{lemma: poles stab}.
\end{proof}

As a straightforward consequence of item (2) of Proposition \ref{Proposition goodness chosen line bundle for bow variety} and uniqueness of stable envelopes, we also obtain
\begin{corollary}
\label{Corollary: stable envelopes match under HW}
    Under the isomorphism $X_1\cong X_2$ induced by a Hanany-Witten transition, the stable envelopes of $X_1$ and $X_2$ are identified.
\end{corollary}

\subsection{Stable envelopes for $T^*\mathbb{P}^n$}
\label{sec:StabForPn}

Let $n\in \NN$ and consider the brane diagrams
\[  \DD=\ttt{\fs 1\fs n\bs n-1\bs \ldots \bs 2\bs 
	1\bs}
\qquad\qquad \text{and} \qquad\qquad
\DD'=\ttt{\bs 1\bs 2\bs\ldots\bs n-1\bs n\fs 1\fs}.
\]
The associated variety in both cases is $T^*\mathbb{P}^{n-1}$.

Let $f_k\in\Ch(\DD)$ be the fixed point whose diagram has ties connecting $\Zb_1$ with $\Ab_k$, as well as $\Zb_2$ with $\Ab_l$ for all $l\not=k$. Let  $f_k'\in\Ch(\DD')$ be the fixed point whose diagram has ties connecting $\Zb_2$ with $\Ab_k$, as well as $\Zb_1$ with $\Ab_l$ for all $l\not=k$. Using the shorthand notation $t$ and $t'$ for the Chern class of the leftmost and rightmost (line) bundles over $\Ch(\DD)$ and $\Ch(\DD')$ we have that the  restriction maps to $f_k$ and $f'_k$ are $t \mapsto a_k\h^{-1}$, $t'\mapsto a_k\h$ respectively.  

\begin{proposition}\label{prop:Pn}
	Let $\mathfrak{C}=\lbrace a_1<\dots<a_n\rbrace$ be the standard chamber (cf. Section \ref{subsection: hamber structure}). We have 
	\[
	\Stab{C}{f_k} =\prod_{i=1}^{k-1} \theta\left(\frac{a_i}{t}\right)\cdot 
	\frac{\theta\left(\frac{t}{a_k}\frac{z_1}{z_2}\h^{k-1}\right)}{\theta\left(\frac{z_1}{z_2}\h^{k-2}\right)}
	\cdot 
	\prod_{i=k+1}^n \theta\left( \frac{t}{a_i}\h\right),
	\]
	\[
	\Stab{C}{f'_k} =\prod_{i=1}^{k-1} \theta\left(\frac{a_i}{t}\h^2\right)\cdot 
	\frac{\theta\left(\frac{t}{a_k}\frac{z_1}{z_2}\h^{k-3}\right)}{\theta\left(\frac{z_1}{z_2}\h^{k-2}\right)}
	\cdot 
	\prod_{i=k+1}^n \theta\left( \frac{t}{a_i}\h^{-1}\right).
	\]
\end{proposition}

These are, in fact, not new theorems, they are equivalent to \cite[Section 3.4]{aganagic2016elliptic} and \cite[Section 5.5]{Felder2018}. To convince the reader about the correctness of the conventions, let us verify that the given formulas are sections of the required line bundles. The definitions of Section \ref{subsection Quadratic forms} for $\Ch(\DD)$ and $\Ch(\DD')$ give
\begin{eqnarray*}
	\alpha\equiv \sum_{i=1}^n \frac{t}{a_i}\h,
	&
	QU=2t(z_1-z_2-\h),
	& 
	N_{f_k/X}^<=\sum_{i=1}^{k-1}\frac{a_i}{a_k}\h+\sum_{i=k+1}^n \frac{a_k}{a_i},
	\\
	\alpha' \equiv \sum_{i=1}^n \frac{a_i}{t}\h^2, 
	& 
	QU'=2t(z_1-z_2+(n-1)\h),
	&
	N_{f'_k/X}^<=\sum_{i=1}^{k-1}\frac{a_i}{a_k}\h+\sum_{i=k+1}^n \frac{a_k}{a_i}.
\end{eqnarray*}
Hence, the quadratic forms of $\Stab{}{f_k}$ and $\Stab{}{f'_k}$ are required to be
\begin{multline}\label{Fiorentina}
	\sum_{i=1}^n(t-a_i+\h)^2   -\sum_{i=1}^{n}(a_k-a_i)^2 + 2(t-a_k+\h)(z_1-z_2-\h) 
	\\
	+\sum_{i=1}^{k-1}(a_i-a_k+\h)^2 + \sum_{i=k+1}^n (a_k-a_i)^2,   
\end{multline}
\begin{multline}\label{Lazio}
	\sum_{i=1}^n(a_i-t+2\h)^2   -\sum_{i=1}^{n}(a_i-a_k+h)^2 + 2(t-a_k-\h)(z_1-z_2+(n-1)\h) 
	\\
	+\sum_{i=1}^{k-1}(a_i-a_k+\h)^2 + \sum_{i=k+1}^n (a_k-a_i)^2, 
\end{multline}
respectively. The quadratic forms of the formulas in the text of the proposition are 
\[
\sum_{i=1}^{k-1}(a_i-t)^2 + (t-a_k+z_1-z_2+(k-1)\h)^2-(z_1-z_2+(k-2)\h)^2 + \sum_{i=k+1}^n(t-a_i+\h)^2,
\]
\[
\sum_{i=1}^{k-1}(a_i-t+2\h)^2 + (t-a_k+z_1-z_2+(k-3)\h)^2-(z_1-z_2+(k-2)\h)^2 + \sum_{i=k+1}^n(t-a_i-\h)^2.
\]
Straightforward calculation shows that these are equal to \eqref{Fiorentina} and \eqref{Lazio}, respectively.

\subsection{Stable envelopes as morphisms}
It is sometimes useful to interpret stable envelopes as morphisms rather than sections. This is obtained by looking at $\Stab{C}{F}$ as a correspondence on $F\times X$, cf. Section~\ref{subsection action by correspondences}. However, we first need the following properness result.

\begin{lemma}
\label{lemma properness full attracting set}
    The composition $\Attfull{C}(F)\hookrightarrow F\times X\to X$ is proper.
\end{lemma}
\begin{proof}
    The argument is the same as the one in the proof of \cite[Prop. 3.5.1]{maulik2012quantum}. The key point is that
    any bow variety $X$ admits an $\At$-equivariant proper morphism $\pi: X\to X_0$ with $X_0$ affine---see Section~\ref{sec:Affinization}. 
\end{proof}

Since $\Stab{C}{F}$ is supported on $\Attfull{C}(F)$, which we now know to be proper over $F$, we can apply the construction of Section \ref{subsection action by correspondences} to obtain a morphism of $\Sh{O}_{\Base{\Tt}{\Tt^!}, mer}$-modules
\[
\Stab{C}{F}: \Sh{L}_{A,F}\to \Sh{L}_X.
\]
As before, the subscript $mer$ stands for meromorphic sections in the variables $z$ and $\h$.

\subsection{Composition of stable envelopes}

Let $X$ be a bow variety with $n\geq 2$ D5 branes and choose a chamber 
\[
\mathfrak{C}=\lbrace a_{\sigma(1)}< \dots <a_{\sigma(n)}\rbrace
\]
for the action of $\At$ on $X$, cf. Section \ref{subsection: hamber structure}.
A choice of partition $n=n'+n''$ gives rise to a partition of the set of D5 branes of $X$ into two disjoint subsets, acted on by the subtori
\[
\At'=\Set{(a_{\sigma(1)}, \dots a_{\sigma(n')}}{a_i\in \Cs}\subset \At \qquad \At''=\Set{(a_{\sigma(n'+1)}, \dots a_{\sigma(n)}}{a_i\in \Cs}\subset \At.
\]
Consider also the rank two subtorus
\[
A_0=\Set{(a',a'')}{a',a''\in \Cs}\subset \At
\]
acting on the D5 branes $\Ab_{\sigma(1)},\dots,\Ab_{\sigma(n')}$ with $a'$ and on the remaining D5 branes $\Ab_{\sigma(n'+1)},\dots,$ $\Ab_{\sigma(n)}$ with $a''$. 
By Theorem \ref{proposition tensor product decomposition fixed locus} and Example \ref{Example fixed point-fixed component interaction}, for any fixed point $f\in X^{\At}$ we have a commutative diagram 
\[
\begin{tikzcd}
  f'\times f''\arrow[d, equal] \arrow[r, hookrightarrow] & X'\times X''\arrow[d, hookrightarrow]
 \\
 f \arrow[r, hookrightarrow] & X
\end{tikzcd}
\]
where $X'\times X''$ is some $A_0$-fixed component in $X$ and $f'\in (X')^{\At'}$ (resp. $f''\in (X'')^{\At''}$). 
The chamber $\mathfrak{C}$ induces two chambers
\[
\mathfrak{C}'=\lbrace a_{\sigma(1)}< \dots <a_{\sigma(n')}\rbrace\qquad \text{and} \qquad \mathfrak{C}''=\lbrace a_{\sigma(n'+1)}< \dots <a_{\sigma(n)}\rbrace
\]
for the actions of $\At'$ and $\At''$. Notice that the stable envelope
\begin{equation*}
    \label{tensor product stable envelopes}
    \StabX{X'\times X''}{C}{f}=\StabX{X'}{C'}{f'}\boxtimes \StabX{X''}{C''}{f''}
\end{equation*}
depends on two sets of  K\"ahler variables $z'$ and $z''$ attached to $X'$ and $X''$ respectively. 

Consider now the chamber $\mathfrak{C}_0=\lbrace a'< a''\rbrace $ for $A_0$ and the stable envelope 
\[
\text{Stab}^{X}_{\mathfrak{C}_0}(X'\times X'').
\]
Like $\StabX{X}{C}{f}$, it only depends on one set of K\"ahler parameters. The following statement is the bow variety version of the so-called triangle lemma of \cite[Lemma 3.6.1]{maulik2012quantum} and \cite[Prop. 3.3]{aganagic2016elliptic}. It states one of the fundamental properties of stable envelopes.

\begin{proposition}
\label{composition stable envelopes}
    We have
    \[
    \StabX{X}{C}{f}(z)= \text{Stab}^{X}_{\mathfrak{C}_0}(X'\times X'')(z) \circ \left(\StabX{X'}{C'}{f'}(z)\boxtimes \StabX{X''}{C''}{f''}(z{\h}^{-r'})\right)
    \]
    Here, $r'$ is the NS5 charge vector of $X'$ and the symbol $\circ$ denotes the convolution product \eqref{equation definition convolution}.
\end{proposition}
Because of the need to compare the line bundles, the proof is somewhat long and technical, so we defer it to the end of the section.

\subsection{Fixed point restrictions of stable envelopes}

Let $X$ be a bow variety with $n$ D5 branes and $m$ NS5 branes. We now focus on the action of the torus $\At$ acting on all $n$ D5 branes. As discussed in Section~\ref{sec:fixedpoints}, the fixed locus $X^{\At}$ is finite. We index the fixed points consistently with the partial order defined in Section \ref{subsection attracting loci}, namely, $f_i\leq f_j$ implies $i\leq j$.

The stable envelope restrictions 
\begin{equation*}
    \label{stable envelope restriction}
    (S_{\mathfrak{C}})_{ij}:=\Stab{C}{f_j}\Big|_{f_i}
\end{equation*}
are sections of some bundle on the abelian variety
\[
\Base{\Tt}{\Tt^!}= E^{n}\times E^{m}\times E_{\h},
\]
and hence can be seen as functions in the variables $(a,z,\h)$ with prescribed quasi-periods. We collect all these restrictions in the stable envelope matrix $S_{\mathfrak{C}}$. By Remark~\ref{Remark axiomatic stab}, it follows that $S_{\mathfrak{C}}$ is an upper triangular matrix of the form 
\[
S_{\mathfrak{C}}=
\begin{pmatrix}
    \vartheta(N^-_{f_1}) & (S_{\mathfrak{C}})_{12} & (S_{\mathfrak{C}})_{13} & \dots &  (S_{\mathfrak{C}})_{1k}\\
           & \vartheta(N^-_{f_2}) & (S_{\mathfrak{C}})_{23} & \dots &  (S_{\mathfrak{C}})_{2k}\\
           &        & \vartheta(N^-_{f_3}) & \dots &  (S_{\mathfrak{C}})_{3k}\\
           &        &        & \ddots&  \vdots\\
           &        &        &       & \vartheta(N^-_{f_k})\\
\end{pmatrix}.
\]
Its diagonal entries $\vartheta(N^-_{f_i})$ are the elliptic Euler classes of the negative parts of the normal bundles $N_{f_i}$ of $f_i$ in $X$. In particular, they only depend on the variables $(a,\h)$. On the other hand, the nontrivial coefficients in the strictly upper triangular part will generally depend on all three types of variables $(a,z,\h)$. Since all the diagonal entries are non-zero, the matrix $S_{\mathfrak{C}}$ is invertible. We now describe its inverse explicitly.

Let $\vartheta(TX)^{-1}$ denote the diagonal matrix whose entry $(\vartheta(TX)^{-1})_{ii}$ is equal to $1/\vartheta(T_{f_i} X)$. Since the fixed points $f_i$ are isolated, none of the $\At$-weights of $T_{f_i} X$ is zero; hence, this matrix is well defined.
\begin{proposition}
\label{proposition duality stable envelopes}
    Let $\mathfrak{C}^{opp}$ be the chamber opposite to $\mathfrak{C}$. Then 
    \[
    \left(S_{\mathfrak{C}}(a,z,\h)\right)^{-1}=\left(S_{\mathfrak{C}^{opp}}(a,z^{-1}\h^{r},\h)\right)^T \vartheta(TX)^{-1}. 
    \]
    Here, $z^{-1}\h^{r}$ means $z^{-1}_ih^{r_i}$ for all $i=1,\dots,m$ and $(-)^T$ stands for transposition.
\end{proposition}

Notice that the Bruhat order induced by $\mathfrak{C}^{opp}$ is opposite to that induced by $\mathfrak{C}$. Hence the matrix $S_{\mathfrak{C}^{opp}}$ is lower triangular, and both sides of the equation above are upper triangular.

\begin{proof}
    The proof is an adaptation of \cite[Prop. 3.4]{aganagic2016elliptic}. Consider the composition 
    \[
    M:=\left(S_{\mathfrak{C}^{opp}}(a,z^{-1}\h^{r},\h)\right)^T \vartheta(TX)^{-1} S_{\mathfrak{C}}(a,z,\h).
    \]
    Since the opposite chamber swaps positive and negative weights and $\vartheta(T_fX)=\vartheta(N_{f}^+)\vartheta(N_{f}^-)$, it is immediate to check that the diagonal entries are equal to $1$. 
    Let now $i\neq j$. By the support axiom, $M_{ij}=0$ unless $f_i>f_j$, so it suffices to consider this case. Since the chambers are opposite to each other, the class
    \begin{equation}
        \label{equation in proof inverss stab}
        \widetilde{M}=\Delta^{\oast}(\text{Stab}_{\mathfrak{C}^{opp}}(f_i)(z^{-1}h^r)\boxtimes \Stab{C}{f_j}(z))
    \end{equation}
    is proper over the point, cf. \cite[Theorem 4.4.1]{maulik2012quantum}. A direct computation similar (but easier) to the one in the proof of Proposition \ref{composition stable envelopes} shows that the pushforward $p_{\oast}(\widetilde{M})$ associated with $p:X\to \pt$ is well defined, and hence gives a section of a line bundle over $\Base{\Tt}{\Tt^!}$ that is regular in the variables $a$. Moreover one can check that for any generic point $(z,\h)\in E_{\Tt^!}$ the restriction of this line bundle to the abelian variety 
    \[
    E_{\At}(\pt)\times \lbrace z\rbrace \times \lbrace \h\rbrace \subset \Base{\Tt}{\Tt^!}
    \]
    is nontrivial and of degree zero. Since such a line bundle has no nonzero global sections, we have $p_{\oast}(\widetilde{M})=0$. On the other hand, Remark \ref{Remark stabs are in image localized pushforward} assures that the hypothesis of Proposition \ref{localization formula pushforward} is satisfied, so we can use the latter to compute $p_{\oast}(\widetilde{M})=0$\footnote{The proofs in Appendix \ref{Appendix: Equivariant localization for bow varieties} are based on results developed later in the paper. However, none of those results rely on Proposition~\ref{proposition duality stable envelopes}, so there is no circular argument.}. The result is exactly $M_{ij}$. Overall, this argument shows that $M_{ij}=\delta_{ij}$, as required.
    
    We remark that the shift of the K\"ahler parameters by $\h^{r}$ is equivalent to passing to the opposite class $\h\alpha^\vee$ in the line bundle $\Sh{L}=\Theta(\alpha)\otimes \Sh{U}$ entering in the definition of $\text{Stab}_{\mathfrak{C}^{opp}}(f_i)$. This opposite class is needed to compensate the $\At$-weights coming from $\vartheta(TX)^{-1}$ and $\Stab{C}{f_j}$ to get a degree zero line bundle.
\end{proof}


\subsection{R-matrices} \label{sec:Rmatrix}

Fix a bow variety $X$ and let $\mathfrak{C}$ and $\mathfrak{C}'$ be two chambers for the $\At$-action. We define the ``geometric R matrix'' as
\[
\Rmat{C}{C'}:= (S_{\mathfrak{C}})^{-1}\circ S_{\mathfrak{C}'}.
\]
It is a matrix whose entries can be seen either as meromorphic sections of line bundles over $\Base{\Tt}{\Tt^!}$ or as functions in the variables $(a,z,\hbar)$.
Its definition can be rephrased to the so-called R-matrix relation for the stable envelopes:
\begin{equation}
\label{localized R-matrix relations}
    \Stab{C'}{f}\Big|_{h}=\sum_{g\in X^{\At}} \left( \Rmat{C}{C'}\right)_{gf}\Stab{C}{g}\Big|_{h}\qquad \forall f,h\in X^{\At}.
\end{equation}

We will need a sharper statement. In the next proposition, we show that the equations above ``integrate'' to the whole (localized) elliptic cohomology scheme of $X$ to give 
\begin{equation}
    \label{R-matrix relations}
    \Stab{C'}{f}=\sum_{g} \left( \Rmat{C}{C'}\right)_{gf}\Stab{C}{g} \qquad \forall f\in X^{\At}.
\end{equation}
The need for localizing, i.e. inverting the equivariant parameters, is due to the R-matrix, which is meromorphic in these parameters.
\begin{proposition}
\label{proposition R matrix relations}
    Equation \eqref{R-matrix relations} holds for any bow variety $X$, fixed point $f\in X^{\At}$ and pair of chambers $\mathfrak{C},\mathfrak{C'}$.
\end{proposition}
 
\begin{proof}
    Using the definition of $\Rmat{C'}{C}^X$, it is easy to see that both sides are meromorphic sections of the same line bundle.
    Then \eqref{localized R-matrix relations} would immediately imply \eqref{R-matrix relations} if the localization map $i^{\oast}$ associated with the inclusion $X^{\At}\hookrightarrow X$ were injective. Among the bow varieties, this is known for type-A quiver varieties but not in general. However, since the class $TX|_{X^{\At}}$ has no trivial $\At$-weights, the map $i^{\oast}i_{\oast}=\vartheta(TX|_{X^{\At}})\cdot$ is injective. Therefore, $i^{\oast}$ is always injective on the image of $i_{\oast}$. 
    Thus, it would suffice to know that the stable envelopes (and hence both sides of \eqref{R-matrix relations}) are in the image of the pushforward $i_{\oast}$. This latter is the content of Remark \ref{Remark stabs are in image localized pushforward}.
\end{proof}

\begin{remark}
    In section \ref{subsection fusion of R-matrices}, we will give a different argument that also provides an explicit formula for the entries $(\Rmat{C'}{C})_{gf}$ in terms of those of partial flag varieties.
\end{remark}

Assume now that the two chambers $\mathfrak{C}$ and $\mathfrak{C}'$ are separated by a single wall, i.e.
\begin{align*}
    \mathfrak{C}&=\lbrace a_{i_1}<\dots< a_{i_k}< a_{i_{k+1}}< \dots< a_{i_n}\rbrace,
    \\
    \mathfrak{C}'&=\lbrace a_{i_1}<\dots< a_{i_{k+1}}< a_{i_k}< \dots< a_{i_n} \rbrace.
\end{align*}
Let $A\subset \At$ be the subtorus given by the equation $a_{i_k}= a_{i_{k+1}}$. 
By multiple applications of Theorem \ref{proposition tensor product decomposition fixed locus}, the components of the $A$-fixed locus of $X=X(\D, r,c)$ are of the form 
\begin{multline*}
\prod_{j=1}^{k-1} X(\D^{(i_j)}, r^{(i_j)},c^{(i_j)})\times  X(\D^{(i_k,i_{k+1})}, r^{(i_k)}+r^{(i_{k+1})}),c^{(i_k)}\sqcup c^{(i_{k+1})}) \times
\\
\prod_{j=k+2}^{n} X(\D^{(i_j)} , r^{(i_j)},c^{(i_j)}),
\end{multline*}
where $\sum r^{(i_j)}=r$ and $\bigsqcup c^{(i_j)}=c$. The diagrams $\D^{(i_j)}$ are obtained from $\D$ by erasing all D5 branes but $\Ab_{i_j}$, while $\D^{(i_k,i_{k+1})}$ is obtained by erasing all of them except $\Ab_{i_k}$ and $\Ab_{i_{k+1}}$. Notice that all the varieties in the product except the central one, which we denote by $F$, are singletons. Nevertheless, it is useful to keep them in the notation. 

On every $A$-fixed component $F$ there is a residual action of $\At/A$ and the chambers $\mathfrak{C}$ and $\mathfrak{C}'$ induce the only two chambers 
$\mathfrak{C}_+=\lbrace a_{i_k}< a_{i_{k+1}}\rbrace$ and $\mathfrak{C}_-=\lbrace a_{i_{k+1}}< a_{i_k}\rbrace$ of $\At/A$. 
Set 
\[
R^{(i_k,i_{k+1})}= \Rmat{C_+}{C_-}.
\]
Combining Proposition \ref{composition stable envelopes} and Proposition \ref{proposition R matrix relations}, we deduce the following refined statement:
\begin{corollary} Given two chambers $\mathfrak{C}$ and $\mathfrak{C}'$ separated by the wall $a_{i_k}=a_{i_{k+1}}$ and two fixed points $f,g\in F^{\At/A}$, we have
\begin{equation*}
    \StabX{X}{C'}{f}(z)=
        \sum_{g\in F^{\At/A}} \left( R^{(i_k,i_{k+1})}\right)_{gf}(z\h^{-r^{(<k)}})\StabX{X}{C}{g}(z)
\end{equation*}
where $r^{(<k)}=\sum_{j=1}^{k-1} r^{(i_j)}$. 

\end{corollary}
As a result, by wall crossing, every R-matrix can be written as a product of simpler matrices of the form $R^{(i,j)}$. An important special case is when the bow variety has exactly three D5 branes and we consider the chambers 
\[
\mathfrak{C}=\lbrace a_1<a_2<a_3\rbrace  \qquad 
    \mathfrak{C}'=\lbrace  a_3< a_2 < a_1 \rbrace.
\]
To describe this situation effectively, it is best to switch to additive notation.

\begin{corollary}
Crossing walls from $\mathfrak{C}$ to $\mathfrak{C}'$ in the two possible ways, we get
\begin{equation*}
    R^{(1,2)}(z-r^{(3)}\h)R^{(1,3)}(z)R^{(2,3)}(z-r^{(1)}\h)
    =R^{(2,3)}(z)R^{(1,3)}(z-r^{(2)}\h)R^{(1,2)}(z).
\end{equation*}
Hence, the matrices $R^{(i,j)}$ solve the dynamical Yang-Baxter equation.
\end{corollary}

\noindent A direct consequence of the definition of the R-matrix and of Proposition \ref{proposition duality stable envelopes} is

\begin{corollary}
\label{symmetry R matrix}
The $R$-matrix $R_{\mathfrak{C},\mathfrak{C}^{opp}}$ is symmetric up to a shift, that is 
\[
\left( R_{\mathfrak{C},\mathfrak{C}^{opp}}\right)_{fg}(a,z,\h)= \left( R_{\mathfrak{C},\mathfrak{C}^{opp}}\right)_{gf}(a,z^{-1} \h^{r},\h).
\]
\end{corollary}

\subsection{Disregarding non-essential 5-branes}
\label{section: Getting rid of the ``trivial'' branes}
Let $X$ be a bow variety with $m$ NS5 branes and $n$ D5 branes. 
In this section, we argue that the stable envelopes are essentially unaffected by the presence of D5 branes with charge $0$ or $m$, or by the presence of NS5 branes of charge $0$ or $n$.

A D5 brane $\Ab$ has charge $0$ (resp. $m$) if and only if in the separated (resp. co-separated) representative of its HW equivalence class its local charge is $\w(\Ab)=0$. Dually, an NS5 brane $\Zb$ has charge $0$ (resp. $n$) if and only if in the separated (resp. co-separated) representative of its HW equivalence class its local charge is $\w(\Zb)=0$.
Since by Corollary \ref{Corollary: stable envelopes match under HW} the HW isomorphism identifies the stable envelopes, it suffices to restrict ourselves to the separated or co-separated settings.

Let us begin with NS5 branes. The insertion of an NS5 brane $\Zb$ satisfying $\w(\Zb)=0$ in a separated or co-separated brane diagram induces an isomorphism of the associated bow varieties, and hence an equality of stable envelopes, up to the appropriate identification of the K\"ahler parameters. This is a degenerate case of the more general results discussed in Section \ref{Fusion of NS5 branes}. There, we will build a correspondence interpolating an arbitrary separated or co-separated bow variety with the one obtained by replacing an NS5 brane $\Zb$ with a pair of adjacent NS5 banes of weight $\w(\Zb')$ and $\w(\Zb'')$ such that $\w(\Zb')+\w(\Zb'')=\w(\Zb)$. If $\w(\Zb')=0$ or $\w(\Zb'')=0$, the correspondence reduces to the aforementioned isomorphism. At the level of elliptic stable envelopes, Theorem \ref{main theorem proof NS5 resolution for stabs} shows that this isomorphism induces an identification of the stable envelopes. Notice that this argument implies that the stable envelopes do not depend on the K\"ahler parameters attached to a brane $\Zb$ with $\w(\Zb)=0$.

We now move to D5 branes. Let now $X$ be a separated or co-separated bow variety with at least one D5 brane $\Ab$ such that $\w(\Ab)=0$ and consider the one-dimensional subtorus $\At_0 \subset \At$ acting on $\Ab$. By \ref{proposition tensor product decomposition fixed locus}, the bow variety $Y$ whose brane diagram is obtained from the one of $X$ by removing $\Ab$ can be seen as a $\At_0$-fixed subvariety of $X$ (the other bow variety in the fiber product is a singleton). Actually, since $\w(\Ab)=0$, $Y$ is the unique $\At_0$-fixed component, and hence all the $\At$-fixed points are contained in $X$ together with all the $\At$-equivariant curves connecting them. This forces the stable envelopes of $X$ and $Y$ to be essentially the same.

Namely, let $\mathfrak{C}$ be a chamber for the action of $\At$ on $X$, and let $\mathfrak{C}_0$ be the induced chamber on $\At_0$. By the same argument used in the proof of Proposition \ref{composition stable envelopes}, 
we obtain 
\begin{equation}
    \label{equation stab with weight zero banes}
    \text{Stab}_{\mathfrak{C}}(f)\Big|_g= \left(\text{Stab}_{\mathfrak{C}}(Y)\circ  \text{Stab}_{\mathfrak{C}/\mathfrak{C}_0}(f)\right) \Big|_Y \Big|_g=  \Theta(N^-_{Y/X})\Big|_g\text{Stab}_{\mathfrak{C}/\mathfrak{C}_0}(f)\Big|_g.
\end{equation}
That is, the fixed point restrictions of stable envelopes are equal, up to a rescaling by $\Theta(N^-_{Y/X})|_g$.

Since our mirror symmetry statement only involves ratios of stable envelope restrictions, it is sufficient to prove mirror symmetry of stable envelopes for bow varieties with no charge zero NS5 branes.

\subsection{Proof of Proposition \ref{composition stable envelopes}}
\begin{proof}
    The topological part of the proof is standard, so we only sketch it. The original references are \cite[Lemma 3.6.1]{maulik2012quantum} and \cite[Prop. 3.3]{aganagic2016elliptic}. Abbreviate $F=X'\times X''$. Assume temporarily that the convolution is well defined and both sides are sections of the same line bundle. Then, Lemma \ref{lemma support convolution} and Lemma \ref{lemma properness full attracting set} imply that the convolution is supported on 
    \[
    \Set{(f,x)\in \lbrace f \rbrace \times X}{\exists \, f'\in F^{\At} \text{ s.t. } (f,f')\in \Attfull{C}^{F}(f) , \,   (f',x)\in\text{Att}_{\mathfrak{C}_0}^{X,f}(f') }.
    \]
    Using an equivariant embedding of $X$ in a projective variety, one can check that the set above is contained in $\Attfull{C}^{X}(f)$. This implies that the convolution product of the stable envelopes satisfies the support axiom for $\StabX{X}{C}{f}$. Similarly, one argues that the diagonal axiom is also satisfied. 
    
    To complete the proof, it remains to check that the convolution is well-defined and the line bundles match. We perform this check in full detail. The topological condition for the convolution, namely properness of the support of $\text{Stab}^{X}_{\mathfrak{C}_0}(X'\times X'')(z)$ over $X_1\times X_2$, directly follows from Lemma \ref{lemma properness full attracting set} (with $F=X'\times X''$), so it only remains to look at the line bundles.

    Let $\tau$ denote the map $z\mapsto (z, z\h^{-r'})$. Consider the classes
    \begin{align*}
        \alpha_{F,f}&:=\StabX{X'}{C'}{f'}\boxtimes \StabX{X''}{C''}{f''}\in \Gamma(\Sh{L}_{\At',f'}^{\triangledown}\boxtimes \Sh{L}_{\At'',f''}^{\triangledown})\boxtimes (\Sh{L}_{X'}\boxtimes \Sh{L}_{X''}))_{mer}
        \\
        \alpha_{X,F}&:=\text{Stab}^{X}_{\mathfrak{C}_0}(X'\times X'')\in \Gamma(\Sh{L}^{\triangledown}_{A_0,F}\boxtimes \Sh{L}_X)_{mer}
        \\
        \alpha_{X,f}&:=\StabX{X}{C}{f}\in \Gamma(\Sh{L}^{\triangledown}_{\At,f}\boxtimes \Sh{L}_X)_{mer}
    \end{align*}
    The proposition claims that $\alpha_{X,F} \circ \tau^{\oast}\alpha_{F,f}=\alpha_{X,f}$, and by definition of convolution we need to prove that 
    \begin{equation}
        \label{equation proof composition stable envelopes}
        \tau^*((\Sh{L}_{\At',f'}^{\triangledown}\boxtimes \Sh{L}_{\At'',f''}^{\triangledown})\boxtimes \left( \tau^*(\Sh{L}_{X'}\boxtimes \Sh{L}_{X''}) \otimes \Sh{L}^{\triangledown}_{A_0,F}\right)\boxtimes \Sh{L}_X\cong \Sh{L}^{\triangledown}_{\At,f}\boxtimes\Theta(TF)\boxtimes \Sh{L}_X
    \end{equation}
    on $E_{\Tt\times \Tt^!}(f\times F\times X)$. Effectively, the left-hand side is the tensor product of the line bundles of $\tau^{\oast}\alpha_{F,f}$ and $\alpha_{X,f}$ while the right-hand side is the line bundle of $\alpha_{X,f}$, tensored with $\Theta(TF)$.
    
    First, we reduce to the case when $X$ is separated. Let $X_1$ be an arbitrary bow variety and $X_2$ be the separated bow variety isomorphic to $X_1$ via Hanany-Witten isomorphism. 
    Since the latter is $\At$-equivariant, we have a commutative diagram 
    \[
    \begin{tikzcd}
          f_1'\times f_1''\arrow[d, equal] \arrow[r, hookrightarrow] & X_1'\times X_1''\arrow[r, hookrightarrow]\arrow[d, "\sim" vlabl] & X_1 \arrow[d, "\sim" vlabl]
         \\
          f_2'\times f_2''\arrow[r, hookrightarrow] & X_2'\times X_2''\arrow[r, hookrightarrow] & X_2
    \end{tikzcd}
    \]
in which the vertical arrows are all Hanany-Witten isomorphisms. 

By Corollary \ref{Corollary: stable envelopes match under HW}, these isomorphisms identify the stable envelopes; hence, the same holds for their line bundles. Since the charge $r'$ is invariant under Hanany-Witten transitions, equation \eqref{equation proof composition stable envelopes} holds 
for $X_1$ iff it holds for its separated counterpart $X_2$. Thus, it suffices to prove \eqref{equation proof composition stable envelopes} assuming that $X$ is separated. 

We claim that 
    \begin{equation}
    \label{second equation proof composition stable envelopes}
        \tau^*(\Sh{L}_{X'}\boxtimes \Sh{L}_{X''})\otimes(\Sh{L}_{A_0,F})^{-1}=\Sh{G},
    \end{equation}
    for some $\Sh{G}$ is pulled back from $\Base{\Tt}{\Tt^!}$. Restricting this equation to $f=f'\times f''$ and using the identity $N^-_{f'\times f''/X}=N^-_{f',X'}+ N^-_{f'',X''}+N^-_{F/X}$, it is straightforward to deduce that
    \[
    \tau^*\left(\Sh{L}_{\At',f'}^{\triangledown}\boxtimes \Sh{L}_{\At'',f''}^{\triangledown}\right)\otimes \Sh{L}_{\At,f}=\Sh{G}^{-1}.
    \]
    Combining the two equations above, \eqref{equation proof composition stable envelopes} follows. Therefore, it remains to prove \eqref{second equation proof composition stable envelopes}. Since $X$ is separated, Proposition \ref{prop:rightbundles} implies that
    \[
    \alpha_X\equiv \h \left( 
    \bigoplus_{i=1}^{m-1} \Hom( \xi_{i+1}, \xi_{i} )-
    \Hom( \xi_{i}, \xi_{i})
    \right)
    \]
    Here, $m$ is the number of NS5 branes, and the tautological bundles are ordered from left to right.
    As usual, $\equiv$ denotes equality up to some class in $K_{\Tt}(\pt)$. 
    Since $X$ is separated, we also have 
    \[
    \Sh{U}_X=\bigotimes_{i=1}^{m-1}\Sh{U}\left(\xi_{i},\frac{z_i}{z_{i+1}}\h^{\rk(\xi_{i-1})-\rk(\xi_{i})}\right).
    \]
    Similarly, the class $\alpha_{X'}$ and the line bundle $\Sh{U}_{X'}$ (resp. $\alpha_{X''}$ and $\Sh{U}_{X''}$) are obtained by replacing $\xi_i$ with $\xi_i'$ (resp. $\xi_i''$) in the formulas above. Also notice that $\xi_i|_F=\xi_i'\oplus \xi_i'' $ for every $i$. 
    
    Since $F=X'\times X''$ is $A_0$-fixed, we get 
    \[
    \alpha_X|_F\equiv \alpha_{X'}+\alpha_{X''} +\alpha_X|_F^+ +\alpha_X|_F^-,
    \]
    where $+\alpha_X|_F^+$ and $\alpha_X|_F^- $ denote the attracting and repelling components of $\alpha$ with respect to the chamber $\mathfrak{C}_0$. 
    Similarly, equation \eqref{equation separated alpha is a polarization} implies that 
    \[
    N^-_{F/X}= +\alpha_X|_F^-+ \hbar(\alpha_X|_F^+)^\vee.
    \]
From this analysis we deduce that $\tau^*(\Sh{L}_{X'}\boxtimes \Sh{L}_{X''})\otimes(\Sh{L}_{A_0,F})^{-1}=$
    \begin{align}
        &\Theta(\alpha_{X'}+\alpha_{X''})\otimes  \tau^*(\Sh{U}_{X'}\boxtimes \Sh{U}_{X''})\otimes \left( \Theta(\alpha_X)\otimes \Sh{U}_X \otimes \Theta(-N^-_{F/X}) \right)\Big|_{F}^{-1} \nonumber
        \\
        &\equiv \Theta(\hbar(\alpha_X|_F^+)^\vee- \alpha_X|_F^+)\otimes \tau^*(\Sh{U}_{X'}\boxtimes \Sh{U}_{X''})\otimes \Sh{U}_X\Big|_F^{-1}. \label{identity to use for composition stable envelopes}
    \end{align}
Observing that 
\[
\hbar(\alpha_X|_F^+)^\vee\equiv \bigoplus_{i=1}^{m-1} \Hom( \xi_{i}'', \xi_{i+1}' )
   - \Hom( \xi_{i}'', \xi_{i}')
\]
and applying the last claim of Lemma \ref{Lemma properties U bundle} multiple times, we have 
\[
\Theta(\hbar(\alpha_X|_F^+)^\vee- \alpha_X|_F^+)\equiv \bigotimes_{i=1}^{m-1} \Sh{U}\left(\xi_i', \h^{\rk(\xi_{i-1}'')-\rk(\xi_i'')}\right)\otimes \Sh{U}\left(\xi_i'', \h^{\rk(\xi_{i+1}')-\rk(\xi_i')}\right).
\]
On the other hand, using the third claim of the same lemma, we get 
\[
\Sh{U}_X\Big|_F^{-1}\cong
\bigotimes_{i=1}^{m-1}
\Sh{U}\left(\xi_i', \frac{z_{i+1}}{z_i}\h^{-\rk(\xi_{i-1})+\rk(\xi_{i})}\right)
\otimes 
\Sh{U}\left(\xi_i'', \frac{z_{i+1}}{z_i}\h^{-\rk(\xi_{i-1})+\rk(\xi_{i})}\right).
\]
Finally, using $r_i'=\rk(\xi_i')-\rk(\xi_{i-1}')$ we get
\begin{multline*}
\tau^*(\Sh{U}_{X'}\boxtimes \Sh{U}_{X''})=\bigotimes_{i=1}^{m-1}
\Sh{U}\left(\xi_i', \frac{z_i}{z_{i+1}}\h^{\rk(\xi'_{i-1})-\rk(\xi'_{i})}\right)
\otimes \\
\Sh{U}\left(\xi_i'', \frac{z_i}{z_{i+1}}\h^{\rk(\xi''_{i-1})-\rk(\xi''_{i})+\rk(\xi'_{i-1})+\rk(\xi'_{i+1})}\right).
\end{multline*}
Plugging in these expressions in \eqref{identity to use for composition stable envelopes} and applying the first part of Lemma \ref{Lemma properties U bundle}, we deduce that $\tau^*(\Sh{L}_{X'}\boxtimes \Sh{L}_{X''})\otimes(\Sh{L}_{A_0,F})^{-1}\equiv \Sh{O}$, which is equivalent to \eqref{second equation proof composition stable envelopes}.
\end{proof}


\section{D5 Resolutions}

\subsection{D5 resolutions for bow varieties}

Let $\D$ be a separated or co-separated brane diagram.
Let $\widetilde \D$ be the brane diagram obtained by replacing a single D5 brane $\Ab$ of local charge $\w=\w(\Ab)\geq 2$ in $\D$ by a pair of consecutive D5 branes $\Ab'$ and $\Ab''$ of local charges $\w'=\w(\Ab')\geq 1$ and $\w''=\w(\Ab'')\geq 1$ such that $\w=\w'+\w''$. We call $\widetilde \D$ a D5 resolution of the brane diagram $\D$, and the branes $\Ab'$ and $\Ab''$ resolving branes. Notice that if $\D$ is separated (resp. co-separated), then $\widetilde \D$ is also separated (resp. co-separated).

Let now $\widetilde X$ and $X$ be the bow varieties associated with $\widetilde \D$ and $\D$. We say that $\widetilde X$ is a D5 resolution of the bow variety $X$\footnote{We only use the word ``resolution'' as a metaphor. No construction in this paper is a ``resolution of singularities'' in its mathematical meaning.}. The ultimate goal of this section is to compare the stable envelopes of $\widetilde X$ and $X$. To this end, we construct a distinguished embedding $j: X\hookrightarrow \widetilde X$ and study its equivariant geometry.

Recall from Section \ref{Section: Brane diagrams, bow varieties, fixed points} that the definition of a bow variety $X$ involves a space of quiver representations
$\MM$ and a gauge group $G$. Both $\MM$ and $G$ depend on the brane diagram $\D$.
In order to define the embedding $j: X\hookrightarrow \widetilde X$, we first define a map $\MM\to \widetilde \MM$ at the level of quiver representations and study its compatibility with the actions of the groups $G$ and $\widetilde G$. Recall that these spaces of representations are defined as the direct sum of certain fundamental blocks associated with the 5-branes in the brane diagrams. This observation implies that $\MM$ and $\widetilde \MM$ only differ in those components that are associated with $\Ab$ and its resolutions $\Ab'$ and $\Ab''$.
As a consequence, we set $\MM\to \widetilde \MM$ to be the identity on most summands of the spaces of quiver representation, except on $\MM_{\Ab}\to \MM_{\Ab'}\oplus\MM_{\Ab''}$, where the map is given by the assignment
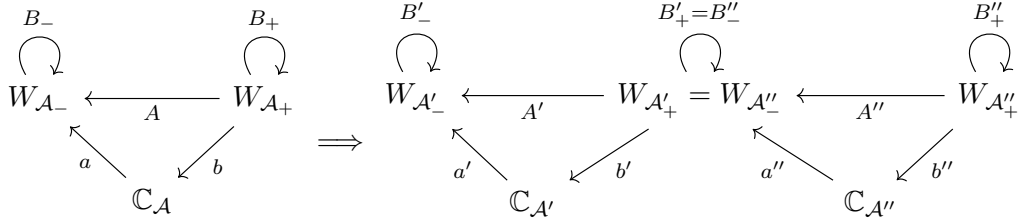
\[
\begin{tikzcd}[column sep=small, row sep=normal]
W_{\Ab_-}\arrow[loop,out=120,in=60,distance=2em, "B_-"]
& &  W_{\Ab_+} \arrow[loop,out=120,in=60,distance=2em, "B_+"] \ar[ld,"b"] \ar[ll,"A"]
\\
&\C_{\Ab} \ar[ul, "a"] &
\end{tikzcd}
\Longrightarrow
\begin{tikzcd}[column sep=small, row sep=normal]
W_{\Ab'_-}\arrow[loop,out=120,in=60,distance=2em, "B'_-"]
& &  W_{\Ab'_+}=W_{\Ab''_-} \arrow[loop,out=120,in=60,distance=2em, "B'_+=B''_-"]\arrow[ld,"b'"] \ar[ll,"A'"]
& &  W_{\Ab''_+}\arrow[loop,out=120,in=60,distance=2em, "B''_+"] \ar[ld,"b''"] \ar[ll,"A''"]
\\
&\C_{\Ab'} \arrow[ul, "a'"] & & \arrow[ul, "a''"]\C_{\Ab''}
\end{tikzcd}
\]
in which $W_{{\Ab}'_-}=W_{\Ab_-}$, $W_{{\Ab}''_+}=W_{\Ab_+}$, and   
\[
     W_{{\Ab}'_+}=W_{\Ab''_-}=
    \begin{cases}
    W_{\Ab_+}\oplus \underbrace{\C\oplus\dots\oplus \C}_{\w''} & \text{if $\D$ is separated}  \\
    W_{\Ab_-}\oplus \underbrace{\C\oplus\dots\oplus \C}_{\w'} & \text{if $\D$ is co-separated},
    \end{cases}
\]
and the linear maps are described in Table \ref{equation definition D5 resolution of bow variety}. 

\setlength{\extrarowheight}{2pt}
\begin{table}[ht]
\caption{Definition of the map $\MM\to \widetilde \MM$.}
\begin{tabular}{ |l|l| } 
 \hline
  if $\D$ is separated & if $\D$ is co-separated \\[2pt]
 \hline
 $ B'_- = B_- $
 &
 $
 B'_- = B_-
 $
 \\[2pt]
 $
 B''_+= B_+
 $
 &
 $
 B''_+= B_+
 $
 \\[2pt]
 $
 B_+' = B_-''=
 \begin{pmatrix}    B_+       & 0    & \dots &        &        \\  
                      b       & 0    & 0     & \dots  &        \\
                      0       & -1   & 0     & 0      &  \dots \\
                      \dots   &0     & -1    & \ddots &    0   \\
                              &\dots & 0     &  -1    &    0
    \end{pmatrix}
 $
 &
 $
 B_+' = B_-''=
 \begin{pmatrix} B_+       & a        & 0     & \dots  &        \\
                  0        & 0        & 1     & 0      &\dots   \\
                 \dots     & 0        & 0     & 1      &    0   \\
                           & \dots    & 0     & \ddots &    1   \\
                           &          &\dots  &   0    &    0
    \end{pmatrix}
$
\\[2pt]
$
A'= 
    \begin{pmatrix} A & -a & -(-B_-)^1a &\dots & -(-B_-)^{\w''-1} a
    \end{pmatrix} 
$
&
$
A'= 
    \begin{pmatrix} 1  & 0 &0 &\dots & 0
    \end{pmatrix}   
$
\\[2pt]
$
A''= \begin{pmatrix}
        1     &
        0     &
        0     &
        \dots &
        0
    \end{pmatrix}^T 
$
&
$
A''= \begin{pmatrix}
        A     &
        b     &
        bB_-
        \dots &
        b(B_-)^{\w'-1}
    \end{pmatrix}^T
$
\\[2pt]
$
a' = (-B_-)^{\w''} a 
$
&
$
a' = a
$
\\[2pt]
$
a''= \begin{pmatrix}
        0     &
        1     &
        0     &
        \dots&
        0
        \end{pmatrix}^T
$
&
$
a''= \begin{pmatrix}
        0     &
        0     &
        \dots &
        0     &
        1
        \end{pmatrix}^T
$
\\[2pt]
$
b'=\begin{pmatrix}
        0 & 0  &\dots & 0& -1
    \end{pmatrix}
$
&
$
b'=\begin{pmatrix}
        0 & 1  &0 & \dots& 0
    \end{pmatrix}
$
\\[2pt]
$
b''=b
$
&
$
b''=b B_+^{\w'}
$
\\[2pt]
\hline
\end{tabular}
\label{equation definition D5 resolution of bow variety}
\end{table}

\begin{example} \rm
\label{Example D5 resolution with charge c''=1}
Resolutions of separated varieties of the form $\w=\w'+\w''$ with $\w''=1$ are particularly easy to understand. In this case, the map $\mathbb{M}\to \widetilde{\mathbb{M}}$ is induced by the assignment
    
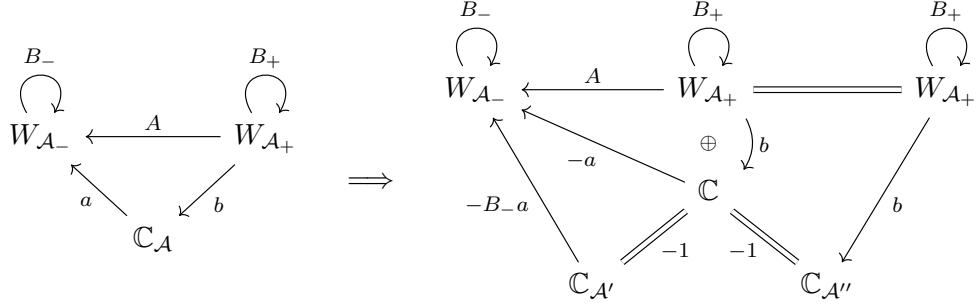
\begin{equation*}
    \label{diagram fusion separated D5 branes charge c=c'+1}
    \begin{tikzcd}[column sep=small, row sep=normal]
W_{\Ab_-}\arrow[loop,out=120,in=60,distance=2em, "B_{-}"]
& &  W_{\Ab_+} \arrow[loop,out=120,in=60,distance=2em, "B_{+}"] \ar[ld,"b"] \ar[ll, swap, "A"]
\\
&\C_{\Ab} \ar[ul, "a"] &
\end{tikzcd}
\quad \Longrightarrow \quad 
    \begin{tikzcd}[column sep=small, row sep=normal]
        W_{\Ab_-}\arrow[loop ,out=120,in=60,distance=2em,   "B_{-}"] & & W_{\Ab_+}\arrow[loop ,out=120,in=60,distance=2em,   "B_{+}"]\arrow[d, bend left, shift left= 3,    "b"] \arrow[ll, ,swap,  "A"] & & W_{\Ab_+}\arrow[loop,out=120,in=60,distance=2em, "B_{+}"] \arrow[ll, equal]\arrow[ddl, "b"] \\   
        & &    \C \arrow[ull,   "-a"]\arrow[dl, equal, "-1"]\arrow[draw=none]{u}[sloped,auto=false]{\oplus} & &    \\
        & \C_{\Ab'} \arrow[uul, "-B_{-}a"] & & \C_{\Ab''}\arrow[ul, equal, "-1"] &
    \end{tikzcd}.
    \end{equation*}
    In the right diagram, only the nonzero components of the linear maps are described. For instance, the matrix description of the map $B'_+$ acting on $W_{\Ab'_+}=W_{\Ab_+}\oplus \C$ consists of four blocks, but only the two nonzero ones, namely $B_{+}$ and $-b$ are shown in the picture. It is instructive to check that the moment map conditions on the two sides are equivalent.

    Similarly, consider the resolution of a co-separated variety of the form $\w=\w'+\w''$, but now with $\w'=1$. The map $\MM\to \widetilde{\MM}$ is induced by
    
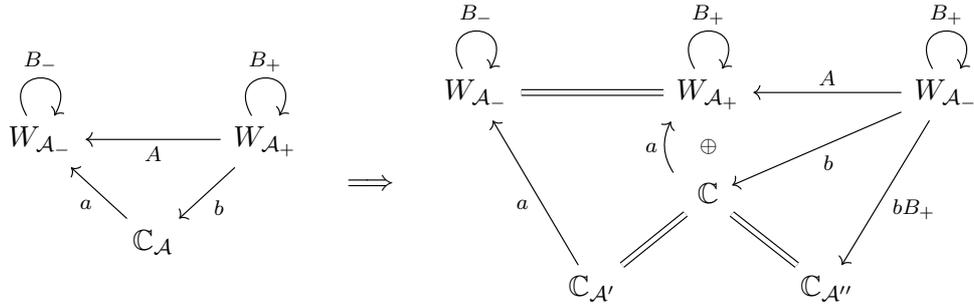
\begin{equation*}
    \label{diagram fusion co-separated D5 branes charge c=c'+1}
    \begin{tikzcd}[column sep=small, row sep=normal]
        W_{\Ab_-}\arrow[loop,out=120,in=60,distance=2em, "B_{-}"]
        & &  W_{\Ab_+} \arrow[loop,out=120,in=60,distance=2em, "B_{+}"] \ar[ld,"b"] \ar[ll,"A"]
        \\
        &\C_{\Ab} \ar[ul, "a"] &
    \end{tikzcd}
    \quad \Longrightarrow \quad 
    \begin{tikzcd}[column sep=small, row sep=normal]
        W_{\Ab_-}\arrow[loop ,out=120,in=60,distance=2em,   "B_{-}"] & & W_{\Ab_-}\arrow[loop ,out=120,in=60,distance=2em,   "B_{-}"] \arrow[ll, equal] & & W_{\Ab_+}\arrow[loop,out=120,in=60,distance=2em, "B_{+}"] \arrow[ll, swap, "A"]\arrow[dll, "b"]\arrow[ddl, "b B_{+}"] \\   
        & &  \C \arrow[u, bend left, shift left= 3,    "a"] \arrow[dl, equal]\arrow[draw=none]{u}[sloped,auto=false]{\oplus} & & \\
        & \C_{\Ab'} \arrow[uul, "a"] & & \C_{\Ab''}\arrow[ul, equal] &
    \end{tikzcd}.
    \end{equation*}
    As before, only the nonzero components of the maps are displayed.
\end{example}

\begin{remark}
\label{Hurtubise normal form}
    Assume that $\dim(W_{\Ab_-})-\dim(W_{\Ab_+})> 0$, i.e. that we are in the separated setting. As shown in \cite[Prop. 3.2]{Nakajima_Takayama}, every three-way part satisfying $B_-A-A B_++ab=0$ as well as the conditions (S1) and (S2) is isomorphic to $\GL(W_{\Ab_-})\times Z_{\Ab}$, where $Z_{\Ab}\subset \gl(W_{\Ab_-})$ is closed. This isomorphism sends a tuple $(A,B_-,B_+,a,b)$ to the so-called Hurtubise normal form $(g,\eta)$\footnote{In the notation of \cite{Nakajima_Takayama}, the matrix $g$ is $-u^{-1}$.}. In particular, the matrix $g$ is given by
    \[
    \begin{pmatrix} A & -a & -(-B_-)^1a &\dots & -(-B_-)^{\w-1} a
    \end{pmatrix},
    \]
    where $\w=|\dim(W_{\Ab_-})-\dim(W_{\Ab_+}|$.
    In this language, the assignment in Table \ref{equation definition D5 resolution of bow variety} gives a map
    \[
    \GL(W_{\Ab_-})\times Z_{\Ab} \to \GL(W_{\Ab'_-})\times Z_{\Ab'} \times \GL(W_{\Ab''_-})\times Z_{\Ab''}.
    \]
    It is easy to see that the projection of this map to the first factor is just the identity map $id: \GL(W_{\Ab_-})\to \GL(W_{\Ab'_-}=W_{\Ab_-})$ while the projection on the third factor is constant (and its image is the identity matrix in $\GL(W_{\Ab''_-})$. The co-separated case is completely analogous.
\end{remark}

Let $\Tt=\At\times \Cs_{\h}$ (resp. $\widetilde \Tt= \widetilde \At \times \Cs_{\h}$) be the torus acting on $X$ (resp. $\widetilde X$). Define $\varphi: \Tt\to \widetilde \Tt$ as the identity on most components, except 
\begin{equation}
\label{group homomorphism A-resolution}
    \Cs_{\Ab} \times \Cs_{\h}  \to  \Cs_{\Ab'} \times  \Cs_{\Ab''} \times \Cs_{\h} \qquad (a,\h)\mapsto \begin{cases}
        (a\h^{-\w''}, a, \h ) & \text{if $\D$ is separated}\\
        (a, a\h^{\w'}, \h ) & \text{if $\D$ is co-separated.}
    \end{cases}
\end{equation}
\begin{proposition}
\label{proposition embedding resolution D5 branes}
    Let $\D$ be a separated bow diagram and let $\widetilde\D$ be a D5 resolution. The map $\mathbb{M}\to \widetilde{\mathbb{M}}$ defined above descends to a regular closed embedding $j: X\hookrightarrow \widetilde X$. Moreover, the map $j$ is equivariant along $\varphi:\Tt\to \widetilde \Tt$.
\end{proposition}

\begin{remark}
\label{remark torus action D5 resolution}
    The map $\MM\to \widetilde{\MM}$ is not equivariant along $\varphi$ with respect to the usual actions of $\Tt$ on $\MM$ and of $\widetilde \Tt$ on $\widetilde{\MM}$. 
    Nevertheless, the induced map $X\to \widetilde X$ will be equivariant along $\varphi$.
    To see this, assume first that $\D$ is separated and twist the action of $\widetilde \Tt$ on $\widetilde{\MM}$ by prescribing additional weights on $W_{\Ab'_+}=W_{\Ab''_-}$ as follows:
    \[
    W_{\Ab'_+}=W_{\Ab''_-}=W_{\Ab_+}\oplus \underbrace{\C_{a}\oplus\C_{a\hbar^{-1}}\oplus\dots\oplus \C_{a\hbar^{-(\w''-1)}}}_{\w''}.
    \]
     As a result, this twisted action is obtained as the composition of an embedding $\widetilde \Tt\hookrightarrow \widetilde \Tt\times \widetilde G$ of the form $id\times f$ with the standard action of $\widetilde \Tt\times \widetilde G$ on $\widetilde \MM$.
     As a result, this twisted action is obtained as the composition of an embedding $\widetilde \Tt\hookrightarrow \widetilde \Tt\times \widetilde G$ which is the identity on the first factor with the standard action of $\widetilde \Tt\times \widetilde G$ on $\widetilde \MM$.
     It is easy to check that the map $\MM\to \widetilde{\MM}$ becomes equivariant along $\varphi$ with respect to this twisted action on $\widetilde{\MM}$. But this implies that $j: X\hookrightarrow \widetilde X$, which will be constructed from $\MM\to \widetilde{\MM}$ by taking quotients by the gauge groups $G$ and $\widetilde G$, is equivariant along $\varphi$ with respect to the usual actions. The co-separated case is similar and is left to the reader.

\end{remark}

\begin{proof}[Proof of Proposition \ref{proposition embedding resolution D5 branes}]

Firstly, notice that the map $\MM\to \widetilde{\MM}$ is clearly a regular closed embedding, equivariant along the map $G\hookrightarrow  \widetilde G$ that is the identity on most components except on $\GL(W_{\Ab+})\to GL(W_{\Ab''_-})\times \GL(W_{\Ab''_+})$, where it is given by
\[
\GL(W_{\Ab_+}) \xrightarrow[]{\Delta} \GL(W_{\Ab_+})\times \GL(W_{\Ab_+})=\GL(W_{\Ab''_+})\times \GL(W_{\Ab''_+})\subset \GL(W_{\Ab''_-})\times \GL(W_{\Ab''_+}).
\]
Notice that $G$ is the maximal subgroup of $\widetilde G$ preserving $\MM\subset \widetilde{\MM}$.
The map $\MM\to \widetilde{\MM}$ is compatible with the moment map equations on both sides. In other words, a tuple $(a,b,A,B,C,D)\in \MM$ solves the moment map equation on $\MM$ iff its image under $\MM \to \widetilde{\MM}$ solves the moment map equation on $\widetilde{\MM}$. This is trivial for almost all components of the moment map of Section~\ref{sec:def of bow variety}, except for those valued in $\NN_{\Ab'}\oplus\NN_{\Ab''}$, in which case it can be easily checked using Table \ref{equation definition D5 resolution of bow variety}. Compatibility with the moment map means that we get a pullback diagram
\[
\begin{tikzcd}
    \mu^{-1}(0)\arrow[d, hookrightarrow]\arrow[r, hookrightarrow] & \tilde\mu^{-1}(0)\arrow[d, hookrightarrow]\\
    \MM \arrow[r, hookrightarrow] & \widetilde{\MM}
\end{tikzcd}
\]
and hence that the inclusion $\mu^{-1}(0)\hookrightarrow \tilde \mu^{-1}(0)$ is also a closed embedding.

Assume temporarily that this map strictly respects (semi)stability, i.e. that we also have a pullback diagram
\[
\begin{tikzcd}
    \MM^s\arrow[d, hookrightarrow]\arrow[r, hookrightarrow] & \widetilde{\MM}^s\arrow[d, hookrightarrow]\\
    \MM\arrow[r, hookrightarrow] & \widetilde{\MM}
\end{tikzcd}.
\]
Then it follows that $\mu^{-1}(0)^s \hookrightarrow \tilde\mu^{-1}(0)^s$ is a closed embedding, which must also be regular because the (semi)stable loci are smooth. Taking quotients, we conclude that the induced map 
\[
j: X(\D)=\mu^{-1}(0)^s/ G \hookrightarrow \tilde \mu^{-1}(0)^s/ \widetilde G=X(\widetilde \D)
\]
satisfies the same properties. The fact that the induced map is still a closed embedding even though the group $ \widetilde G$ is larger than $G$ follows from Remark \ref{Hurtubise normal form}.


In conclusion, it remains to prove that the diagram above associated with the map $\MM\hookrightarrow \widetilde{\MM}$ commutes and is a pullback. We call such a property of the map $\MM\hookrightarrow \widetilde{\MM}$ property (P). We argue by induction on the weight $\w''$ of the decomposition $\w=\w'+\w''$. Assume temporarily that property (P) holds whenever $1\leq \w''\leq k$ and let $\w''=k+1$. Consider the additional decomposition $\w=\w'+\w_1''+\w_2''$ with $\w_1'', \w_2''\neq 0$. This, in particular, implies that  $\w_1'', \w_2''\leq k$. For the time being, redefine $\MM_{\w'+\w_1''+\w_2''}:= \MM$ and $\MM_{\w', \w_1''+\w_2''}:=\widetilde{\MM}$. The partition $\w=\w'+\w_1''+\w_2''$ invites to consider also the bow varieties  $\MM_{\w'+\w_1'', \w_2''}$ and $\MM_{\w',\w_1'',\w_2''}$, with obvious notation. For instance, $\MM_{\w',\w_1'',\w_2''}$ is the resolution of $\MM_{\w', \w_1''+\w_2''}$ by further splitting the brane with charge $\w_1''+\w_2''$. A direct computation shows that the diagram
\[
\begin{tikzcd}
    \MM_{\w'+\w_1''+\w_2''} \arrow[d, hookrightarrow]\arrow[r, hookrightarrow] & \MM_{\w',\w_1''+\w_2''} \arrow[d, hookrightarrow]\\
    \MM_{\w'+\w_1'', \w_2''} \arrow[r, hookrightarrow] & \MM_{\w',\w_1'',\w_2''}
\end{tikzcd}
\]
is commutative (indeed, the map $\MM_{\w'+\w''}\hookrightarrow \MM_{\w',\w''}$ for a general partition $\w= \w'+\w''$ is uniquely determined by the cases discussed in Example \ref{Example D5 resolution with charge c''=1} and the requirement that the diagram above commutes). Since by the inductive hypothesis property (P) holds for the two vertical arrows and for the bottom arrow, it holds for the top arrow as well. This proves the inductive step.
Therefore, we are left to show the base case, namely that property (P) holds for the map $\MM_{\w'+1}\hookrightarrow \MM_{\w',1}$. We need to show that the three stability conditions (S1), (S2) and GIT stability hold on the left-hand side of the diagrams in Example \ref{diagram fusion separated D5 branes charge c=c'+1} iff they hold on its right-hand side. We will show this assuming that $X$ is a separated bow variety. The co-separated case is 
analogous.

We start with (S1) and (S2). On the rightmost three-way part in the right-hand side of Example \ref{diagram fusion separated D5 branes charge c=c'+1}, these conditions are trivial, hence it suffices to show that (S1) and (S2) on the two triangles
\begin{equation*}
    \begin{tikzcd}[column sep=small, row sep=normal]
W_{\Ab_-}\arrow[loop,out=120,in=60,distance=2em, "B_{-}"]
& &  W_{\Ab_+} \arrow[loop,out=120,in=60,distance=2em, "B_{+}"] \ar[ld,"b_{}"] \ar[ll,"A_{}"]
\\
&\C_{\Ab} \ar[ul, "a_{}"] &
\end{tikzcd}
\qquad \qquad
    \begin{tikzcd}[column sep=small, row sep=normal]
        W_{\Ab_-}\arrow[loop ,out=120,in=60,distance=2em, "B_{-}"] & & W_{\Ab_+}\arrow[loop ,out=120,in=60,distance=2em, "B_{+}"]\arrow[d, bend left, shift left= 3, "b_{}"] \arrow[ll, "A_{}"] \\   
        & &    \C \arrow[ull, "-a_{}"]\arrow[dl, equal, "-1"]\arrow[draw=none]{u}[sloped,auto=false]{\oplus}    \\
        & \C_{\Ab'} \arrow[uul, "-B_{-}a_{}"]
    \end{tikzcd}
\end{equation*}
are equivalent.

Let $p_1: W_{\Ab_+}\oplus \C\to W_{\Ab_+}$ and $p_2: W_{\Ab_+}\oplus \C\to\C$ be the two projections. By definition, the right diagram satisfies (S1) iff there is no nonzero subspace $S\subset W_{\Ab_+}\oplus \C$ such that $A\circ p_1(S)=0$, $-a\circ p_2(S)=0$, $-id \circ p_2(S)=0$, $B_+(S)\subset p_1(S)$ and $ b\circ p_1(S)\subset p_2(S)$. Condition $id \circ p_2(S)=0$ implies that $p_1(S)=S$ must be contained in $W_{\Ab_+}\subset W_{\Ab_+}\oplus \C$. Hence (S1) on the right-hand side holds iff there exists no nonzero subspace $S\subset  W_{\Ab_+}$ such that $A(S)=0$, $B_+(S)\subset S$, and $b(S)=0$. But this is exactly condition (S1) on the left three-way part.

Verifying condition (S2) is analogous. Namely the three-way part on the right satisfies (S2) iff there exist no proper subspace $T\subset W_{\Ab_-}$ such that $\Image(A) +\Image(-a)+\Image(-B_-\circ a) \subset T$ and $B_-(T)\subset T$. But $B_-(T)\subset T$ implies that $\Image(-B_-\circ a)\subset T$, hence the former is equivalent to $\Image(A) +\Image(a) \subset T$ and $B_-(T)\subset T$, which is nothing but (S1) for the left three-way part. 

Finally, we check GIT stability.
By \cite[Section 2.4.2]{Nakajima_Takayama}, a tuple $(\tilde a,\tilde b, \tilde A,$ $\tilde B, \tilde C, \tilde \D) \in \MM_{c',1}$ is stable iff there exists no proper graded subspace $\tilde S\subset W= \oplus_{\tilde \Xb} W_{\tilde \Xb} $ invariant under $\tilde A$, $\tilde B$, $\tilde C$, $\tilde \D$ such that $\Image(\tilde a) \subset \tilde S$ and $A$ induces isomorphisms $W_{\Ab+}/\tilde S_{ \Ab+}\to W_{ \Ab-}/\tilde S_{ \Ab-}$. Consider the image under $\MM_{\w'+1}\hookrightarrow \MM_{\w',1}$ of a tuple $( a, b, A, B, C, D) \in \MM_{\w'+1}$. As shown in Example \ref{Example D5 resolution with charge c''=1}, the map $\tilde A:W_{\Ab+}\to W_{\Ab+}\oplus \C$ is the identity on the first component and the zero map on the second one; hence $\tilde A$ induces an isomorphism on the quotients by $\tilde S$ iff its component inside $W_{\Ab+}\oplus \C$ is of the form $S_{\Ab+}\oplus \C$, with $S_{\Ab+}\subseteq W_{\Ab+}$. As a consequence, the component of $\tilde A$ from $W_{\Ab+}\oplus \C$ to $W_{\Ab-}$ is of the form $(A, a)$ induces an isomorphism in the quotient iff $\Image(a)\subset S_{\Ab-}$ and $A$ induces an isomorphism $W_{\Ab+}/ S_{\Ab+}\to W_{\Ab-}/S_{\Ab-}$. But these conditions are exactly the stability conditions on $( a, b, A, B, C, D) \in \MM_{\w'+1}$, proving equivalence.

The equivariance of $j: X\to \widetilde X$ along $\varphi: \Tt\to \widetilde \Tt$ was already discussed in Remark \ref{remark torus action D5 resolution}
\end{proof}

Recall the morphisms \eqref{maps from bow to affine bow and handsaw}.
The proof above is based on the observation that $\MM\to \widetilde \MM$ is compatible with all the conditions that we impose on these two spaces to construct the corresponding bow varieties, namely $\mu=0$, the open conditions (S1) and (S2) and GIT stability. Forgetting the latter, we deduce the following corollary.
\begin{corollary}
\label{corollary enhancement D5 resolution with affine quotients}
There exists a commutative diagram
\[
\begin{tikzcd}
    X\arrow[r, hookrightarrow, "j"]\arrow[d, "\pi"] & \widetilde X \arrow[d, "\widetilde \pi "]\\
    X_0\arrow[r, hookrightarrow]\arrow[d, "\rho"] & \widetilde X_0 \arrow[d, "\tilde 
    \rho"]\\
    HS\arrow[r, hookrightarrow] & \widetilde{HS}
\end{tikzcd}
\]
in which all horizontal maps are closed immersions. Moreover, both squares are Cartesian.
\end{corollary}

\begin{corollary}
\label{corollary any bow can be embedded in a flag variety} 
Any bow variety $X$ can be embedded in the cotangent bundle of a partial flag variety. More precisely, if $X(\D)$ is separated or co-separated and has no D5 branes with local charge equal to zero, then it can be embedded in the bow variety $X(\widetilde \D)$ whose brane diagram $\widetilde \D$ is obtained from $\D$ by replacing each D5 brane $\Ab$ with $\w(\Ab)$ consecutive D5 branes with local charge equal to one.


\end{corollary}
\begin{proof}
According to Section~\ref{Sec:0charge} we may assume that the charges of all 5-branes in $\mathcal D$ are positive. Then the statements follow from Proposition~\ref{proposition embedding resolution D5 branes} and from the fact if all D5 branes have charge one, then the bow variety is isomorphic to the cotangent bundle of a partial flag variety, cf. Section \ref{sec:quivers}.
\end{proof}

\begin{example} \rm
    The following diagrams describe a separated bow variety on the left, and its maximal resolution, on the right. Each D5 brane with local charge $\w_i>0$ is replaced by $\w_i$ D5 branes with local charges all equal to one. The resulting variety is isomorphic, via Hanany-Witten, to the contangent bundle of the partial flag variety $T^*\Fl_\lambda$, where $\lambda=(\w_1, \w_1+\w_2, \w_1+\w_2+\w_3,\w_1+\w_2+\w_3+\w_4)$.
     
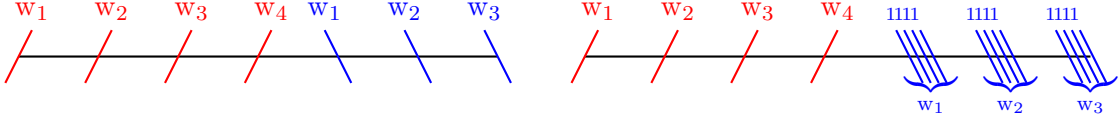
\[
\begin{tikzpicture}[baseline=0pt,scale=.35]
\draw[thick] (4,1)--(22,1) ; 
\draw[thick,red] (3.5,0)  -- (4.5,2) node[above]{$\w_1$};
\draw[thick,red] (6.5,0)  -- (7.5,2) node[above]{$\w_2$};
\draw[thick,red] (9.5,0) -- (10.5,2)  node[above]{$\w_3$};
\draw[thick,red] (12.5,0) -- (13.5,2) node[above]{$\w_4$};
\draw[thick,blue] (16.5,0) -- (15.5,2) node[above]{$\w_1$};
\draw[thick,blue] (19.5,0) -- (18.5,2) node[above]{$\w_2$};
\draw[thick,blue] (22.5,0) -- (21.5,2) node[above]{$\w_3$};
\end{tikzpicture}
\qquad  
\begin{tikzpicture}[baseline=0pt,scale=.35]
\draw[thick] (4,1)--(23,1) ; 
\draw[thick,red] (3.5,0)  -- (4.5,2) node[above]{$\w_1$};
\draw[thick,red] (6.5,0)  -- (7.5,2) node[above]{$\w_2$};
\draw[thick,red] (9.5,0) -- (10.5,2)  node[above]{$\w_3$};
\draw[thick,red] (12.5,0) -- (13.5,2) node[above]{$\w_4$};
\draw[thick,blue] (16.7,0) --(15.7,2);
\draw[thick,blue] (17.0,0) -- (16.0,2);
\draw[thick,blue] (17.3,0) --(16.3,2);
\draw[thick,blue] (17.6,0) --(16.6,2); 
\draw[thick,blue] (19.7,0) --(18.7,2);
\draw[thick,blue] (20.0,0) -- (19.0,2);
\draw[thick,blue] (20.3,0) --(19.3,2);
\draw[thick,blue] (20.6,0) --(19.6,2);
\draw[thick,blue] (22.7,0) --(21.7,2);
\draw[thick,blue] (23.0,0) -- (22.0,2);
\draw[thick,blue] (23.3,0) --(22.3,2);
\draw[thick,blue] (23.6,0) --(22.6,2);
\node[blue] at (17.0,-0.45) {$\underbrace{}_{\w_1}$};
\node[blue] at (20.0,-0.45) {$\underbrace{}_{\w_2}$};
\node[blue] at (23.0,-0.45) {$\underbrace{}_{\w_3}$};
\node[blue] at (15.5,2.6) {\tiny $1$}; \node[blue] at (15.8,2.6) {\tiny $1$}; \node[blue] at (16.1,2.6) {\tiny $1$}; \node[blue] at (16.4,2.6) {\tiny $1$};
\node[blue] at (18.5,2.6) {\tiny $1$}; \node[blue] at (18.8,2.6) {\tiny $1$}; \node[blue] at (19.1,2.6) {\tiny $1$}; \node[blue] at (19.4,2.6) {\tiny $1$};
\node[blue] at (21.5,2.6) {\tiny $1$}; \node[blue] at (21.8,2.6) {\tiny $1$}; \node[blue] at (22.1,2.6) {\tiny $1$}; \node[blue] at (22.4,2.6) {\tiny $1$};
\end{tikzpicture}
    \]

\end{example}

\subsection{The equivariant geometry of D5 resolutions}

We begin this section by showing that the map $j: X\hookrightarrow \widetilde X$ is compatible with the description of the fixed locus from Section~\ref{section Actions of subtori}.

Let $X$ be a separated or co-separated bow variety, and let $\widetilde X$ be a resolution of the brane $\Ab$. Set $\w=\w(\Ab)$, $\w'=\w(\Ab')$ and $\w''=\w(\Ab'')$ (hence $\w=\w'+\w''$). The torus $\At \subset \Tt$ acts, via $\varphi: \Tt\to \widetilde \Tt$, on the resolution $\widetilde X$. 

Consider a rank two torus $A=\Cs_{a'}\times \Cs_{a''}\subset \A$ acting with $a'$ on a subset of the D5 branes in $\D$ and with $a''$ on the remaining D5 branes. Without loss of generality, we assume that $\Ab$ is acted on by $a'$.
Recall that the $A$-fixed components of $X$ are of the form $X'\times X''$ for certain bow varieties $X'$ and $X''$ described in Theorem \ref{proposition tensor product decomposition fixed locus}. In particular, $X'$ contains the D5 branes acted on by $a'$ and $X''$ the remaining ones. Moreover, the local charge of $\Ab$ in $X'$ is still $\w$.

Let $j': X'\hookrightarrow \widetilde X'$ be the resolution of the brane $\Ab$ in $X'$ with decomposition $\w=\w'+\w''$. We have

\begin{lemma}
\label{lemma: D5 resolutions fixed loci}
    The unique $A$-fixed component in $\widetilde X$ containing $X'\times X'' $ is $\widetilde{X'}\times X''$ and we have a pullback diagram
    \[
    \begin{tikzcd}
        X'\times X'' \arrow[d, hookrightarrow]\arrow[r, hookrightarrow, "j'\times id"] &\widetilde{X'}\times X''\arrow[d, hookrightarrow]\\
        X\arrow[r, hookrightarrow, "j"] & \widetilde X
    \end{tikzcd}.
    \]
    In addition, all the maps are $\Tt$-equivariant. 
\end{lemma}
\begin{proof}
    All the morphisms in the diagram descend from maps at the level of quiver representation. All these maps are explicitly defined there, so commutativity and being a pullback are straightforward checks. But then the same must hold at the level of bow varieties. 
\end{proof}


\begin{lemma} Let $F$ be an  $\At$-fixed component of $\widetilde X$ such that $F\cap X\neq \emptyset$. 
\label{lemma Fixed components D5 res}
\begin{enumerate}
    \item If $X$ is separated (resp. co-separated), then $F$ is equivariantly isomorphic to the bow variety 
    \ttt{{\fs}1\fs2{\fs}3\fs\dots {\fs}$\w'+\w'' ${\bs}$\w''${\bs}} (respectively,  \ttt{{\bs}$\w'${\bs}$\w'+\w'' ${\fs}\dots{\fs}3{\fs}2{\fs}1{\fs}}). 
    \item Under this isomorphism, the residual action of $\widetilde{\At}$ on $F$ is identified with the canonical action of the torus $(\Cs)^{2}$ on \ttt{{\fs}1\fs2{\fs}3\fs\dots {\fs}$\w'+\w'' ${\bs}$\w''${\bs}} (respectively, \ttt{{\bs}$\w'${\bs}$\w'+\w'' ${\fs}\dots{\fs}3{\fs}2{\fs}1{\fs}}).
    \item The intersection $F\cap X$ is a singleton $f\in X^{\At}$. 
    \item Each fixed point $f=F\cap X$ admits exactly $\binom{\w'+\w''}{\w'}$ resolutions, i.e. fixed points $\tilde f\in F^{\widetilde \At}$.
\end{enumerate}
\end{lemma}
Note that point (1) implies that all fixed components $F$ satisfying $F\cap X\neq \emptyset$ are isomorphic to the mirror dual of the cotangent bundle of a Grassmannian. We will exploit this fact when proving our main theorem: mirror symmetry of stable envelopes.
\begin{proof}
    We prove the result in the separated case. Let $m$ be the number of D5 branes in $X$ and let $\Ab_k$ be the brane resolved in $\widetilde X$ for some given $k\in \{1, \dots, m\}$. Notice that $m=\rk(\At)$. Applying Theorem \ref{proposition tensor product decomposition fixed locus} $m$ times, it follows that any fixed point in $X$ is the product $X_1\times \dots \times X_m$ of zero dimensional bow varieties, cf. Example \ref{Example fixed point-fixed component interaction}. By iteration of Lemma \ref{lemma: D5 resolutions fixed loci}, it follows that 
    \[
    F=X_1\times\dots \times  X_{k-1}\times \widetilde X_k\times X_{k+1}\dots \times X_m,
    \] 
    where $\widetilde  X_k$ is of the form \ttt{{\fs}$d_1$\fs $d_2${\fs}$d_3$\fs\dots {\fs}$\w'+\w'' ${\bs}$\w''${\bs}}. In particular, $X_k$ and $\widetilde X_k$ have the same NS5 brane arrangement. Moreover, in a tie diagram, a D5 brane can be connected to the same NS brane at most once. Hence, it follows that $X_k$ is of the form \ttt{{\fs}$d_1$\fs $d_2${\fs}$d_3$\fs\dots {\fs}$\w'+\w'' ${\bs}} where $d_1\geq 1$ and $d_{i+1}-d_i\in \{0,1\}$. Therefore, 
    \begin{equation}
        \label{eq: fixed compD5 res}
            F\cong \widetilde X_k\cong \ttt{{\fs}$1$\fs\dots $1$\fs $2$\fs\dots $2${\fs}$3$\fs\dots $3$\fs \dots \fs$\w'+\w'' ${\bs}$\w''${\bs}}.
    \end{equation}
    By Section \ref{section: Getting rid of the ``trivial'' branes}, this is isomorphic to \ttt{{\fs}1\fs2{\fs}3\fs\dots {\fs}$\w'+\w'' ${\bs}$\w''${\bs}}. The proof of point (1) follows. Point (2) and (3) are also consequences of our analysis. Finally, point (4) follows from the tie diagram description of the fixed points of $F$: a tie diagram representing a fixed point in $F^{\widetilde \At}$ is characterized by the choice of $\w'$ out of the $\w'+\w''$ NS5 branes in \eqref{eq: fixed compD5 res} to connect to the leftmost D5 brane. 
\end{proof}

\begin{remark}
\label{remark interpretation fixed points and fixed components A resolution}
   Since $X^{\At}$ is finite, Lemma \ref{lemma Fixed components D5 res} defines a one-to-one correspondence between the fixed points $f\in X^A$ and the components $F\in \widetilde X^A$ containing them. This correspondence can be nicely interpreted via tie diagrams. The residual action of $\widetilde \At/\At$ on $F$ admits exactly $\binom{\w'+\w''}{\w'}$ fixed points, the resolutions of $f$. Pick an arbitrary resolution $\tilde f\in F^{\widetilde A}$ and draw its tie diagram. Then the tie diagram of $f$ can be obtained by merging the resolving branes while keeping all the ties attached to them. Here is an example:
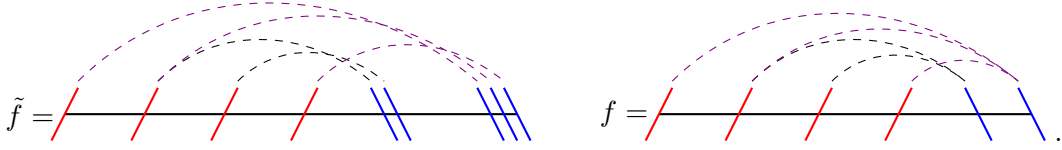
\[
\begin{tikzpicture}[baseline=0pt,scale=.35]
\draw[thick] (0,1) node [left]{$\tilde f=$}--(17,1) ;
\draw[thick,red] (-.5,0)--(.5,2);
\draw[thick,red] (2.5,0)--(3.5,2);
\draw[thick,red] (5.5,0)--(6.5,2);
\draw[thick,red] (8.5,0)--(9.5,2);
\draw[thick,blue] (12.5,0) --(11.5,2);
\draw[thick,blue] (13,0) --(12,2);
\draw[thick,blue] (16.5,0) --(15.5,2);
\draw[thick,blue] (17,0) --(16,2);
\draw[thick,blue] (17.5,0) --(16.5,2);
\draw [dashed, black](6.5,2.25) to [out=45,in=-220] (12,2.25);
\draw [dashed, black](3.5,2.25) to [out=45,in=-220] (11.5,2.25);
\draw [dashed, violet](.5,2.25) to [out=45,in=-220] (15.5,2.25);
\draw [dashed, violet](3.5,2.25) to [out=45,in=-220] (16,2.25);
\draw [dashed, violet](9.5,2.25) to [out=45,in=-220] (16.5,2.25);
\end{tikzpicture}
\qquad
\begin{tikzpicture}[baseline=0pt,scale=.35]
\draw[thick] (0,1) node [left]{$f=$}--(14,1) ;
\draw[thick,red] (-.5,0)--(.5,2);
\draw[thick,red] (2.5,0)--(3.5,2);
\draw[thick,red] (5.5,0)--(6.5,2);
\draw[thick,red] (8.5,0)--(9.5,2);
\draw[thick,blue] (12.5,0) node [below]{}--(11.5,2);
\draw[thick,blue] (14.5,0) node [below]{}--(13.5,2);
\draw [dashed, black](6.5,2.25) to [out=45,in=-220] (11.5,2.25);
\draw [dashed, black](3.5,2.25) to [out=45,in=-220] (11.5,2.25);
\draw [dashed, violet](.5,2.25) to [out=45,in=-220] (13.5,2.25);
\draw [dashed, violet](3.5,2.25) to [out=45,in=-220] (13.5,2.25);
\draw [dashed, violet](9.5,2.25) to [out=45,in=-220] (13.5,2.25);
\end{tikzpicture}.
\]
To conclude the remark, we stress that none of the $\widetilde \At/\At$ fixed points in $F$ coincide with $f=F\cap X$.
\end{remark}

We now study the interplay between the embedding $j:X\hookrightarrow \widetilde X$ and the geometry of the equivariant curves in $X$ and $\widetilde X$. Any equivariant embedding $X\hookrightarrow Y$ constrains the full attracting set of $X$ by the one of its ambient space $Y$. However, in general, the full attracting set of $X$ is smaller than the intersection of the full attracting set of $Y$ with $X$. In the next lemma, we argue that for the resolutions $X\hookrightarrow \widetilde X$ we actually have equality.

\begin{lemma} $ $
\label{lemma equality attracting sets bow varieties}
\begin{enumerate}
    \item Let $f\in X^{\At}$ and let $F\subseteq {\widetilde{X}}^{\At}$ be the unique component containing $f$. Then $\Attfull{C}^{\widetilde{X}}(F)\cap (F\times X)=\Attfull{C}^{X} (f)$.
    \item Let instead $G\subseteq \widetilde X^{\At}$ be a fixed component such that $G\cap X=\emptyset$. Then $\Attfull{C}^{\widetilde X}(G)\cap (G\times X)=\emptyset$.
\end{enumerate}
    
\end{lemma}
\begin{proof}
Let $\sigma : \Cs\to A$ be a generic subgroup in $\mathfrak{C}$. The full attracting set $\Attfull{C}^{\widetilde{X}}(F)\subset F\times \widetilde X$ consists of pairs $(y,x)$ of points belonging to a chain of closures of attracting $\sigma$-orbits. Thus, we have 
\[
\Attfull{C}^{\widetilde{X}}(F)\cap (F\times X)=\bigcup_{\substack{f'\in X^{\At} \\ (y,f')\in \Attfull{C}^{\widetilde{X}}(F)}} y  \times \Att{C}^X(f').
\]
By definition, $(y,f')\in \Attfull{C}^{\widetilde{X}}(F) $ iff there exists a chain of closures of attracting $\sigma$-orbits in $\widetilde X$ with endpoints $y$ and $f'$. Notice that each of these orbit closures is the image of a morphism $\Pe^1\to \widetilde X$. 
On the other hand, we have 
\[
\Attfull{C}^{X}(f)=\bigcup_{\substack{f'\in X^{\At} \\ (f,f')\in \Attfull{C}^{X}(f)}} f\times \Att{C}^X(f'),
\]
where now $ (f,f')\in\Attfull{C}^{X}(f)$ iff there exists a chain of closures of attracting $\sigma$-orbits in $X$ with endpoints $f$ and $f'$. 
Therefore, it suffices to show that if a chain of orbit closures in $\widetilde X$ connects $y\in F\subseteq \widetilde X$ to $f'\in X^{\At}$, all the orbits actually lie in $X$ (this in particular forces $y=f$ because by Lemma \ref{lemma Fixed components D5 res} we have $f=F\cap X$). Assume the contrary; then there exists at least one orbit closure $\Pe^1\to \widetilde X$ within such a chain that intersects both $X$ and $\widetilde X\setminus X$. But then Corollary \ref{corollary enhancement D5 resolution with affine quotients} implies that the composition $\Pe^1\to \widetilde X\xrightarrow{\tilde \pi} \widetilde X_0$ is non-constant. Since $\widetilde X_0$ is affine, this gives a contradiction. The proof of (2) is analogous.
\end{proof}

\subsection{D5 fusion of stable envelopes} Let $X$ be separated or co-separated.
In this section, we investigate the relation between the stable envelopes of the bow variety $X$ and those of a resolution $\widetilde X$.

Let $j: X\hookrightarrow \widetilde X$ be the embedding introduced in the previous section.  It is equivariant along the morphism $\varphi: \Tt\to \widetilde \Tt$ defined in equation \eqref{group homomorphism A-resolution}. 
Although the unique $\At$-fixed point $f =X\cap F$ contained in the $\At$-fixed component $F\subset \widetilde X$ does not coincide with any of the $\widetilde \At/\At$-fixed points in $F$, we will now  argue that $f$ admits a distinguished resolution $\tilde f_\sharp\in F^{\widetilde A}$.

As discussed in Remark \ref{remark interpretation fixed points and fixed components A resolution}, the tie diagrams of the $\widetilde \At/\At$-fixed points in $F$ are obtained from the tie diagram of $f$ by distributing the $\w$ ties connected to $\Ab$ between the two resolving branes $\Ab'$ and $\Ab''$ according to the decomposition $\w=\w'+\w''$. This operation produces all the $\binom{\w'+\w''}{\w'}$ fixed points in $F$. The number of tie-crossings in these diagrams ranges from zero to $\w'\w''$. 

If $X$ is separated, we denote by $\tilde f_\sharp\in F^{\widetilde A}$ the unique fixed point whose tie diagram has no crossings. If 
$X$ is co-separated, we denote by $\tilde f_\sharp\in F^{\widetilde A}$ the unique fixed point whose tie diagram has $\w'\w''$ crossings, as illustrated below.

 \begin{equation}
     \label{diagrams distinguished resolving fixed points D5 case}
\begin{tikzpicture}[scale=.35]
\draw [thick] (0,1)node[left]{$f=$}  --(7,1);
\draw [thick, blue] (4,0) node [right]{$a$} -- (3,2);
\draw [dashed](3,2) to [out=120,in=0] (0,4.6) ;
\draw [dashed](3,2) to [out=120,in=0] (0,4.8) ;
\draw [dashed](3,2) to [out=120,in=0] (0,5) node [left]{$\w'+\w'' \Big\lbrace$};
\draw [dashed](3,2) to [out=120,in=0] (0,5.2);
\draw [dashed](3,2) to [out=120,in=0] (0,5.4);
\draw[ultra thick, <->] (9.5,1)--(12.5,1) node[above]{separated} node[below]{D5 res.} -- (15.5,1);
\begin{scope}[xshift=22cm]
\draw [thick] (0,1) node[left]{$\tilde{f}_\sharp=$} --(8,1);  
\draw [thick, blue] (3,0) node [right]{$a'$} -- (2,2);
\draw[thick, blue] (6,0) node [right]{$a''$} --(5,2);
\draw [dashed](2,2) to [out=120,in=0] (0,3) ;
\draw [dashed](2,2) to [out=120,in=0] (0,3.2) node [left] {$\w' \lbrace$} ;
\draw [dashed](2,2) to [out=120,in=0] (0,3.4) ;
\draw [dashed](5,2) to [out=120,in=0] (0,5.6) node [left]{$\w'' \lbrace$};
\draw [dashed](5,2) to [out=120,in=0] (0,5.8);
\end{scope}

\draw [thick] (0,-4) node[left]{$f=$} --(7,-4);
\draw [thick, blue] (4,-5) node [left]{$a$} -- (3,-3);
\draw [dashed](4,0-5) to [out=-60,in=180] (7,-7.6) ;
\draw [dashed](4,-5) to [out=-60,in=180] (7,-7.8) ;
\draw [dashed](4,-5) to [out=-60,in=180] (7,-8) node [right]{$\Big\rbrace \w'+\w''$};
\draw [dashed](4,-5) to [out=-60,in=180] (7,-8.2);
\draw [dashed](4,-5) to [out=-60,in=180] (7,-8.4);
\draw[ultra thick, <->] (9.5,-4)--(12.5,-4) node[above]{co-separated} node[below]{D5 res.} -- (15.5,-4);

\begin{scope}[xshift=22cm]
\draw [thick] (0,-4) node[left]{$\tilde{f}_\sharp=$} --(8,-4);  
\draw [thick, blue] (3,-5) node [left]{$a'$} -- (2,-3);
\draw[thick, blue] (6,-5) node [left]{$a''$} --(5,-3);
\draw [dashed](3,-5) to [out=-60,in=180] (9,-7) ;
\draw [dashed](3,-5) to [out=-60,in=180] (9,-7.2) node [right] {$\rbrace \w'$} ;
\draw [dashed](3,-5) to [out=-60,in=180] (9,-7.4) ;
\draw [dashed](6,-5) to [out=-60,in=180] (9,-9.6) node [right]{$\rbrace \w''$};
\draw [dashed](6,-5) to [out=-60,in=180] (9,-9.8);
\end{scope}
\end{tikzpicture}
 \end{equation}
The rationale for this choice of representative lies in the following lemma. 

\begin{lemma}
\label{restrction Chern roots in A resolution}
    Let $\mathcal \xi$ be a tautological bundle of $\widetilde X$. Then $j^* \xi\Big|_f= \xi\Big|_{\tilde f_\sharp}$ as representations of $\Tt$. As a consequence, the diagram
    \[
    \begin{tikzcd}
        K_{ \widetilde \Tt}(\widetilde X)^{taut}\arrow[d, "j^*\circ \varphi^*"] \arrow[r] &K_{ \widetilde \Tt}(\tilde f_\sharp)\arrow[d, "\varphi^*"]\\
        K_{\Tt}(X)^{taut} \arrow[r] &K_{\Tt}(f)\\
    \end{tikzcd}
    \]
    associated with fixed point localization at $\tilde f_\sharp$ and $f$ and the change of group map $\varphi:  \Tt\to \widetilde \Tt$ is commutative. 
\end{lemma}
\begin{proof}
    The proof follows from the combinatorics of the fixed point restrictions developed in \cite[Section 4.4]{rimanyi2020bow}.
\end{proof}

By construction, the K-theory class $N_j$ of the normal bundle of the embedding $j: X\hookrightarrow \widetilde X$ can be expressed uniquely in terms of the tautological bundles on the D5 part of the separated or co-separated brane diagram. By Proposition \ref{prop:rightbundles}, it follows that $N_j$ is topologically trivial and hence decomposes in attractive, repelling, and fixed directions:
\[
N_j=N_j^+ + N_j^0 + N_j^-\in K_{\Tt}(\pt).
\]
With this observation, we are ready to state the main result of this section.
\begin{theorem}
\label{Fusion of D5 branes for separated brane diagrams}
Fix some $f\in X^{\At}$ and a chamber $\mathfrak{C}$ for the $\At$-action on $X$. Let $F$ be the unique $\At$-fixed component of $\widetilde X$ containing $f$ and $\widetilde{\mathfrak{C}}$ be any chamber for $\widetilde \At$ restricting to $\mathfrak{C}$ on $\At$. Then 
\[
\varphi^{\oast}\left(\StabX{F}{\widetilde C/C}{\tilde f}{(z\h^{\pm \gamma(f)})}\Big|_{\tilde f_\sharp}\right)\Theta(N^-_j)  \StabX{X}{C}{f}= j^{\oast} \varphi^{\oast}\StabX{\widetilde X}{\widetilde C}{\tilde f} \qquad \forall \tilde f\in F^{\widetilde \At}.
\]
Here, $\gamma(f)$ is a multi-index whose entries $\gamma(f)_i$ denote the number of ties of $f$ connecting the $i$-th NS5 brane to the D5 branes left to the resolved D5 brane. The sign of the shift by $\gamma(f)$ is negative in the separated case and positive in the co-separated one.

\begin{corollary}
\label{corollary nicest one-term D5 resolution}
Assuming that $\widetilde{\mathfrak{C}}/\mathfrak{C}=\lbrace a'>a''\rbrace$ and choosing $\tilde f=\tilde f_\sharp$, we get 
    \begin{equation}
    \label{nicest one-term D5 resolution}
    \prod_{s=1}^{\w''} \vartheta(\h^s)\Theta(N^-_j)  \StabX{X}{C}{f}= j^{\oast}\varphi^{\oast}\StabX{\widetilde X}{\widetilde C}{\tilde f_\sharp}.
\end{equation}
\end{corollary}

\end{theorem}

\begin{proof}[Proof of Theorem \ref{Fusion of D5 branes for separated brane diagrams}]

Consider the action on $\widetilde X$ of the torus $\At \subset \Tt$ induced by the map $\varphi:\Tt\to \widetilde \Tt$. 
We divide the proof into two steps. In the first one, we compare $\StabX{X}{C}{f}$ to $\StabX{\widetilde X}{C}{F}$. These are, respectively, sections of certain line bundles on the elliptic cohomologies of $ f \times X$ and $F\times \widetilde X$.
Consider the pair of maps 
\[
\begin{tikzcd}
    f\times X\arrow[r, hookrightarrow, "i"]& F\times X\arrow[r, hookrightarrow, "j"] &F\times \widetilde X
\end{tikzcd}
\]
and notice that the inclusion $i$ is closed, hence proper.
We claim that 
\begin{equation}
\label{fundamental equality proof A fusion of stable envelopes}
    \vartheta(N^-_j)i_{\oast}\left( \StabX{X}{C}{f}\right)=\left(\StabX{\widetilde X}{C}{F}\right)\Big|_{F\times X}.
\end{equation}
Firstly, we show that both sides are sections of the same line bundle. Recall that $\StabX{X}{C}{f}$ is a section of the line bundle $\mathcal{L}^\triangledown_{\At,f}\boxtimes \mathcal{L}_{X}$ on $E_{\Tt}(f\times X )$ and $\StabX{\widetilde X}{C}{F}$ is a section of the line bundle 
$\mathcal{L}^\triangledown_{ \At, F}\boxtimes \mathcal{L}_{\widetilde X}$  on $E_{\Tt}(F\times \widetilde X )$---see Section \ref{section: Stable envelopes of bow varieties: definition}.
Since the D5 resolution only affects the D5 part of the bow diagram and all the tautological bundles associated with the latter are trivial, it follows that $\mathcal{L}_{\widetilde X}|_{X}=\Sh{L}_{X}\otimes \Sh{G}$, where $\Sh{G}$ is pulled back from $\Pic{}{E_\T(\pt)}$. Hence
\[
i^*\mathcal{L}^\triangledown_{\At,F}=\mathcal{L}^\triangledown_{\At, f}\otimes \Theta(TF)\otimes \Theta(N^-_j)\otimes \Sh{G}^{-1}.
\]
Thus, the composition
\[
\begin{tikzcd}
   i_*\mathcal{L}^\triangledown_{\At, f}\otimes \Sh{G}^{-1}=i_*(i^*\Sh{L}^\triangledown_{\At, F}\otimes \Theta(-TF)\otimes\Theta(-N_j^-))\arrow[r, "i_{\oast}"] 
   & \Theta(-N_j^-)\otimes  i^* \mathcal{L}^\triangledown_{ \At,F} \arrow[r, "\cdot \vartheta(N^-_j)"] 
   & \mathcal{L}^\triangledown_{ \At,F}
\end{tikzcd} 
\]
is well defined and sends $\StabX{X}{C}{f}$ to a section of $(\mathcal{L}^\triangledown_{\At,F}\boxtimes \mathcal{L}_{\widetilde X})|_{F\times X}$, as required by \eqref{fundamental equality proof A fusion of stable envelopes}.

Let now $U:=F\times X^{\geq f}$, where  
\[
X^{\geq f}=X\setminus \bigcup_{g< f} \Att{C}(g).
\]
 
In order to prove $\eqref{fundamental equality proof A fusion of stable envelopes}$, it suffices to show that both sides of the equation are supported on $\Attfull{C}(f)\subset F\times X$ and have the same restriction on $U$. Indeed, the claim then follows from Aganagic-Okounkov's argument for the uniqueness of elliptic stable envelopes \cite[Section 3.5]{aganagic2016elliptic}, which also applies in this situation without modifications. 

We first check the support condition. Since $\StabX{X}{C}{f}$ is supported on $\Attfull{C}^{X}(f)\subset f\times X$, the pushforward $\vartheta(N^-_j) i_{\oast}\StabX{X}{C}{f}$ is supported on $\Attfull{C}(f)\subset F\times X$. On the other hand, $\StabX{\widetilde X}{C}{F}$ is supported on $\Attfull{C}^{\widetilde X}(F)\subset F\times \widetilde X$, hence 
the restriction $\StabX{\widetilde X}{C}{F}|_{F\times X} $ is supported on $\Attfull{C}^{ \widetilde X}(F)\cap (F\times X)$. But part (1) of Lemma \eqref{lemma equality attracting sets bow varieties} implies that the latter is also equal to $\Attfull{C}^{X}(f)$.


We now compare the restrictions to $U$. Since we now know that both sides of \eqref{fundamental equality proof A fusion of stable envelopes} are supported on 
\[
\Attfull{C}^{X}(f)=f\times \bigcup_{g\leq f} \Att{C}(g)\subset F\times X,
\]
it follows that their restrictions to $U$ are supported on $f\times \Att{C}(f)$, which is closed in $U$. Hence, the long exact sequence in elliptic cohomology implies that the restrictions to $U$ of the two sides of \eqref{fundamental equality proof A fusion of stable envelopes} are in the image of pushforward associated with the inclusion 
\[
f\times \Att{C}(f)\hookrightarrow U,
\]
and thus they are both multiples of the fundamental class $[\Att{C}(f)]$. Therefore, to conclude that they are the same class, it suffices to check that they restrict in the same way on $f\times f$. But this follows from Remark \ref{Remark axiomatic stab} together with the fact that 
$N_{f/\widetilde X}^-=N_{f/ X}^-\oplus N_j^-$.


This observation concludes the first step. In the second step, we deduce the statement of the theorem. We will make use of the following three commutative diagrams
\[
    \begin{tikzcd}
        \tilde f \times F\times  X \arrow[d, " p_{13}"] \arrow[r, hookrightarrow] & \tilde f \times F\times  \widetilde X \arrow[d, "\tilde p_{13}"]\\
        \tilde f\times  X  \arrow[r, hookrightarrow]  & \tilde f\times \widetilde X
    \end{tikzcd}
    \qquad 
    \begin{tikzcd}
        \tilde f \times F\times  X \arrow[d, hookrightarrow] \arrow[r, "\Delta"] & \tilde f \times F\times F\times  X \arrow[d, hookrightarrow]\\
        \tilde f\times F\times \widetilde X  \arrow[r, "\widetilde \Delta"]  & \tilde f\times F\times F\times  \widetilde X
    \end{tikzcd}
\]
\[
\begin{tikzcd}
        \tilde f \times f\times X \arrow[d, "\Delta_f"] \arrow[r, hookrightarrow, "k"] & \tilde f \times F\times  X \arrow[d, "\Delta"]\\
        \tilde f\times F\times f \times X  \arrow[r, "id\times i"]  & \tilde f\times F\times F \times  X.
    \end{tikzcd}
\]
All the maps are obvious, and the diagrams are Cartesian (although we don't need this property for the second one). 
By Proposition \ref{composition stable envelopes}, we have
\begin{equation}
    \label{equation convolution of stable envelopes}
    \StabX{\widetilde X}{\widetilde C}{\tilde f}=(\tilde p_{13})_{\oast}\widetilde \Delta^{\oast}\left(\StabX{\widetilde X}{ C}{F}(z)\boxtimes \StabX{F}{ \widetilde C/ C}{\tilde f}(z\h^{\pm \gamma(f)})\right).
\end{equation}
Recall that although the map $\tilde p_{13}$ is not proper in general, its restriction to the support of 
\[
\widetilde\Delta^{\oast}\left(\StabX{\widetilde X}{ C}{f}\boxtimes \StabX{\widetilde X}{ \widetilde C/ C}{\tilde f}(z\h^{\pm \gamma(f)})\right)
\]
is proper, so the pushforward in elliptic cohomology is well defined. This argument justifies all the pushforwards in this proof. 

We now begin the computation that will complete the proof. All the steps below will be done in $\Tt$-equivariant elliptic cohomology. The $\Tt$-equivariance on $\widetilde X$ is induced by $\varphi: \Tt\to \widetilde \Tt$. For the sake of clarity, we drop the pullback $\varphi^{\oast}$, and we also denote by $\tau$ the shift $z\mapsto z\h^{\pm \gamma(f)}$.
We compute
\begin{align*}
    \StabX{\widetilde X}{\widetilde C}{\tilde f}\Big|_{\tilde f\times X}
    &=\left((\tilde p_{13})_{\oast} \widetilde\Delta^{\oast}\left(\StabX{\widetilde X}{ C}{F}\boxtimes \tau^{\oast}\StabX{F}{ \widetilde C/ C}{\tilde f}\right)\right)\Big|_{\tilde f\times X }
    \\
    &=( p_{13})_{\oast} \left( \widetilde\Delta^{\oast}\left(\StabX{\widetilde X}{ C}{F}\boxtimes \tau^{\oast}\StabX{F}{ \widetilde C/ C}{\tilde f}\right)\Big|_{\tilde f\times F\times  X }\right)
    \\
    &=( p_{13})_{\oast} \left( \Delta^{\oast}\left(\StabX{\widetilde X}{ C}{F}\Big|_{F\times  X }\boxtimes \tau^{\oast}\StabX{F}{ \widetilde C/ C}{\tilde f}\right)\right).
\end{align*}
In the first step, we applied equation \eqref{equation convolution of stable envelopes}. In the second one, we used the compatibility of pushforward and pullback induced by the first Cartesian diagram above. Finally, in the last line, we exchanged the order of the restrictions, as prescribed by the second diagram above. Applying \eqref{fundamental equality proof A fusion of stable envelopes}, we continue our computation as follows
\begin{align*}
    \StabX{\widetilde X}{\widetilde C}{\tilde f}\Big|_{\tilde f\times X}
    &=(p_{13})_{\oast} \left( \Delta^{\oast}\left(\Theta(N^-_j)i_{\oast}\left( \StabX{X}{C}{f}\right)\boxtimes \tau^{\oast}\StabX{F}{ \widetilde C/ C}{\tilde f}\right)\right)  \\
    &=\Theta(N^-_j)(p_{13})_{\oast} \left( \Delta^{\oast}\left(i_{\oast}\left( \StabX{X}{C}{f}\right)\boxtimes \tau^{\oast}\StabX{F}{ \widetilde C/ C}{\tilde f}\right)\right)  \\
    &=\Theta(N^-_j)(p_{13})_{\oast} \Delta^{\oast} (i\times id)_{\oast}\left( \StabX{X}{C}{f}\boxtimes \tau^{\oast}\StabX{F}{ \widetilde C/ C}{\tilde f}\right)  \\
    &=\Theta(N^-_j)(p_{13})_{\oast} k_{\oast}\Delta^{\oast}_f \left( \StabX{X}{C}{f}\boxtimes \tau^{\oast}\StabX{F}{ \widetilde C/ C}{\tilde f}\right)  \\
    &=\Theta(N^-_j)\Delta^{\oast}_f \left( \StabX{X}{C}{f}\boxtimes \tau^{\oast}\StabX{F}{ \widetilde C/ C}{\tilde f}\right). 
\end{align*}
In the second step, we used the fact that $\Theta(N^-_j)$ only depends on the equivariant variables (because $N_j^-$ is topologically trivial) and hence can be pulled outside the chain of maps. In the fourth step, we used the compatibility of pushforward and pullback induced by the third Cartesian diagram above. Finally, in the last step we used the equality $(p_{13})_{\oast} k_{\oast}=(p_{13}\circ k)_{\oast}=id_{\oast}=id$.

Reintroducing the dropped pullback $\varphi^{\oast}$, we get the formula 
\begin{align*}
j^{\oast}\varphi^{\oast}\StabX{\widetilde X}{\widetilde C}{\tilde f}
&=\Theta(N^-_j)\Delta^{\oast}_f \left( \StabX{X}{C}{f}\boxtimes \varphi^{\oast} \tau^{\oast}\StabX{F}{ \widetilde C/ C}{f}\right)\\
&=  \Theta(N^-_j) \varphi^{\oast} \left(\StabX{F}{ \widetilde C/ C}{\tilde f}(z\h^{\pm \gamma(f)})\right)\Big|_{f} \StabX{X}{C}{f}.
\end{align*}
To conclude the proof, it suffices to observe that, in virtue of Lemma \ref{restrction Chern roots in A resolution} together with the fact that the line bundles defining the stable envelopes are tautological, restricting to $f$ after pulling back by $\varphi^{\oast}$ is the same as pulling back and restricting to $\tilde f_\sharp$.
\end{proof}

\subsection{More fusion formulas}

\label{explicit D5 resolutions}
Let $X$ be separated or co-separated and let $\widetilde X$ be the resolution of its $i$-th D5 brane $\Ab$ with local charge decomposition $\w=\w'+\w''$.
Corollary \ref{corollary nicest one-term D5 resolution} implies, via localization and triangularity of the stable envelopes, the following statements, which relate the fixed point localizations of the stable envelopes of $X$ and $\widetilde X$. The coefficients of these formulas are functions of the coefficients of the R-matrix of the bow varieties $\ttt{{\fs}1\fs2{\fs}3\fs\dots {\fs}$\w'+\w'' ${\bs}$\w''${\bs}}$ and \ttt{{\bs}$\w'${\bs}$\w'+\w'' ${\fs}\dots{\fs}3{\fs}2{\fs}1{\fs}}, in the separated and co-separated case, respectively. Remarkably, these fusion-like statements for stable envelopes will turn out be mirror dual to the NS5 resolutions that we develop in the next chapter.

\begin{proposition}
\label{multiple terms D5 resolution stabs separated case}
     Assume $X$ is separated. Let $f,g\in X^{\At}$ and let $\tilde f_\sharp ,\tilde g_\sharp\in F^{\widetilde \At}\subset \widetilde X^{\widetilde \At}$ be their distinguished resolutions. We have
    \begin{multline*}
        \frac{\StabX{X}{C}{f}\Big|_g}{\StabX{X}{C}{g}\Big|_g}(a,z,\h)=\\
        \sum_{\tilde f\in F^{\widetilde \At}}\frac{R_{\tilde{f}, \tilde{f}_\sharp}}{R_{\tilde{g}_\sharp, \tilde{g}_\sharp}}(a_i\h^{-\w_i''},a_i, z\hbar^{-\gamma(f)}, \hbar)\frac{\StabX{\widetilde X}{\widetilde C}{\tilde f}\Big|_{\tilde{g}_\sharp}}{\StabX{\widetilde X}{\widetilde C}{\tilde g_\sharp}\Big|_{\tilde g_\sharp}}(a_1,\dots,a_i\h^{-\w''_i},a_i,\dots, a_n, z,\h).
    \end{multline*}
Here, $\mathfrak{C}$ and $\widetilde{\mathfrak{C}}$ are the standard chambers (cf. section \ref{subsection: hamber structure}) and $R$ is the R-matrix 
\[
R(a'_i, a''_i, z, \h)=R_{{\lbrace a_i'<a_i''\rbrace}{\lbrace a_i'>a_i''\rbrace}}(a'_i, a''_i, z, \h).
\]
of the bow variety $F=\ttt{{\fs}1\fs2{\fs}3\fs\dots {\fs}$\w'+\w'' ${\bs}$\w''${\bs}}$ and $\gamma_i(f)$ is the number of ties of $f$ connecting the $i$-th NS5 brane to the D5 branes left to the resolved D5 brane. 
\end{proposition}
\begin{proposition}
\label{multiple terms D5 resolution stabs co-separated case}
    Assume $X$ is co-separated. Let $f,g\in X^{\At}$ and let $\tilde f_\sharp ,\tilde g_\sharp\in F^{\widetilde \At}\subset \widetilde X^{\widetilde \At}$ be their distinguished resolutions. We have
    \begin{multline*}
        \frac{\StabX{X}{C}{f}\Big|_g}{\StabX{X}{C}{g}\Big|_g}(a,z,\h)=\\
        \sum_{\tilde f\in F^{\widetilde \At}}\frac{R_{\tilde{f}, \tilde{f}_\sharp}}{R_{\tilde{g}_\sharp, \tilde{g}_\sharp}}(a_i,a_i\h^{\w_i'},z\hbar^{\gamma(f)}, \hbar)\frac{\StabX{\widetilde X}{\widetilde C}{\tilde f}\Big|_{\tilde{g}_\sharp}}{\StabX{\widetilde X}{\widetilde C}{\tilde g_\sharp}\Big|_{\tilde g_\sharp}}(a_1,\dots,a_i,a_i\h^{\w'_i},\dots, a_m, z,\h).
     \end{multline*}
Here, $\mathfrak{C}$ and $\widetilde{\mathfrak{C}}$ are the standard chambers (cf. section \ref{subsection: hamber structure}) and $R$ is the R-matrix 
\[
R(a'_i, a''_i, z, \h)=R_{{\lbrace a_i'<a_i''\rbrace}{\lbrace a_i'>a_i''\rbrace}}(a'_i, a''_i, z, \h)
\]
of the bow variety $F=\ttt{{\bs}$\w'${\bs}$\w'+\w'' ${\fs}\dots{\fs}3{\fs}2{\fs}1{\fs}}$ and $\gamma_i(f)$ is the number of ties of $f$ connecting the $i$-th NS5 brane to D5 branes left to the resolved D5 brane. 
\end{proposition}
\begin{proof}
    The proof of both propositions follows from Corollary \ref{corollary nicest one-term D5 resolution}, the definition of R-matrix \eqref{localized R-matrix relations}, and the fact that $\tilde g_\sharp$ is maximal in the order induced by $\lbrace a_i'<a_i''\rbrace$ (hence only the diagonal term at the denominator survives).
\end{proof}

\begin{remark}
    As observed in Remark \ref{remark interpretation fixed points and fixed components A resolution}, there are exactly $\binom{\w'+\w''}{\w'}$ fixed points $\tilde f$ in $F$. As a consequence, the Propositions above equate the normalized stable envelopes of $X$ with a linear combination of stable envelopes of $\widetilde X$ involving $\binom{\w'+\w''}{\w'}$ terms. Moreover, since the R-matrix $R(a,b, \h, z)$ of $F$ only depends on the ratio $a/b$, the coefficients in these formulas depend only on $z$ and $\h$. 
\end{remark}


\subsection{Symmetric group action}

In this section, we assume that all bow varieties are separated or co-separated. Consider the transition interchanging two adjacent D5 branes in such a way that their (local) charges do not change:

\begin{equation*}
\begin{tikzpicture}[baseline=(current  bounding  box.center), scale=.35]
\draw[thick] (0,1)--(1,1) node [above] {$d_1$} -- (2,1);
\draw[thick,blue] (3,0) node [below] {$\Ab_i$}--(2,2);
\draw[thick] (3,1) -- (4,1) node [above] {$d_2$}--(5,1);
\draw[thick,blue] (6,0)  node [below] {$\Ab_{i+1}$} -- (5,2);
\draw[thick] (6,1)--(7,1) node [above] {$d_3$}--(8,1);
\draw[ultra thick, <->] (10,1)--(11.5,1)  -- (13,1);
\draw[thick] (15,1)--(16,1) node [above] {$d_1$} -- (17,1);
\draw[thick,blue] (18,0) node [below] {$\Ab_{i+1}$} --(17,2);
\draw[thick] (18,1)--(19,1) node [above] {$d'_2$}--(20,1);
\draw[thick,blue] (21,0) node [below] {$\Ab_i$} --(20,2);
\draw[thick] (21,1)--(22,1) node [above] {$d_3$}--(23,1);
\end{tikzpicture}
\qquad\quad \text{for } d_2+d'_2=d_1+d_3.
\end{equation*}

This transition defines an action of the symmetric group $S_n$ on the set of bow varieties with $n$ D5 branes via its standard generators $\sigma_{i,i+1}$ of $S_n$. This $S_n$-action is compatible with the equivariant geometry of the bow varieties, in the sense that there exists a one-to-one correspondence between the torus fixed points of $X$ and $\sigma\cdot X$ for every $\sigma\in S_n$. For the action of the generator $\sigma_{i,i+1}$, the correspondence swaps the ties arising from the interchanged branes:
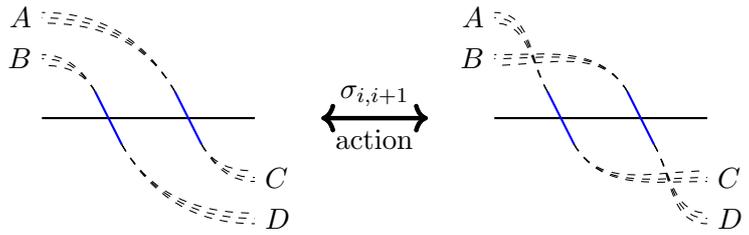
\[
\begin{tikzpicture}[scale=.35]
\draw [thick] (0,1) --(8,1);  
\draw [thick, blue] (3,0) -- (2,2);
\draw[thick, blue] (6,0)--(5,2);
\draw [dashed](2,2) to [out=120,in=0] (0,3) ;
\draw [dashed](2,2) to [out=120,in=0] (0,3.25) node [left] {$B$} ;
\draw [dashed](2,2) to [out=120,in=0] (0,3.4) ;
\draw [dashed](6,0) to [out=-60,in=180] (8,-1)  ;
\draw [dashed](6,0) to [out=-60,in=180] (8,-1.25)  node [right] {$C$} ;
\draw [dashed](6,0) to [out=-60,in=180] (8,-1.4);
\draw [dashed](3,0) to [out=-60,in=180] (8,-2.6);
\draw [dashed](3,0) to [out=-60,in=180] (8,-2.8) node [right]{$D$};
\draw [dashed](3,0) to [out=-60,in=180] (8,-3) ;
\draw [dashed](5,2) to [out=120,in=0] (0,4.6) ;
\draw [dashed](5,2) to [out=120,in=0] (0,4.8) node [left]{$A$};
\draw [dashed](5,2) to [out=120,in=0] (0,5);
\draw[ultra thick, <->] (10.5,1)--(12.5,1) node[above]{$\sigma_{i,i+1}$} -- (14.5,1);
\draw[ultra thick, <->] (10.5,1)--(12.5,1) node[below]{action} -- (14.5,1);
\begin{scope}[xshift=17cm]
\draw [thick] (0,1) --(8,1);  
\draw [thick, blue] (3,0) -- (2,2);
\draw[thick, blue] (6,0)--(5,2);
\draw [dashed](5,2) to [out=120,in=0] (0,3) ;
\draw [dashed](5,2) to [out=120,in=0] (0,3.25) node [left] {$B$} ;
\draw [dashed](5,2) to [out=120,in=0] (0,3.4) ;
\draw [dashed](3,0) to [out=-60,in=180] (8,-1)  ;
\draw [dashed](3,0) to [out=-60,in=180] (8,-1.25)  node [right] {$C$} ;
\draw [dashed](3,0) to [out=-60,in=180] (8,-1.4);
\draw [dashed](6,0) to [out=-60,in=180] (8,-2.6);
\draw [dashed](6,0) to [out=-60,in=180] (8,-2.8) node [right]{$D$};
\draw [dashed](6,0) to [out=-60,in=180] (8,-3) ;
\draw [dashed](2,2) to [out=120,in=0] (0,4.6) ;
\draw [dashed](2,2) to [out=120,in=0] (0,4.8) node [left]{$A$};
\draw [dashed](2,2) to [out=120,in=0] (0,5);
\end{scope}
\end{tikzpicture}.
\]
Given some fixed point $f$ in $X$, we denote by $\sigma\cdot f$ the corresponding fixed point in $\sigma \cdot X$. Notice that, even if $X=\sigma \cdot X$, the fixed points $f$ and $\sigma\cdot f$ generally do \emph{not} coincide.

The subset of bow varieties fixed by the $S_n$ action consists of those bow varieties whose D5 branes all have the same local charge. Among these, those whose local charges are equal to one are the cotangent bundles of partial flag varieties (cf. Section \ref{sec:quivers}):
\[
\ttt{\fs$k_1$\fs\dots \fs$k_m$\fs$n$\bs$n-1$\bs\dots\bs2\bs1\bs}\cong T^*\Fl(k_1,\dots, k_m; n).
\]
The shuffling of the D5 branes of $T^*\Fl(k_1,\dots, k_m; n)$ has a clear geometric meaning: it coincides with the standard action of $S_n\subset \GL(n)$ on the variety. This action descends to a nontrivial symmetry of the stable envelopes of $T^*\Fl$ and hence, via Theorem \ref{Fusion of D5 branes for separated brane diagrams}, to a relation between the stable envelopes of any pair of bow varieties $X$ and $\sigma\cdot X$.

\begin{proposition}
\label{innocent identity}
Let $\mathfrak{C}=\lbrace a_1<\dots< a_n \rbrace$ be the standard chamber of a bow variety $X$ with $n$ D5 branes. Consider the following action of the symmetric group on the set of equivariant parameters 
\[
\sigma\cdot (a_1,\dots a_n)=(a_{\sigma(1)},\dots, a_{\sigma(n)}) 
\]
and, accordingly, on the chambers
\[
\sigma\cdot \mathfrak{C}=\lbrace a_{\sigma(1)}<\dots< a_{\sigma(n)} \rbrace.
\]
Then the formula 
\begin{equation}
\label{formula innocent identity}
    \frac{\Stab{C}{f}\Big|_g}{\Stab{C}{g}\Big|_g} (a,z,\h) = \frac{\Stab{\sigma \cdot C}{\sigma \cdot f}\Big|_{\sigma\cdot g}}{\Stab{\sigma \cdot C}{\sigma\cdot g}\Big|_{\sigma\cdot g}}(\sigma \cdot a, z,\h)
\end{equation}
holds for all $f,g\in X^A$ and $\sigma \in S_n$.
\end{proposition}
\begin{proof}
Assume first that $X\cong T^*\Fl(k_1,\dots, k_m; n)$. The action of $S_n\subset \GL(n)$ on $X$ descends to an action on its elliptic cohomology. Notice that $S_n$ also acts on the maximal torus $A\subset \GL(n)$ by conjugation, and we have
\[
\sigma \cdot (a_1,\dots, a_n)= (a_{\sigma(1)},\dots, a_{\sigma(n)}).
\]
Since a weight $\chi(a)$ is $\mathfrak{C}$-positive iff $\chi(\sigma\cdot a)$ is $\sigma \cdot\mathfrak{C}$-positive, 
then
\[
\Stab{\sigma \cdot C}{\sigma \cdot f}(\sigma \cdot a, z,\h)= \Stab{C}{f}(a,z,\h)
\]
for all $f,g\in X^A$ and $\sigma \in S_n$. This proves the proposition if $X\cong T^*\Fl(k_1,\dots, k_m; n)$\footnote{Alternatively, one can check this statement using the explicit formulas for the stable envelopes of partial flag varieties provided by \cite{botta2021shuffle, RRTBshuffle}.}. For a general bow variety $X$ with no branes $\Ab$ such that $\w(\Ab)=0$, the result follows by considering its maximal D5 resolution $\widetilde X$ (which is the cotangent bundle of a partial flag variety) and applying Theorem \ref{Fusion of D5 branes for separated brane diagrams} on both sides of \eqref{formula innocent identity}. Finally, for a general bow variety (possibly with branes such that $\w(\Ab)=0$), the result follows from \eqref{equation stab with weight zero banes}.
\end{proof}


\subsection{Fusion of R-matrices}
\label{subsection fusion of R-matrices}

Let $X$ be separated or co-separated with D5 resolution $\widetilde X$. Consider the usual actions of $\At$ on $X$ and $\widetilde \At$ on $\widetilde X$. In this section, we exploit Theorem \ref{Fusion of D5 branes for separated brane diagrams} to relate the R-matrices of $X$ and $\widetilde X$.

By Theorem \ref{Fusion of D5 branes for separated brane diagrams}, for every chamber $\mathfrak{C}$ of $\At$ there exist a chamber $\widetilde{\mathfrak{C}}$ of $\widetilde{\At}$ and classes $\beta_i$ such that
    \[
    \beta_i\StabX{X}{C}{f}=j^{\oast} \varphi^{\oast}\StabX{\widetilde X}{\widetilde C}{\tilde f_i} \qquad \forall \tilde f_i\in F^{\widetilde \At/\At}.
    \]
    Here, $F$ is the fixed component of $\widetilde X^{\At}$ containing $f$, and the pullback by $\varphi$ encodes a shift of the equivariant parameters. Notice that $\varphi$ is independent of the chambers while $\beta_i$ only depends on the quotient $\widetilde{\mathfrak{C}}/\mathfrak{C}$. By Corollary \ref{corollary nicest one-term D5 resolution}, we can always choose a chamber $\widetilde{\mathfrak{C}}$ for $\widetilde \At$ and a fixed-point $\tilde f_\sharp\in \widetilde X^{\widetilde \At}$ such that $\beta_\sharp\neq 0$.
    Consider now two chambers $\mathfrak{C}$ and $\mathfrak{C}'$ for the action of $\At$ on $X$ and choose resolutions $\widetilde{\mathfrak{C}}$ and $\widetilde{\mathfrak{C}}'$ such that $\widetilde{\mathfrak{C}}/\mathfrak{C}=\widetilde{\mathfrak{C}}'/\mathfrak{C}'$ and $\beta_\sharp\neq 0$.
    \begin{proposition}
    \label{proposition R-matrix fusion}
        The following formula holds:
        \begin{equation*}
            \left(\Rmat{C'}{C}\right)_{gf}=\sum_{i}\varphi^{\oast}(\Rmat{\widetilde C'}{\widetilde C})_{\tilde g_i\tilde f_\sharp} \frac{\beta_i}{\beta_\sharp}.
        \end{equation*}
    \end{proposition}
    \begin{example} \rm
        Assume that $X$ is a separated bow variety $X$ with $m$ NS5 branes and two D5 branes with (local) charges $\w_1$ and $\w_2$. Then the cohomology $H^*_{\Tt}(X^{\At})$ is a weight subspace of the quantum group representation $\Lambda^{\w_1}\C^{m}(a_1)\otimes \Lambda^{\w_2}\C^{m}(a_2)$ (cf. Remark \ref{remark rep. thy. interpretation cohomology bow varieties}). The actual weight is determined by the NS5 charges of $X$. In this case, there are only two chambers $\mathfrak{C}=\lbrace a_2<a_1\rbrace$ and $\mathfrak{C}'=\lbrace a_1<a_2\rbrace$, so we set $R_{\w_1,\w_2}:=\Rmat{C'}{C}$.

        Now consider the D5 resolution of the first brane of X with charge decomposition $\w_1=\w'_1+\w''_1$. The resulting variety $\widetilde X$ has $m$ NS5 branes and three D5 branes of charges $\w'_1$, $\w''_1$, and $\w_2$. Its cohomology $H^*_{\widetilde{\Tt}}(\widetilde X^{\widetilde \At})$ is a subspace of  $\Lambda^{\w'_1}\C^{m}(a'_1)\otimes \Lambda^{\w''_1}\C^{m}(a''_1) \otimes \Lambda^{\w_2}\C^{m}(a_2)$. For the R-matrix $\Rmat{\widetilde C'}{\widetilde C}$ in Proposition \ref{proposition R-matrix fusion}, we can choose $\Rmat{\widetilde C'}{\widetilde C}=R_{\lbrace a_1'<a_1''<a_2\rbrace, \lbrace a_2< a_1'<a_1'' \rbrace}$.
        Via wall-crossing, we get 
        \[
        \Rmat{\widetilde C'}{\widetilde C}(a'_1,a''_1,a_2)= R^{(23)}_{\w''_1,\w_2}(a''_1,a_2) R^{(13)}_{\w'_1,\w_2}(a'_1,a_2)
        \]
        on $\Lambda^{\w'_1}\C^{m}(a'_1)\otimes \Lambda^{\w''_1}\C^{m}(a''_1) \otimes \Lambda^{\w_2}\C^{m}(a_2)$.  Here, the superscripts indicate the factors on which the operators act: for example, $R^{(23)}_{\w''_1,\w_2}=1\otimes R_{\w''_1, \w_2}$.
        Since the pullback $\varphi^{{\oast}}$ simply forces the change of variables $(a_1',a_1'')=(a_1h^{-\w_1''}, a_1)$, Proposition \ref{proposition R-matrix fusion} describes the matrix elements of $R_{\w_1,\w_2}(a_1,a_2)$ as a linear combination of matrix elements of 
        \[
        R^{(23)}_{\w''_1,\w_2}(a_1,a_2) R^{(13)}_{\w'_1,\w_2}(a_1h^{-\w_1''},a_2).
        \]
        Moreover, not all the matrix elements appear, but only those corresponding to the fixed points of $\widetilde X^{\widetilde \At}$ resolving the fixed points in $X^{\At}$ in the sense of Lemma \ref{lemma Fixed components D5 res}. It is easy to check that these elements correspond to those basis elements of $\Lambda^{\w'_1}\C^{m}\otimes \Lambda^{\w''_1}\C^{m} \otimes \Lambda^{\w_2}\C^{m}$ that are not in the kernel of the canonical map
        \[
        \Lambda^{\w'_1}\C^{m}\otimes \Lambda^{\w''_1}\C^{m} \otimes \Lambda^{\w_2}\C^{m}\xrightarrow[]{\wedge\otimes 1} \Lambda^{\w_1}\C^{m}\otimes \Lambda^{\w_2}\C^{m}
        \]
        wedging the first two tensor components. 
        Altogether, these observations reveal that Proposition \ref{proposition R-matrix fusion} provides an elliptic generalization of the famous fusion formula for R-matrices introduced in the '80s \cite{Kulish:1981gi}. This point of view also offer a representation theoretic interpretation of the results of this section: they describe the geometry behind the fusion procedure for R-matrices. 
    \end{example}
    \begin{proof}[Proof of Proposition \ref{proposition R-matrix fusion}]
        It suffices to prove that 
    \begin{equation}
        \label{equation to prove for proof R-matrix relations}
        \StabX{X}{C}{f}=\sum_{g}\alpha_{gf}\StabX{X}{C'}{g},
        \qquad 
        \alpha_{gf}=\sum_i\varphi^{\oast}(\Rmat{\widetilde C'}{\widetilde C})_{\tilde g_i\tilde f_\sharp} \frac{\beta_i}{\beta_\sharp}.
    \end{equation}
    Indeed, localizing at fixed points and using \eqref{localized R-matrix relations}, one deduces that $\alpha_{gf}$ must coincide with $\left( \Rmat{C'}{C}\right)_{gf}$.
    By the R-matrix relation for $\widetilde X$, we get
    \begin{align*}
        \StabX{X}{C}{f}
        &=\frac{1}{\beta_\sharp}j^{\oast} \varphi^{\oast} \StabX{\widetilde X}{\widetilde C}{\tilde f_\sharp}
        \\ 
        &=\frac{1}{\beta_\sharp} j^{\oast} \varphi^{\oast}\sum_{\tilde g_i\in \widetilde X^{\widetilde \At}} (\Rmat{\widetilde C'}{\widetilde C})_{\tilde g_i\tilde f_\sharp}\StabX{\widetilde X}{\widetilde C'}{\tilde g_i}
        \\
        &=\frac{1}{\beta_\sharp} j^{\oast} \varphi^{\oast}\sum_{\substack{G\cap X\neq \emptyset \\ \tilde g_i\in G^{\widetilde \At}}} (\Rmat{\widetilde C'}{\widetilde C})_{\tilde g_i\tilde f_\sharp}\StabX{\widetilde X}{\widetilde C'}{\tilde g_i}+\frac{1}{\beta_\sharp} j^{\oast} \varphi^{\oast}\sum_{\substack{G\cap X= \emptyset \\ \tilde g_i\in G^{\widetilde \At}}} (\Rmat{\widetilde C'}{\widetilde C})_{\tilde g_i\tilde f_\sharp}\StabX{\widetilde X}{\widetilde C'}{\tilde g_i}.
        \end{align*}
        Here, the spaces $G$ are the $\At$-fixed components of $\widetilde X^{\At}$. Triangularity of stable envelopes together with part (2) of Lemma \ref{lemma equality attracting sets bow varieties} imply that all the terms of the second summation vanish after restricting to $X$, i.e. after applying $j^{\oast}$. Hence, applying Theorem \ref{Fusion of D5 branes for separated brane diagrams} backward, we get 
        \begin{align*}
        \StabX{X}{C}{f}&=\sum_{\tilde g_i} \varphi^{\oast}(\Rmat{\widetilde C'}{\widetilde C})_{\tilde g_i\tilde f_\sharp} \frac{1}{\beta_\sharp} j^{\oast} \varphi^{\oast}\left(\StabX{\widetilde X}{\widetilde C'}{\tilde g_i}\right)
        \\
        &=\sum_{\tilde g_i} \varphi^{\oast}(\Rmat{\widetilde C'}{\widetilde C})_{\tilde g_i\tilde f_\sharp} \frac{\beta_i}{\beta_\sharp}\StabX{X}{C'}{g}
        \\
        &=\sum_{g}\left(\sum_i\varphi^{\oast}(\Rmat{\widetilde C'}{\widetilde C})_{\tilde g_i\tilde f_\sharp} \frac{\beta_i}{\beta_\sharp}\right)\StabX{X}{C'}{g}.
    \end{align*}
\end{proof}


\section{NS5 Resolutions}
\label{Fusion of NS5 branes}

\subsection{NS5 resolutions for bow varieties}


Let $\D$ be a separated or co-separated brane diagram.
Let $\overline \D$ be the brane diagram obtained by replacing a single NS5 brane $\Zb$ of weight $\w=\w(\Zb)$ in $\D$ by a pair of consecutive NS5 branes $\Zb'$ and $\Zb''$ of weights $\w'=\w(\Zb')\geq 0$ and $\w''=\w(\Zb'')\geq 0$ such that $\w=\w'+\w''$.
We call $\overline \D$ an NS5 resolution of the brane diagram $\D$, and the branes $\Zb'$ and $\Zb''$ resolving branes. Notice that if $\D$ is separated (resp. co-separated), then $\overline \D$ is also separated (resp. co-separated).

Let now $\overline X$ and $X$ be the bow varieties associated with $\overline \D$ and $\D$, respectively. We say that $\overline X$ is a NS5 resolution of the bow variety $X$.
The main goal of this section is to build a correspondence 
\[
\overline X\leftarrow L\rightarrow X
\]
and exploit it to compare the stable envelopes of $\overline X$ and $X$. As for the D5 resolutions introduced in the previous section, we first work at the level of the spaces of representations $\overline{\MM}$ and $\MM$ and consider the additional actions of the gauge groups $\overline G$ and $G$.
By definition, the space $\overline{\MM}$ differs from $\MM$ by replacing the two-way part associated with the brane $\Zb$ with two consecutive two-way parts associated with $\Zb'$ and $\Zb''$:
\[
\begin{tikzcd}
    W_{\Zb_-}  \arrow[r, bend right, swap, "D"]& W_{\Zb_+} \arrow[l, bend right, swap, "C",  "\circ" marking]
\end{tikzcd}
\quad \Longrightarrow \quad 
\begin{tikzcd}
    W_{\Zb_-}=W_{\Zb'_-}  \arrow[r, bend right, swap, "D'"]&  W_{\Zb'_+}=W_{\Zb''_-} \arrow[r, bend right, swap, "D''"]\arrow[l, bend right, swap, "C'",  "\circ" marking] & W_{\Zb''_+} =W_{\Zb_+} \arrow[l, bend right, swap, "C''",  "\circ" marking]
\end{tikzcd}.
\]
As usual, the circles in the arrows indicate rescaling by $\h$.
We now introduce an ``in between'' space $\mathbb{L}$, defined as the space of representations obtained from $\MM$ by the replacement described below. If $\D$ is a separated diagram, we substitute
\begin{equation}
    \label{equation separated quiver modification for in-between space}
\begin{tikzcd}
    W_{\Zb_-}  \arrow[r, bend right, swap, "D"]& W_{\Zb_+} \arrow[l, bend right, swap, "C",  "\circ" marking]
\end{tikzcd}
\quad \Longrightarrow \quad 
\begin{tikzcd}
    W_{\Zb_-}=W_{\Zb'_-}  \arrow[rr, bend right, swap, "D"]&  W_{\Zb'_+}=W_{\Zb''_-} \arrow[l, bend right, swap, "C'"] & W_{\Zb''_+} =W_{\Zb_+}   \arrow[l, bend right, swap, "C''",  "\circ" marking]
\end{tikzcd}.
\end{equation}
The absence of the $\h$-action on $C'$ is not a typo.
Instead, if $\D$ is co-separated, we substitute 
\begin{equation}
    \label{equation co-separated quiver modification for in-between space}
\begin{tikzcd}
    W_{\Zb_-}  \arrow[r, bend right, swap, "D"]& W_{\Zb_+} \arrow[l, bend right, swap, "C",  "\circ" marking]
\end{tikzcd}
\quad \Longrightarrow \quad 
\begin{tikzcd}
    W_{\Zb_-}=W_{\Zb'_-} \arrow[r, bend right, swap, "D'"] &  W_{\Zb'_+}=W_{\Zb''_-}  \arrow[r, bend right, swap, "D''"] & W_{\Zb''_+} =W_{\Zb_+}   \arrow[ll, bend right, swap, "C",  "\circ" marking]
\end{tikzcd}.
\end{equation}
The asymmetry in this construction, namely the choice of ``resolving'' the map $C$ rather than $D$ or vice versa, is forced by our choice of stability condition. 

Since the three-way parts of the quiver defining $\mathbb{L}$ are the same as those defining $\MM$ and $\overline{\MM}$, the torus $\At$ naturally acts on $\mathbb{L}$. We extend this to an action of $\At\times\Cs_{\h}\times \overline G$ by declaring that $\overline G$ acts on all the maps by conjugation and $\Cs_{\h}$ acts with weight one on the maps $C$ and $C'$ as prescribed by the circles in the diagrams above.

The three spaces discussed above are related by means of a diagram of the form
\begin{equation}
    \label{map prequotients NS5 resolution}
    \begin{tikzcd}
    \overline{\MM} & \mathbb{L} \arrow[l] \arrow[r] & \MM.
\end{tikzcd}
\end{equation}
If $\D$ is separated, the left pointing map supplements the assignments $D'=C''D$ and $D''=DC'$, while the right one supplements $C=C'C''$. If instead $\D$ is co-separated, the left map supplements the assignments $C'=CD''$ and $C''=D'C$, while the right one supplements $D=D''D'$.

\begin{remark}
\label{remark equivariance NS5 resolutions}
    In the co-separated case, both maps in \eqref{map prequotients NS5 resolution} are clearly $\overline G\times \At\times \Cs_{\h}$-equivariant. In the separated case instead, the left-pointing map in \eqref{map prequotients NS5 resolution} does not respect the action of $\Cs_{\h}$. However, the map does become equivariant if we prescribe an additional weight one $\Cs_{\h}$ action on the vertices left of the resolving brane $\Zb'$. But since $\overline G$ acts on all these vertices, all the quotient maps will be unaffected by this change of $\Cs_{\h}$ action, and hence will be $\Tt=\At\times \Cs_{\h}$-equivariant.
\end{remark}

Next, we introduce a moment map $\mu_{\mathbb{L}}: \mathbb{L}\to \NN$ by forcing commutativity of
\[
\begin{tikzcd}
    \mathbb{L}\arrow[d]\arrow[r, "\mu_{\mathbb{L}}"] &  \NN\\
    \MM\arrow[ur, swap, "\mu"] &
\end{tikzcd}.
\]
Explicitly, in the separated (resp. co-separated) case, $\mu_{\mathbb{L}}$ is obtained from $\mu$ by replacing the map $C$ in \eqref{equation separated quiver modification for in-between space} with $C''C'$ (resp. $D$ in \eqref{equation co-separated quiver modification for in-between space} with $D'D''$). 

As for bow varieties, we say that a tuple $(a,b,A,B,C,D)\in\mathbb{L}$ is stable if all its three-way parts (which are unaffected by the modifications \eqref{equation separated quiver modification for in-between space} and \eqref{equation co-separated quiver modification for in-between space}) satisfy the conditions (S1), (S2), and if there exists no proper $(a,b,A,B,C,D)$-invariant graded subspace $S\subset \oplus_i W_i$ such that $\Image(A)\subset S$ and $A_{\Ab}$ induces isomorphisms $W_{\Ab+}/ S_{\Ab+}\to W_{\Ab-}/S_{\Ab-}$ for all $D5$ branes $\Ab$. We define 
\[
L:=\mu_{\mathbb{L}}^{-1}(0)^{s}/\overline G.
\]
\begin{proposition}
\label{proposition NS5 resolution bow varieites}
     The maps \eqref{map prequotients NS5 resolution} descend to morphisms 
     \begin{equation}
     \label{correspondence NS5 resolution}
         \begin{tikzcd}
       \overline X & L \arrow[l, swap, "j"] \arrow[r, "p"] & X.
       \end{tikzcd}
     \end{equation}
     Moreover, both maps are $\Tt$-equivariant, $j$ is a closed immersion, and $p$ is proper. Moreover, if $X$ is separated (resp. co-separated) then the fibers of $p$ are isomorphic to the Grassmannian $\Gr{\w'}{\w'+\w''}$ (resp. $\Gr{\w''}{\w'+\w''}$). 
\end{proposition}

\begin{proof}
    Firstly, we show that the diagram
    \[
    \begin{tikzcd}
    \bar\mu^{-1}(0)^{s} \arrow[d, hookrightarrow] & \mu^{-1}_{\mathbb{L}}(0)^{s} \arrow[l] \arrow[r] \arrow[d, hookrightarrow] & \mu^{-1}(0)^{s} \arrow[d, hookrightarrow]\\
    \bar\mu^{-1}(0) & \mu^{-1}_{\mathbb{L}}(0) \arrow[l] \arrow[r] & \mu^{-1}(0)
    \end{tikzcd}
    \]
    is well defined (and hence commutative).
    Well definiteness of $\mu^{-1}_{\mathbb{L}}(0)\to \mu^{-1}(0)$ follows immediately from the definition of $\mu_{\mathbb{L}}$, while the one of $\mu^{-1}_{\mathbb{L}}(0)\to \bar\mu^{-1}(0)$ follows from an explicit computation. Similarly, the well-definiteness of the top line easily follows from the definition of stability on $\mathbb{L}$.
    Taking quotients of the top line by $\overline G$ and $G$, we get the sought-after maps $j$ and $p$.
    

    We now show that $p$ is proper. We only prove it in the separated case since the co-separated one is entirely analogous. 
    Assume first that all the D5 local charges of $X$ are equal to one, so that the latter can be expressed via the brane diagram $X=\ttt{\fs$k_1$\fs\dots \fs$k_m$\fs$k$\bs$k-1$\bs\dots\bs1\bs}$.
    Here, $\w=k_{i}-k_{i-1}$ for some $i\geq 0$. The NS5 resolution $\overline X$ associated with the splitting $\w=\w'+\w''$ corresponds to the diagram $\ttt{\fs$k_1$\fs\dots\fs$k'_i$\fs$k_i$\fs\dots \fs$k_m$\fs$k$\bs$k-1$\bs\dots\bs1\bs}$, where $k'_i=k_{i-1}+\w'$. Identifying $X$ and $\overline X$ with the varieties $T^*\Fl(k_1, \dots, k_i,\dots,  k_m;k)$ and $T^*\Fl(k_1, \dots, k'_i, k_i,\dots k_m;k)$, respectively, one sees that the map $p: L\to X$ fits in the pullback diagram 
    \[
        \begin{tikzcd}
       L\arrow[d] \arrow[r, "p"] & T^*\Fl(k_1, \dots, k_i,\dots k_m;k)\arrow[d]\\
       \Fl(k_1, \dots, k'_i, k_i,\dots k_m;k) \arrow[r, "p_0"] & \Fl(k_1, \dots, k_i,\dots k_m;k)
    \end{tikzcd}
    \]
    where the right vertical map is the projection to the base and $p_0$ is the canonical forgetful map between the two flag varieties. Since the latter is proper and its fiber is isomorphic to $\Gr{\w'}{\w'+\w''}$, the same holds for $p$. This proves the claim when all the D5 local charges of $X$ are one. 

    Next, assume that none of the D5 local charges is zero and let $X\hookrightarrow \widetilde X$ be the maximal resolution of the D5 branes. By definition, all the D5 branes of $\widetilde X$ have charge one. Notice that the NS5 parts of the bow diagrams $X$ and $\widetilde X$ coincide. Let $\overline{\widetilde X}$ be the resolution of the brane $\Zb$ in the bow diagram of $\widetilde X$. Then the maps $\overline X\leftarrow L \to X$ fit in the pullback\footnote{Informally speaking, the two diagrams are Cartesian because the NS5 resolution is uniquely determined by a modification of the NS5 part of the brane diagram (and hence of the associated part of the quiver) that is independent of the D5 branes.} diagrams  
    \begin{equation}
        \label{diagram proof maps NS5 resolutions}
        \begin{tikzcd}
        \overline{X}\arrow[d, hookrightarrow] & L \arrow[d, hookrightarrow] \arrow[l, swap, "j"] \arrow[r, "p"] & X\arrow[d, hookrightarrow] \\
        \overline{\widetilde{X}} &\widetilde{L} \arrow[l, swap, "\tilde j"] \arrow[r, "\tilde p"] & \widetilde{X}.
    \end{tikzcd}
    \end{equation}
    Since we already know that $\tilde p$ is proper with fibers isomorphic to $\Gr{\w'}{\w'+\w''}$, the same holds for $p$. It remains to consider the case when at least one of the $n$ D5 branes have charge zero. 
    Adding an extra NS5 brane with charge $n$ in the middle of the separated bow diagram (i.e. on the right of the last NS5 brane) induces an isomorphism of bow varieties and increases all the D5 charges by one---see Section~\ref{Sec:0charge}.
    By construction, this isomorphism is compatible with $p$ (and also with $j$), so we are done.

    We now prove that $j$ is a closed immersion. By the argument above, it suffices to assume that all the D5 local charges are strictly positive. Moreover, since the left square in diagram \eqref{diagram proof maps NS5 resolutions} is Cartesian, it suffices to assume that all the D5 charges are one and hence identify $X$ and $\overline X$ with $T^*\Fl(k_1, \dots, k_i,\dots,  k_m;k)$ and $T^*\Fl(k_1, \dots, k'_i, k_i,\dots k_m;k)$, respectively. 
    Seen as the Springer resolution of a nilpotent orbit closure, $T^*\Fl(k_1,\dots,k_m;k)$ admits the following description:
    \[
    T^*\Fl(k_1,\dots,k_m;k)=\Set{(F^\bullet, \phi)}{\phi(F^{l})\subset F^{l-1}\quad  \forall l=1,\dots,m }\subset \Fl(k_1,\dots,k_m;k)\times \gl_{k}.
    \]
    Here $F^\bullet= \lbrace {0}\subset F^{1}\subseteq F^{2}\subseteq\dots \subseteq F^{m}\subseteq \C^{k}\rbrace $ is a flag in $\C^{k}$. Explicitly, the identification between $T^*\Fl(k_1,\dots,k_m;k)$ and the 
    bow variety $X=\ttt{\fs$k_1$\fs\dots \fs$k_m$\fs$k$\bs$k-1$\bs\dots\bs1\bs}$, whose brane diagram has the form\footnote{Notice that by our choice of stability conditions, our partial flag varieties $\Fl$ parametrize quotients, not injective maps.}
    \[
    \begin{tikzcd}
        \C^{k_1} \arrow[r, bend right, swap, "D_1"] & \C^{k_2} \arrow[l, bend right,  swap, two heads, "C_1"] \arrow[r, bend right,  swap, "D_2"] & \C^{k_3} \arrow[l, bend right,  swap,  two heads, "C_2"] \arrow[r, no head, dotted]&  \C^{k_{m-1}} \arrow[r, bend right, swap, "D_{m-1}"] & \C^{k_m} \arrow[l, bend right,  swap, two heads, "C_{m-1}"] \arrow[r, bend right,  swap, "D_m"] & \C^{k} \arrow[l, bend right,  swap,  two heads, "C_m"] 
    \end{tikzcd}
    + \text{ three way parts}
    \]
    is given by the assignments $F^{m+1-k}=\ker(C_k\circ C_{k+1}\circ\dots \circ C_m\circ g^{-1})$ and $\phi= g\circ D_m \circ C_m\circ g^{-1}$. Here, $g\in \GL(k)$ can be expressed in terms of the maps $A$ and $a$ in the three-way part, and $\GL(k)$ acts on it from the left. Using the diagrammatic description of $L$ from \eqref{equation separated quiver modification for in-between space}, it is easy to check that the diagram 
    \[
    \begin{tikzcd}
        T^*\Fl \arrow[r, hookrightarrow] & \Fl\times \gl_k \arrow[dr]\\
        L \arrow[r, hookrightarrow]\arrow[u, "p"]\arrow[d, swap, "j"] & \overline \Fl\times \gl_k \arrow[u, "p_0\times id"]\arrow[r] & \gl_k \\
        T^*\overline \Fl \arrow[r, hookrightarrow] & \overline \Fl\times \gl_k \arrow[u, equal]\arrow[ur]
    \end{tikzcd}
    \]
    is commutative, and the upper square is Cartesian. 
    By Springer theory, the compositions $T^*\Fl\to \gl_k$ and $T^*\overline \Fl\to \gl_k$ are onto two nilpotent orbit closures $Z_1$ and $Z_2$. By surjectivity of $p$ and commutativity of the diagram above, it follows that $Z_1\subset Z_2$. Thus we have a well-defined diagram
    \begin{equation}
    \label{cartesian diagram NS5 res. flag variety case}
        \begin{tikzcd}
        L\arrow[r, "j"]\arrow[d] & T^*\overline \Fl\arrow[d]\\
        Z_1\arrow[r, hookrightarrow] & Z_2
    \end{tikzcd}
    \end{equation}
    which, as a consequence of the diagram above, is Cartesian. But $Z_1$ is closed in $\gl_k$ and hence also in $Z_2$; thus $Z_1\hookrightarrow Z_2$ is a closed immersion, and the same is true for $j$.
    
    Finally, equivariance follows from Remark \ref{remark equivariance NS5 resolutions}.  
    \end{proof}
    Set $L_0=\Spec(\C[\mu_\mathbb{L}^{-1}(0)]^{\overline G})$ and consider the canonical map $L\to L_0$. In analogy with Corollary \ref{corollary enhancement D5 resolution with affine quotients}, the maps $j$ and $p$ nicely descend to the affine bow varieties:
    \begin{corollary}
    \label{Affine NS5 resolutions}
        There exists a commutative diagram 
        \[
        \begin{tikzcd}
            \overline X \arrow[d, swap,  "\bar\pi"]& L\arrow[d] \arrow[l, swap, "j"] \arrow[r, "p"] & X \arrow[d, "\pi"] \\
            \overline X_0 &  L_0 \arrow[l] \arrow[r, equal] & X_0.
        \end{tikzcd}
        \]
        Moreover, the left square is Cartesian.
    \end{corollary}
    \begin{proof}
    Since $p$ and $j$ are defined at the level of quiver representations, they induce maps $\overline X_0\leftarrow L_0\to X_0$ making the two squares commute. The fact that $L_0\to X_0$ is an isomorphism follows from \cite[Section 2]{KraftProcesi}.
    It remains to check that the left square is Cartesian.
    As in the proof of Proposition \ref{proposition NS5 resolution bow varieites}, we can assume that all the D5 local charges of $X$ are strictly positive. By resolving all D5 branes like in diagram \eqref{diagram proof maps NS5 resolutions} and applying Corollary \ref{corollary enhancement D5 resolution with affine quotients}, we can reduce ourselves to the case $X=T^*\Fl$. In this case, the nilpotent orbit closure $Z$ resolved by the Springer resolution $T^*\Fl\to Z$ is just the image of $\pi: X=T^*\Fl\to X_0$\footnote {Although $X_0$ is itself a nilpotent orbit closure, the map $\pi:X\to X_0$ may not be surjective \cite{shmelkin2009remarks}. Hence, $\pi(X)\subsetneq X_0$ in general.}; hence the claim follows from the Cartesian diagram \eqref{cartesian diagram NS5 res. flag variety case}.
    \end{proof}


\subsection{The equivariant geometry of NS5 resolutions}

In Proposition \ref{proposition NS5 resolution bow varieites} we identified the fibers of the map $\pi:L \to X$ with Grassmannians. We begin this section by refining our analysis of these fibers. For a given separated bow variety $X$ and NS5 resolution $\overline X$ such that $\w=\w'+\w''$, we set $Y:=\ttt{\fs$\w'$\fs$\w'+ \w''$\bs\dots \bs 3\bs2\bs1\bs}$. If instead $X$ is co-separated, we set $Y=\ttt{\bs1\bs2\bs3\bs\dots\bs$\w'+\w''${\fs}$\w''${\fs}}$. The following Lemma can be thought of as the NS5 version of Lemma \ref{lemma Fixed components D5 res}.

\begin{lemma} $ $
\label{lemma points in lagrangian variety Z resolution}
\begin{enumerate}
    \item Let $f\in X^{\At}$. The $\Cs_{\h}$-fixed locus $Y^{\h}$ of the bow variety $Y$ fits in the following pullback diagram:
    \begin{equation*}
    \label{fiber of NS5 resolution fixed points}
        \begin{tikzcd}
            Y^{\h}\arrow[r, hookrightarrow]\arrow[d] & L\arrow[d, "p"]\\
            \lbrace f\rbrace \arrow[r, hookrightarrow] &X.
        \end{tikzcd}
    \end{equation*}
    \item Any fixed point $f\in X^{\At}$ admits exactly $\binom{\w'+\w''}{\w''}$ resolutions, i.e. fixed points $\bar f \in \overline X^{\At}\cap L$ such that $p(\bar{f})=f$.
    \item For any fixed point $f\in X^{\At}$, we have $T\overline X|_{Y^{\hbar}}-TL|_{Y^{\hbar}}=TY|_{Y^{\hbar}}-TY^{\h}$ in $K_{\Tt}(Y^{\h})$.
\end{enumerate}
\end{lemma}

\begin{proof}
     The first point follows from the quiver descriptions of $L$, $X$, and of the fixed points $f\in X^{\At}$ in \cite[Section 4.3]{rimanyi2020bow}.  Alternatively, it follows from Proposition \ref{proposition NS5 resolution bow varieites} together with the fact that $Y$ is exactly the bow variety description of $T^*\Gr{\w'}{ \w'+\w''}$ (or $T^*\Gr{\w''}{ \w'+\w''}$ in the coseparated case), so its $\Cs_{\hbar}$ fixed locus is the sero section.
     
     Consider the Grassmannian fibration $p:L\to X$. Since $p$ is equivariant, the fiber over an arbitrary $\At$-fixed point $f\in X^{\At}$ is preserved by the $\At$-action and a subtorus of rank $\w'+\w''$ acts on it non-trivially. Under the identification $p^{-1}(f)= \Gr{\w'}{\w'+\w''}$ (or $p^{-1}(f)= \Gr{\w''}{\w'+\w''}$ in the co-separated case), this action coincides with the standard action of the maximal torus of $\GL(\w'+\w'')$ on the framing. As a consequence, the number of fixed points in $L$ over $f\in X^{\At}$ is equal to $\binom{\w'+\w''}{\w''}$.
     
     We now prove the last point. All the computations will be implicitly in the ring $K_{\Tt}(Y^{\h})$. By the first point of the lemma, we have $TY^{\h}=TL-TX$. Since the symplectic form of $Y$ is rescaled by $\h$, we have $TY-TY^{\h}=\h(TY^{\h})^\vee$, hence $TY-TY^\hbar= \hbar(TL-TX)^\vee$. Thus, it suffices to prove that 
    \begin{equation}
    \label{equation for corollaries NS5 resolutions}
        T\overline X-TL= \hbar(TL-TX)^\vee.
    \end{equation}
    Graphically, we have 
    \[
    TL-TX=\left(
    \begin{tikzcd}
    \zeta_1  \arrow[rr, bend right]&  \zeta_2 \arrow[l, bend right] & \zeta_3   \arrow[l, bend right, swap, "\circ" marking]
    \end{tikzcd}
    -
    \begin{tikzcd}
    \zeta_2  \arrow[loop left]
    \end{tikzcd}
    \right)
    -
    \begin{tikzcd}
    \zeta_1  \arrow[r, bend right, swap]& \zeta_3 \arrow[l, bend right, swap, "\circ" marking]
    \end{tikzcd}
    \]
    and 
    \[
    T\overline X-TL=
    \left(
    \begin{tikzcd}
    \h^{-1}\zeta_1  \arrow[r, bend right]&  \zeta_2  \arrow[r, bend right, swap, ] \arrow[l, bend right, "\circ" marking] & \zeta_3   \arrow[l, bend right, swap, "\circ" marking]
    \end{tikzcd}
    -
    \begin{tikzcd}
    \zeta_2  \arrow[loop left, "\circ" marking]\arrow[loop right]
    \end{tikzcd}
    \right)
    -
    \left(
    \begin{tikzcd}
    \zeta_1  \arrow[rr, bend right]&  \zeta_2 \arrow[l, bend right] & \zeta_3   \arrow[l, bend right, swap, "\circ" marking]
    \end{tikzcd}
    -
    \begin{tikzcd}
    \zeta_2  \arrow[loop left]
    \end{tikzcd}
    \right).
    \]
    The shift of $\zeta_1$ by $\h^{-1}$ follows from Remark \ref{remark equivariance NS5 resolutions}. Recalling that the circles stand for action by $\h$ and that $(-)^\vee$ reverses the arrows, checking \eqref{equation for corollaries NS5 resolutions} becomes a straightforward computation.
\end{proof}

\begin{remark}
\label{remark torus action on fibers of NS5 resolutions}
    The action of $\Tt$ on $Y^{\h}$ induced by diagram \eqref{fiber of NS5 resolution fixed points} is a twist of the standard action on $Y$ seen as a bow variety. Indeed, if $X$ is separated (resp. co-separated), then $\Tt$ acts on the $i$-th brane of $Y$ with weight equal to $a_i h^{-\gamma_i}$ (resp. $a_i h^{\gamma_i}(f)$), where $\gamma_i(f)$ is equal to the number of ties in the tie diagram of $f\in X^{\At}$ that are connected to the NS5 branes left to the resolved NS5 brane $\Zb$. This again follows from the explicit description of $f\in X^{\At}$ given in \cite[Section 4.3]{rimanyi2020bow}.
\end{remark}

\begin{remark}
\label{remark resolution of fixed points NS5 resolutions}
By part (2) of Lemma \ref{lemma points in lagrangian variety Z resolution}, the resolutions of $f\in X^A$ are in one-to-one correspondence with fixed points in $Y$. In terms of tie diagrams, the correspondence sends the tie diagram of $\bar f\in \overline X$ to the tie diagram of $Y=\ttt{\fs$\w'$\fs$\w'+ \w''$\bs\dots \bs 3\bs2\bs1\bs}$ (resp. $Y=\ttt{\bs1\bs2\bs3\bs\dots\bs$\w'+\w''${\fs}$\w''${\fs}}$) obtained by erasing all the branes and ties not connected to $\Zb'$ or $\Zb''$. By slightly abusing notation, we will still denote the resulting fixed point in  $Y$ by $\bar f$.
\end{remark}

In the next lemma, we compare the full attracting sets of $X$, $L$ and $\overline X$.

\begin{lemma}
\label{lemma supports NS5 resolutions}
Let $f\in X^{\At}$ and let $\bar f\in \overline X^{\At}$ be any resolution of $f$. Then
\begin{enumerate}
    \item 
    \[
        p\left(\Attfull{C}^{\overline X}(\bar f)\cap (L\times \bar f)\right)\subseteq \Attfull{C}^{X}(f).
    \]
    \item Let $\bar f'\in \overline X^{\At}$ and assume that $\bar f'\in \Attfull{C}^{\overline X}(\bar f)$. Then $\bar f'\in L^{\At}\subseteq \overline X^{\At}$.
\end{enumerate}
\end{lemma}
\begin{proof}
    Since the map $p$ is $\At$-equivariant, to prove the first claim of the lemma it suffices to show that $\Attfull{C}^{\overline X}(\bar f)\cap (L\times \bar f)\subseteq \Attfull{C}^{L}(\bar f)$. But this follows from the same argument of Lemma \ref{lemma equality attracting sets bow varieties}, with Corollary \ref{corollary enhancement D5 resolution with affine quotients} replaced by Corollary \ref{Affine NS5 resolutions}. The second claim follows from a completely analogous argument.
\end{proof}

\subsection{NS5 fusion of stable envelopes}

As stated in Lemma \ref{lemma points in lagrangian variety Z resolution}, an arbitrary fixed point $f\in X^{\At}$ admits exactly $\binom{\w'+\w''}{\w'}$ possible resolutions $\bar f\in \overline X^{\At}$. Among these, we denote by $\bar f_\sharp$ the smallest one with respect to the order determined by the chamber $\mathfrak{C}$. In this section, we relate the stable envelope $\StabX{X}{C}{f}$ of $f$ in $X$ to the stable envelope $\StabX{\overline X}{\overline C}{\bar f_\sharp}$ of $\bar f_\sharp$ in $\overline X$ via pull-push along the correspondence \eqref{correspondence NS5 resolution}. According to the general philosophy of elliptic stable envelopes, the equality holds up to a shift of the K\"ahler parameters, which we now introduce. Consider the morphism $\psi: \Tt^!\to \overline \Tt^!$ that is the identity on most components except
\[
    \Cs_{\Zb} \times \Cs_{\h}  \to  \Cs_{\Zb'} \times  \Cs_{\Zb''} \times \Cs_{\h} \qquad (z,\h)\mapsto \begin{cases}
        (z, z\h^{-\w'}, \h ) & \text{if $\D$ is separated}\\
        (z\h^{\w''}, z, \h ) & \text{if $\D$ is co-separated.}
    \end{cases}
\]
Following our convention, we denote the induced embedding $\psi: E_{\Tt^!}\to E_{\overline \Tt^!}$ in the same way.
\begin{lemma}
\label{lemma comparison line bundles Z resolution}
Let $N_p$ be the virtual normal bundle of $p: L\to X$. 
Then we have 
\[
p^* \left(\Sh{L}^\triangledown_{\At,f}\boxtimes \Sh{L}_X \right)\otimes \Theta(-N_p) = j^* \psi^* \left(\Sh{L}^\triangledown_{\At,\bar f_\sharp}\boxtimes \Sh{L}_{\overline{X}}\right)
\]
for every $f\in X^{\At}$.
\end{lemma}
We defer the proof of this lemma to the end of this section. It implies that the following chain of maps of $\mathcal{O}_{\mathscr{B}_{\Tt,\Tt^!}}$-modules is well defined:
\[
\begin{tikzcd}
   p_*  \psi^*\left( \Sh{L}^\triangledown_{\At,\bar f_\sharp}\boxtimes \Sh{L}_{\overline{X}}\right) \arrow[r, "j^{\oast}"]  &  p_*j^*  \psi^*\left( \Sh{L}^\triangledown_{\At,\bar f_\sharp}\boxtimes \Sh{L}_{\overline{X}}\right) \arrow[r, "p_{\oast}"]&  \Sh{L}^\triangledown_{\At,f}\boxtimes \Sh{L}_X .
\end{tikzcd}
\]
We now state the main result of this section.

\begin{theorem}
\label{main theorem proof NS5 resolution for stabs}

Let $f\in X^{\At}$ and let $\bar f_\sharp\in p^{-1}(f)^{\At}$ be the minimal fixed point with respect to the attracting order determined by $\mathfrak{C}$. Then the following formula holds
\[
p_{\oast} (j^{\oast}\psi^{\oast} \Stab{C}{\bar f_\sharp})=\Stab{C}{f}.
\]
\end{theorem}

\begin{remark} The maps $j$ and $p$ define a closed embedding $L\hookrightarrow  \overline X\times X$ and it is easy to check that $L\subseteq  \overline X\times X$ is Lagrangian with symplectic form $\bar \omega-\omega$. Indeed, checking that $\dim(L)=1/2(\dim(\overline X)+\dim(X))$ is straightforward from our quiver description and isotropy follows from an easy computation using the fact that the symplectic form of a two-way part like the one left-hand side of \eqref{equation separated quiver modification for in-between space} is given by $\text{tr}(dD\wedge dC)$. 
    
    From this symplectic point of view, the cohomological limit of Proposition \ref{proposition NS5 resolution bow varieites} can be nicely interpreted in terms of Lagrangian correspondences. Indeed, the cohomological stable envelope $\Stab{C}{f}$ can be seen as a Lagrangian cycle $[\Stab{C}{f}]\in H_T^{\text{BM}} (f\times X)$ and the cohomological limit of Proposition \ref{proposition NS5 resolution bow varieites} says that
    \begin{equation}
        \label{convolution interpretation of Z resolution}
        [\Stab{C}{f}]= [L]\ast [\Stab{C}{\bar f_\sharp}]
    \end{equation}
    upon identifying $\bar f_\sharp$ with $f$. The operation on the right-hand side is the convolution product in Borel-Moore homology, cf. Section 3.2.5 of \cite[Section 3.2.5]{maulik2012quantum}.
\end{remark}
Note that the image of $\psi$ is the divisor $\{z'/z''\hbar^{ \w'}=1\}$ (or $\{z'/z''\hbar^{ -\w''}\}$ in the co-separated case). On the other hand, the stable envelopes $\Stab{C}{\bar f}$ are meromorphic sections with poles on $\{z_i/z_j\hbar^{\alpha}=1\}$ for certain $\alpha$. Hence, a preliminary step towards the proof of Theorem \ref{main theorem proof NS5 resolution for stabs}, is to show that $\psi^{\oast} \Stab{C}{\bar f_\sharp}$ is well-defined. In the proof, we will see that the absence of such poles relies on our distinguished choice of fixed point $\bar f_\sharp\in p^{-1}(f)^A$ in an essential way. 
\begin{lemma}
\label{lemma: pullback NS5 is well defined}
    Let $\overline\Delta$ be the resonant locus of $\Stab{C}{\bar f_\sharp}$. Then $\Image(\psi)\setminus\left(\Image(\psi)\cap \overline\Delta\right)$ is open dense in $\Image(\psi)$. Hence, the pullback $\psi^{\oast} \Stab{C}{\bar f_\sharp}$ is well-defined on the pre-image of this open dense set.
\end{lemma}
The proof of this lemma is also deferred to the end of this section.

\begin{proof}[Proof of Theorem \ref{main theorem proof NS5 resolution for stabs}]
By the previous lemma, the pullback $\psi^{\oast} \Stab{C}{\bar f_\sharp}$ is defined on open dense set of $E_{\Tt^!}$, and, by definition, the same is true for $\Stab{C}{f}$ (although on an a priori different open dense set). By uniqueness of the stable envelopes, which is local over $E_{\Tt^!}$, it suffices to check the statement on the intersection of these two open dense sets.

By Lemma \ref{lemma comparison line bundles Z resolution}, $p_{\oast} (j^{\oast} \psi^{\oast} \Stab{C}{\bar f_\sharp})$ is a section of the correct line bundle so it suffices to check the axioms of Section \ref{section Stable envelopes: basics}. Recall the notation introduced in \eqref{equation spaces in the definition of stable envelopes} and consider the diagram 
\[
\begin{tikzcd}
    \overline{X}^{> \bar f_\sharp}\arrow[d, hookrightarrow, "\bar i_>"] & p^{-1}( X^{> f})\arrow[l, dashed, hookrightarrow, swap, "j_>"]\arrow[d, hookrightarrow, "k_>"] \arrow[r, "p^>"]& X^{> f} \arrow[d, hookrightarrow, "i_>"] \\
    \overline{X}& \arrow[l, swap, hookrightarrow, "j"]L\arrow[r, "p"] & X
\end{tikzcd}
\]
in which the right square is simply the pullback. We first prove the support axiom, i.e. that 
\[
i_>^{\oast}p_{\oast} (j^{\oast}\psi^{\oast} \Stab{C}{\bar f_\sharp})=0.
\]
Assume temporarily that $p^{-1}(X^{> f})\subseteq \overline{X}^{> \bar f_\sharp}$, i.e. that the dashed inclusion in the diagram exists. Then compatibility of pushforward and pullback in Cartesian squares implies that
\begin{equation*}
i_>^{\oast}p_{\oast} (j^{\oast}\psi^{\oast} \Stab{C}{\bar f_\sharp})=p^>_{\oast} k_>^{\oast} (j^{\oast}\psi^{\oast} \Stab{C}{\bar f_\sharp})=
p^>_{\oast} j_>^{\oast} \bar i_>^{\oast} \psi^{\oast} \Stab{C}{\bar f_\sharp})=0.
\end{equation*}
In the last step, we used the support axiom for $ \Stab{C}{\bar f_\sharp}$. It remains to show that $p^{-1}(X^{> f})\subseteq \overline{X}^{> \bar f_\sharp}$ or, equivalently, that 
\begin{equation}
    \label{equaiton in proof support axoim Z resolution}
    p\left(\Attfull{C}^{\overline X}(\bar f_\sharp)\cap L\right)\subseteq \Attfull{C}^{X}(f).
\end{equation}
But this is exactly the first point of Lemma \ref{lemma supports NS5 resolutions}.

Consider now the commutative diagram
\[
\begin{tikzcd}
    \Att{C}^{\overline X}(\bar f_\sharp )\arrow[d, hookrightarrow, "l^{\overline{X}}"]  & \Att{C}^{L}(\bar f_\sharp) \arrow[l, hookrightarrow]  \arrow[d, hookrightarrow, "l^L"]\arrow[r, "p^{\text{Att}}"] & \Att{C}^X(f) \arrow[d, hookrightarrow, "l^X"]
    \\
    \overline{X}^{\geq  \bar f_\sharp}\arrow[d, hookrightarrow, "\bar i_\geq"] & p^{-1}( X^{\geq  f})\arrow[l, dashed, hookrightarrow, swap, "j_\geq"]\arrow[d, hookrightarrow, "k_\geq"] \arrow[r, "p^\geq"]& X^{\geq f} \arrow[d, hookrightarrow, "i_\geq"]
    \\
    \overline{X}& \arrow[l, swap, "j"]L\arrow[r, "p"] & X.
\end{tikzcd}
\]
It remains to prove the diagonal axiom, i.e. that 
\begin{equation}
    \label{diagonal axiom proof NS5 resolution}
    i_\geq^{\oast}p_{\oast} (j^{\oast}\psi^{\oast} \Stab{C}{\bar f_\sharp})= [\Att{C}^X(f)].
\end{equation}
The same argument used for $j^>$ shows that the dashed map $j^\geq$ is well defined. From the second point of Lemma \ref{lemma supports NS5 resolutions}, it follows that $\Att{C}^{L}(\bar f_\sharp)=\Att{C}^{\overline X}(\bar f_\sharp )\cap L$. Hence, the top left square is Cartesian. Clearly, also the bottom right square is Cartesian.
We now claim that $p^{\text{Att}}$ is an isomorphism. It suffices to work locally at $\bar f_\sharp$ and show that the map
\[
Dp^{\text{Att}}: T_{\bar f_\sharp}\Att{C}^{L}(\bar f_\sharp)\to T_f\Att{C}^X(f)\to 0
\]
is an isomorphism.
Since by assumption $\bar f_\sharp$ is $\mathfrak{C}$-minimal among the fixed points in the fiber $p^{-1}(f)\cong \Gr{\w'}{ \w'+\w''}$ (or $p^{-1}(f)\cong \Gr{\w''}{ \w'+\w''}$ in the co-separated case), it follows that $T_{\bar f_\sharp} (p^{-1}(f))$ has no positive weights, and hence
\[
\dim(T_{\bar f_\sharp}\Att{C}^{L}(\bar f_\sharp))=\dim(T_f\Att{C}^X(f)).
\]
This proves the claim. We can now prove the diagonal axiom. As before, we have
\begin{equation*}
i_\geq ^{\oast}p_{\oast} (j^{\oast}\psi^{\oast} \Stab{C}{\bar f_\sharp})
=p^\geq_{\oast} k_\geq^{\oast} (j^{\oast}\psi^{\oast} \Stab{C}{\bar f_\sharp})
\\
=p^\geq_{\oast} j_\geq^{\oast} \bar i_\geq^{\oast} \psi^{\oast} \Stab{C}{\bar f_\sharp}.
\end{equation*}
By the diagonal axiom for $\Stab{C}{\bar f_\sharp}$, we get 
\[
p^\geq_{\oast} j_\geq^{\oast}\bar i_\geq^{\oast} \psi^{\oast} \Stab{C}{\bar f_\sharp}
=p^\geq_{\oast} j_\geq^{\oast}[\Att{C}^{\overline X}(\bar f_\sharp)].
\]
Now recall that $[\Att{C}^{\overline X}(\bar f_\sharp)]= l^{\overline X}_{\oast} (1)$ and the top left diagram is Cartesian, hence 
\[
 p^\geq_{\oast} j_\geq^{\oast}[\Att{C}^{\overline X}(\bar f_\sharp)]
 = p^\geq_{\oast} j_\geq^{\oast} l^{\overline X}_{\oast} (1) =
 p^\geq_{\oast} l^{L}_{\oast} (1) = l^X_{\oast}p^{\text{Att}}_{\oast} (1)
 = l^X_{\oast}(1)=[\Att{C}^X(f)].
\]
In the penultimate we used the fact that $p^{\text{Att}}$ is an isomoprhism, so $p^{\text{Att}}_{\oast}=id$. Combining the three lines of equations above, we deduce \eqref{diagonal axiom proof NS5 resolution}. This completes the proof.
\end{proof}

\begin{remark}
Assume for simplicity that $X$ is separated. The localization formula  (Proposition \ref{localization formula pushforward}) implies that
\[
\label{localization formula Z resolution partial flags}
    p_{\oast} (j^{\oast} \psi^{\oast} \Stab{C}{\bar f_\sharp})\Big|_f = \sum_{\bar f} \frac{\Stab{C}{\bar f_{\sharp}}\Big|_{\bar f}}{\vartheta(N_p)\Big|_{\bar f}} (z'=z, z''=z\h^{-\w'}).
\]
Since $\bar f_\sharp$ is $\mathfrak{C}$-minimal, only the diagonal term on the right-hand side is nonzero; hence we get 
\[
p_{\oast} (j^{\oast}\psi^{\oast} \Stab{C}{\bar f_\sharp})\Big|_f= \frac{\Stab{C}{\bar f_{\sharp}}\Big|_{\bar f_\sharp}}{\vartheta(N_p)\Big|_{\bar f_\sharp}}(z'=z, z''=z\h^{-\w'})= \frac{\vartheta(N^-_{\bar f_\sharp})}{\vartheta(N_{p,\bar f\sharp})}.
\]
Again by minimality of $\bar f_\sharp$, we deduce that $N^-_{\bar f_\sharp}= N_{p, \bar f\sharp}\oplus N^-_f$ and hence 
\[
p_{\oast} (j^{\oast} \psi^{\oast} \Stab{C}{\bar f_\sharp})\Big|_f =\frac{\vartheta(N^-_{\bar f_\sharp})}{\vartheta(N_{p, \bar f\sharp})} =\vartheta( N^-_f),
\]
as required by Remark \ref{Remark axiomatic stab}.
\end{remark}


We conclude this section with the proof of Lemma \ref{lemma comparison line bundles Z resolution}.
\begin{proof}[Proof of Lemma \ref{lemma comparison line bundles Z resolution}]
As usual, we prove the theorem assuming that $X$ is separated since the co-separated case is completely analogous.
We need to show that
\[
j^* \psi^* \left(\Sh{L}^\triangledown_{\At,\bar f_\sharp}\boxtimes \Sh{L}_{\overline{X}}\right) \otimes p^* \left(\Sh{L}^\triangledown_{\At,f}\boxtimes \Sh{L}_X \right)^{-1}= \Theta(-N_p) .
\] 
Recall that for any bow variety $X$ and isolated fixed point $i: \lbrace f \rbrace \hookrightarrow X $, we have 
\[
\Sh{L}_{X}=\Theta(\alpha_X)\otimes \U_X \qquad \Sh{L}_{\At,f}^\triangledown = i^*\Sh{L}_{X}^{-1}\otimes \Theta(-N^-_{f/X})
\]
and the dependence on the K\"ahler parameters is only encoded in $\U_X$ and $ \U_{\overline X}$. Since by assumption $\bar f_\sharp$ is smallest in the attracting order determined by the chamber $\mathfrak{C}$, it follows that
\[
N^-_{\bar f_\sharp/\overline X}-N^-_{f/X}=N_p|_{\bar f_\sharp}.
\]
Hence, to prove the claim, it suffices to show that 
\begin{align}
    \label{first equation to prove in NS5 resolution of stabs}
    &p^*\Theta(\alpha_X)^{-1}\otimes j^*\Theta(\alpha_{\overline X})= \Theta(-N_p)\otimes \Sh{G}_1
    \\
    \label{second equation to prove in NS5 resolution of stabs}
    &p^*\U_X^{-1}\otimes \psi^* j^* \U_{\overline X}= \Sh{G}_2,
\end{align}
where $\Sh{G}_1$ and $\Sh{G}_2$ are line bundles pulled back from the base $\Base{\Tt}{\Tt^!}$. All the computations will be local around the resolution, so we relabel the various ingredients for convenience:
\[
\begin{tikzpicture}[baseline=0,scale=.45]
\draw [thick,red] (0.6,0) --(1.4,2); 
\draw [thick,red](3.6,0) --(4.4,2);  
\draw [thick,red](6.6,0) -- (7.4,2); 
\draw [thick,red](9.6,0) -- (10.4,2); 
\draw [thick,black] (1,1) --(10,1);
\draw [thick, dashed, black] (-1,1) --(1,1);
\draw [thick, dashed, black] (10,1) --(12,1);
\node at (-2.5,1.2) {$\overline X=$};
\node at (1,2.6) {$\w^-$};
\node at (4,2.5) {$\w'$};
\node at (7,2.5) {$\w''$};
\node at (10,2.6) {$\w^+$};
\node at (1,-.5) {$z^-$};
\node at (4,-.5) {$z'$};
\node at (7,-.5) {$z''$};
\node at (10,-.5) {$z^+$};
\end{tikzpicture}
\qquad \qquad  
\begin{tikzpicture}[baseline=0,scale=.45]
\draw [thick,red] (0.6,0) --(1.4,2); 
\draw [thick,red](3.6,0) --(4.4,2);  
\draw [thick,red](6.6,0) -- (7.4,2);  
\draw [thick,black] (1,1) --(7,1);
\draw [thick, dashed, black] (-1,1) --(1,1);
\draw [thick, dashed, black] (7,1) --(9,1);
\node at (-2.5,1.2) {$X=$};
\node at (1,2.6) {$\w^-$};
\node at (4,2.5) {$\w$};
\node at (7,2.6) {$\w^+$};
\node at (1,-.5) {$z^-$};
\node at (4,-.5) {$z$};
\node at (7,-.5) {$z^+$};
\end{tikzpicture}.
\]
We now analyze the virtual normal bundle $N_p$ to the map $p: L\mapsto X$.
By \eqref{equation separated quiver modification for in-between space} and the definition of $p$, we obtain the following diagrammatic description for $-N_p$:
\[
-N_p=TL-TX=
\left(
\begin{tikzcd}
    \zeta_1  \arrow[rr, bend right, swap, ]&  \zeta_2 \arrow[l, bend right, swap] & \zeta_3   \arrow[l, bend right, swap, "\circ" marking]
\end{tikzcd}
-
\begin{tikzcd}
    \zeta_2  \arrow[loop left]
\end{tikzcd}
\right)
-
\begin{tikzcd}
    \zeta_1  \arrow[r, bend right, swap]& \zeta_3 \arrow[l, bend right, swap, "\circ" marking]
\end{tikzcd}.
\]
Here, $\zeta_1$, $\zeta_2$, and $\zeta_3$ are the three tautological bundles on the vertices of the right-hand side of \eqref{equation separated quiver modification for in-between space}, and the decoration of the arrows by a circle indicates a weight one action of $\h$.
Similarly, we have
\[
j^*\alpha_{\overline X}-p^*\alpha_X=
\left(
\begin{tikzcd}
    \h^{-1}\zeta_1 &  \zeta_2 \arrow[l, bend right, swap, "\circ" marking] & \zeta_3  \arrow[l, bend right, swap, "\circ" marking]
\end{tikzcd}
-
\begin{tikzcd}
    \zeta_2  \arrow[loop left, "\circ" marking]
\end{tikzcd}
\right)
-
\begin{tikzcd}
    \zeta_1 & \zeta_3. \arrow[l, bend right, swap, "\circ" marking] 
\end{tikzcd}
\]
The twist of $\zeta_1$ by $\h^{-1}$ follows from Remark \ref{remark equivariance NS5 resolutions}.
Overall, it follows that 
\[
j^*\alpha_{\overline X}-p^*\alpha_X=-N_p +\left(\hbar\End(\xi_2)-\End(\xi_2)\right).
\]
But then Lemma \ref{Lemma properties U bundle} implies \eqref{first equation to prove in NS5 resolution of stabs} with $\Sh{G}_1=\Theta(\hbar\End(\xi_2)-\End(\xi_2))$. 

We now check \eqref{second equation to prove in NS5 resolution of stabs}. By definition,
\[
p^* \U_{X}= \dots
\otimes
\Sh{U}\left(\zeta_1, \frac{z^-}{z} \hbar^{-\w^-}\right)
\otimes
\Sh{U}\left(\zeta_3, \frac{z}{z^+} \hbar^{-\w}\right)
\otimes
\dots
\]
where the dots stand for the other terms not depending on $z$. On the other hand
\[
j^* \U_{\overline X}= 
\dots
\otimes
\Sh{U}\left(\zeta_1\otimes \h^{-1}, \frac{z^-}{z'} \hbar^{-\w^-}\right)
\otimes
\Sh{U}\left(\zeta_2, \frac{z'}{z''}  \hbar^{-\w'}\right)
\otimes
\Sh{U}\left(\zeta_3, \frac{z''}{z^+} \hbar^{-\w''}\right)
\otimes
\dots
\]
where now the dots stand for the other terms not depending on $z'$ and $z''$. The twist of $\zeta_1$ by $\h^{-1}$ has the same explanation as before.
Applying $\psi^*$, i.e. setting $z'=z$ and $z''=z\hbar^{-\w'}$ and using the identity $\w=\w'+\w''$, we get
\begin{align*}
\psi^* j^* \U_{\overline X}
&= 
\dots
\otimes
\Sh{U}\left(\zeta_1\otimes \h^{-1}, \frac{z^-}{z} \hbar^{-\w^-}\right)
\otimes
\Sh{U}\left(\zeta_2, 1 \right)
\otimes
\Sh{U}\left(\zeta_3, \frac{z}{z^+} \hbar^{-\w}\right)
\otimes
\dots
\\
&=
\dots
\otimes
\Sh{U}\left(\zeta_1\otimes \h^{-1}, \frac{z^-}{z} \hbar^{-\w^-}\right)
\otimes
\Sh{U}\left(\zeta_3, \frac{z}{z^+} \hbar^{-\w}\right)
\otimes
\dots .
\end{align*}
In the second step, we have used the fact that $\Sh{U}(V,1)\cong \Sh{O}$ for all $V$.
Therefore, applying Lemma~\ref{Lemma properties U bundle}, we get 
\begin{equation}
\label{eq: difference U bundles NS5}
p^*\U_X^{-1}\otimes \psi^* j^* \U_{\overline X}
=
\dots
\otimes
\Sh{U}\left(\h^{-1}, \frac{z^-}{z} \hbar^{-\w^-}\right)^{\otimes \rk(\zeta_1)}
\otimes
\dots .
\end{equation}
Notice that $\Sh{U}\left(\h^{-1}, \frac{z^+}{z} \hbar^{-\w^-}\right)$ only depends on the K\"ahler variables and $\hbar$, hence it is pulled back from the base $\Base{\Tt}{\Tt^!}$. Similarly, Remark \ref{remark equivariance NS5 resolutions} implies that the dots on the right are trivial bundles while those on the left are products of terms of the form $\Sh{U}(\h^{-1}, z_i/z_{i+1}\hbar^{-\w(\Zb_{i})})^{\otimes l_i}$ for some $l_i\in\N$. Hence the latter are also pulled back from $\Base{\Tt}{\Tt^!}$. This completes the proof of \eqref{second equation to prove in NS5 resolution of stabs} (with $\Sh{G}_2$ given by the product of all these pullbacks) and hence of the lemma. The co-separated case is completely analogous and even slightly easier because of the absence of the twists by $\h^{-1}$ of certain tautological bundles like the bundle $\zeta_1$ above.
\end{proof}
\begin{proof}[Proof of Lemma \ref{lemma: pullback NS5 is well defined}]
As usual, we give the proof assuming that $X$ is separated. Let $\bar f'\in \overline X^{\At}$ and consider the line bundle $\Sh{L}_{A,\bar f'}\boxtimes (\Sh{L}_{A,\bar f_{\sharp}})^{-1}$ .
Since the fixed locus $\overline X^{\At}$ is zero dimensional, this line bundle lives on the abelian variety 
\[
E_{\At\times \overline\Tt^!}(\bar f'\times \bar f_\sharp)=E_{\At\times \overline{\Tt}^!}.
\]
Let $\bar\phi: E_{\At\times \overline\Tt^!}\to E_{\overline \Tt^!}$ be the natural projection.
By definition, the resonant locus $\overline \Delta$ consists of those points $\bar s\in E_{\overline \Tt^!/\At}$ with non-trivial stalk
\begin{equation*}
    \left(R^\bullet \bar\phi_*\left(\Sh{L}_{\At,\bar f'}\boxtimes (\Sh{L}_{\At,\bar f_{\sharp}})^{-1}\right)\right)_{\bar s}\neq 0
\end{equation*}
for some fixed point $\bar f'\in \overline X^{\At}$ such that $\bar f'<\bar f_\sharp$. Since $\bar f_{\sharp}\in L^{\At} $, by the second point of Lemma \ref{lemma supports NS5 resolutions} we deduce that $\bar f'\in L^{\At}$. To prove the lemma, we need to show that 
\begin{equation*}
\label{eq: resonant condition}
\left(R^\bullet \bar\phi_*\left(\Sh{L}_{\At,\bar f'}\boxtimes (\Sh{L}_{\At,\bar f_{\sharp}})^{-1}\right)\right)_{\bar s}= 0
\end{equation*}
for generic $\bar s\in \Image(\psi)$. 
 Let $s$ be the generic point in $E_{\Tt^!}$. Hence, $\psi(s)$ is the generic point of $\Image{\psi}\subset E_{\Tt^!}$. Let $\var{E}_{\psi(s)}$ be the fiber of $\bar{\phi}$ over $\psi(s)$. By \cite[Section 2.3.6]{okounkov2020inductiveI}, it suffices to show that 
 \[
 \Sh{L}_{\At,\bar f'}\boxtimes (\Sh{L}_{\At,\bar f _{\sharp}})^{-1}|_{\psi(s)}\neq  \Sh{O}_{\var{E}_{\psi(s)}}.
 \]
This will follow from our specific choice of representative $\bar f_{\sharp}$. Let $f'=p(\bar f')$. Since $p: L\to X$ is equivariant, the condition $\bar f'<\bar f_\sharp$ implies that there is a chain of attracting curves in $X$ connecting $f$ and $f'$. Moreover, our definition of $\bar f_{\sharp}$ as the minimal fixed point in the fiber $p^{-1}(f)$ forces $f'\neq f$. Combining the last two observations we get $f'<f$.

Recall that $s=(z,\hbar)\in E_{\Tt^!}$ is generic. Let $\var{E}_s$ be the fiber of  $E_{\At\times \Tt^!}\to E_{\Tt^!}$ at $s$. We have a natural identification $\var{E}_s=\var{E}_{\psi(s)}$.
Equation \eqref{eq: difference U bundles NS5} implies that the two line bundles
\begin{equation}
\label{eq: two line bundles NS5 comparison}
    \Sh{L}_{\At, f'}\boxtimes (\Sh{L}_{\At,f})^{-1}|_s
\qquad 
\Sh{L}_{\At,\bar f'}\otimes (\Sh{L}_{\At,\bar f _{\sharp}})^{-1}|_{\psi(s)}
\end{equation}
are both equal to 
\begin{equation}
    \label{eq: common U bundle}
    i^*_{f'}\U_X\otimes (i^*_{f}\U_X)^{-1}|_s,
\end{equation}
up to a twist by a line bundle pulled back from $E_{\Tt}$. But  it is shown in \cite[Section 2.3.8]{okounkov2020inductiveI} that the condition $f'<f$ implies that \eqref{eq: common U bundle} is non-trivial for for generic $s=(z,\hbar)$. Therefore $\Sh{L}_{\At,\bar f'}\otimes (\Sh{L}_{\At,\bar f _{\sharp}})^{-1}|_{\psi(s)}$ is non-trivial too. 

Equivalently, in terms of quadratic forms: the quadratic forms characterizing the line bundles \eqref{eq: two line bundles NS5 comparison} are polynomials in the $a$-variables, but depend parametrically on $z$ and $\hbar$, and the parts with non-trivial dependence on $z$ are both equal to $z$-dependent piece of the quadratic form of \eqref{eq: common U bundle}. Since the condition $f'<f$ implies that the latter is non-trivial  for generic $z$, it follows that $\Sh{L}_{\At,\bar f'}\otimes (\Sh{L}_{\At,\bar f _{\sharp}})^{-1}|_{\psi(s)}$ is non-trivial too. The lemma is proved.

\end{proof}


\subsection{More explicit formulas}

\label{explicit NS5 resolutions}
Let $X$ be separated or co-separated and let $\overline X$ be the resolution of its $i$-th NS5 brane $\Zb$ with local charge decomposition $\w=\w'+\w''$. Localizing at fixed points and using the localization formula for the pushforward, we now produce explicit formulas relating the fixed point restrictions of the stable envelopes of $X$ and $\overline X$.


Recall that if $X$ is separated, we set $Y:=\ttt{\fs$\w'$\fs$\w'+ \w''$\bs\dots \bs 3\bs2\bs1\bs}$. If instead $X$ is co-separated, we set $Y=\ttt{\bs1\bs2\bs3\bs\dots\bs$\w'+\w''${\fs}$\w''${\fs}}$. 
As discussed in Lemma \ref{lemma points in lagrangian variety Z resolution}, fibers of $p:L\to X$ over fixed points are canonically isomorphic to $Y^{\h}$. As a consequence, whenever a fixed point $f\in X^{\At}$ is chosen, we can identify its resolutions $\bar f \in \overline{X}^{\At}\cap L$ with the $\At$-fixed points in $Y$ (or in $Y^{\hbar}$, since $Y^{\At}=Y^{\Tt}$). For an explicit description of this correspondence via tie diagrams, see Remark \ref{remark resolution of fixed points NS5 resolutions}.

\begin{proposition}
\label{proposition separated NS5 fusion ratios stabs}
Assume that $X$ is separated. Then
\begin{multline*}
  \frac{\StabX{X}{C}{f}\Big|_g}{\StabX{X}{C}{g}\Big|_g}(a, z, \h) \\
  = \sum_{\bar g\in Y^{\At}} \frac{\vartheta(N_{\bar g/Y}^-)}{\vartheta(N_{\bar g/Y^{\h}})} (a\h^{-\gamma(g)}) 
  \frac{\StabX{\overline X}{C}{\bar f_{\sharp}}\Big|_{\bar g}}{\StabX{\overline X}{C}{\bar g}\Big|_{\bar g}}
  (a, z_1,\dots,z'_{i}=z_i, z''_i=z_i\h^{-\w'}, \dots, z_m, \h).
\end{multline*}
Here $Y=\ttt{\fs$\w'$\fs$\w'+ \w''$\bs\dots \bs 3\bs2\bs1\bs}$ and $N_{\bar g/Y^{\h}}$ is the restriction to $\bar g$ of the tangent class 
$TY^{\h}\in K_{\Tt}(Y^{\h})$. The multi-index $\gamma(g)$ indicates the shift introduced in Remark \ref{remark torus action on fibers of NS5 resolutions}.
\end{proposition}

\begin{proposition}
\label{proposition co-separated NS5 fusion ratios stabs}
Assume that $X$ is co-separated. Then
\begin{multline*}
  \frac{\StabX{X}{C}{f}\Big|_g}{\StabX{X}{C}{g}\Big|_g}(a, z, \h) \\
  = \sum_{\bar g\in Y^{\At}} \frac{\vartheta(N_{\bar g/Y}^-)}{\vartheta(N_{\bar g/Y^{\h}})} (a\h^{\gamma(g)}) 
  \frac{\StabX{\overline X}{C}{\bar f_{\sharp}}\Big|_{\bar g}}{\StabX{\overline X}{C}{\bar g}\Big|_{\bar g}}
  (a, z_1,\dots,z'_{i}=z_i\h^{\w''}, z''_i=z_i, \dots, z_m, \h).
\end{multline*}
Here $Y=\ttt{\bs1\bs2\bs3\bs\dots\bs$\w'+\w''${\fs}$\w''${\fs}}$ and $N_{\bar g/Y^{\h}}$ is the restriction to $\bar g$ of the tangent class
$TY^{\h}\in K_{\Tt}(Y^{\h})$. The multi-index $\gamma(g)$ indicates the shift introduced in Remark \ref{remark torus action on fibers of NS5 resolutions}.
\end{proposition}

\begin{proof}
We only prove the separated case as the co-separated one is analogous. Applying the localization formula (Proposition \ref{localization formula pushforward}) to Theorem \ref{main theorem proof NS5 resolution for stabs}, we get
   \begin{multline*}
  \StabX{X}{C}{f}\Big|_g(a, z, \h) = \sum_{\bar g} \frac{1}{\vartheta(N_{\bar g/Y^{\h}})} \times 
  \\
  (a\h^{-\gamma(g)}) 
  \StabX{\overline X}{C}{\bar f_{\sharp}}\Big|_{\bar g}
  (a, z_1,\dots,z'_{i}=z_i, z''_i=z_i\h^{-\w'}, \dots, z_m, \h).
\end{multline*}
Thus, it remains to prove that 
\[
\vartheta(N^-_{\bar g/Y})\StabX{X}{C}{g}\Big|_g = \StabX{\overline X}{C}{\bar g}\Big|_{\bar g}.
\]
But this follows from point (3) of Lemma \ref{lemma points in lagrangian variety Z resolution} and the diagonal axiom of the stable envelopes (cf. Remark \ref{Remark axiomatic stab}).

\end{proof}

\begin{remark}
\label{remark distinguished resolutions fixed points NS5 case}
    The coefficients appearing in these formulas are ratios of topological classes associated with the bow varieties \ttt{\fs$\w'$\fs$\w'+ \w''$\bs\dots \bs 3\bs2\bs1\bs} and \ttt{\bs1\bs2\bs3\bs\dots\bs$\w'+\w''${\fs}$\w''${\fs}}, in the separated and co-separated case, respectively. 
    Remarkably, these varieties are dual to those appearing in the results of Section \ref{explicit D5 resolutions}. This is not a lucky coincidence, and the key step in the proof of mirror symmetry will be showing that the coefficients above are identified under mirror symmetry with the coefficients appearing in the D5 resolutions of stable envelopes from Section \ref{explicit D5 resolutions}.
    
    This duality between D5 and NS5 resolutions can also be detected at the level of the resolving fixed points. If $\mathfrak{C}$ is the standard chamber $\mathfrak{C}=\lbrace a_1<a_2<\dots<a_n\rbrace$, then the fixed point $\bar f_\sharp$ appearing in the formulas above
    can be described as follows:
    
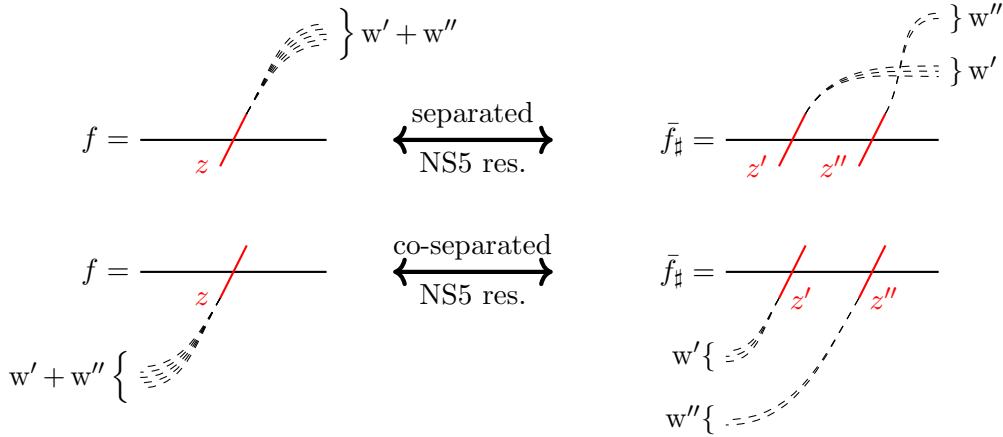
\[
\begin{tikzpicture}[scale=.35]
\draw [thick] (0,1) node[left]{$f=$} --(7,1);  
\draw [thick, red] (3,0) node [left]{$z$} -- (4,2);
\draw [dashed](4,2) to [out=60,in=180] (7,4.6) ;
\draw [dashed](4,2) to [out=60,in=180] (7,4.8) ;
\draw [dashed](4,2) to [out=60,in=180] (7,5) node [right]{$ \Big\rbrace \w'+\w'' $};
\draw [dashed](4,2) to [out=60,in=180] (7,5.2);
\draw [dashed](4,2) to [out=60,in=180] (7,5.4);
\draw[ultra thick, <->] (9.5,1)--(12.5,1) node[above]{separated} node[below]{NS5 res.} -- (15.5,1);
\begin{scope}[xshift=22cm]
\draw [thick] (0,1) node[left]{$\bar f_\sharp=$} --(8,1);  
\draw [thick, red] (2,0) node [left]{$z'$} -- (3,2);
\draw[thick, red] (5,0) node [left]{$z''$} --(6,2);
\draw [dashed](3,2) to [out=60,in=180] (8,3.4) ;
\draw [dashed](3,2) to [out=60,in=180] (8,3.6) node [right] {$\rbrace \w'$} ;
\draw [dashed](3,2) to [out=60,in=180] (8,3.8) ;
\draw [dashed](6,2) to [out=60,in=180] (8,5.6) node [right]{$\rbrace \w''$};
\draw [dashed](6,2) to [out=60,in=180] (8,5.8);
\end{scope}

\draw [thick] (0,-4) node[left]{$f=$} --(7,-4);
\draw [thick, red] (3,-5) node [left]{$z$} -- (4,-3);
\draw [dashed](3,-5) to [out=-120,in=0] (0,-7.6) ;
\draw [dashed](3,-5) to [out=-120,in=0] (0,-7.8) ;
\draw [dashed](3,-5) to [out=-120,in=0] (0,-8)  node [left]{$\w'+\w'' \Big\lbrace$};;
\draw [dashed](3,-5) to [out=-120,in=0] (0,-8.2);
\draw [dashed](3,-5) to [out=-120,in=0] (0,-8.4);
\draw[ultra thick, <->] (9.5,-4)--(12.5,-4) node[above]{co-separated} node[below]{NS5 res.} -- (15.5,-4);

\begin{scope}[xshift=22cm]
\draw [thick] (0,-4) node[left]{$\bar f_\sharp=$} --(8,-4);  
\draw [thick, red] (2,-5) node [right]{$z'$} -- (3,-3);
\draw[thick, red] (5,-5) node [right]{$z''$} --(6,-3);
\draw [dashed](2,-5) to [out=-120,in=0] (0,-7) ;
\draw [dashed](2,-5) to [out=-120,in=0] (0,-7.2) node [left] {$\w' \lbrace$} ;
\draw [dashed](2,-5) to [out=-120,in=0] (0,-7.4) ;
\draw [dashed](5,-5) to [out=-120,in=0] (0,-9.6) node [left]{$ \w'' \lbrace$};
\draw [dashed](5,-5) to [out=-120,in=0] (0,-9.8);
\end{scope}
\end{tikzpicture}
\]
Namely, if $X$ is separated (resp. co-separated), then the tie diagram of $\tilde f_\sharp$ is obtained from the one of $f$ by making all the ties connected to the resolving branes $\Zb'$ and $\Zb''$ cross (resp. by avoiding all the crosses).

Comparing these pictures with \eqref{diagrams distinguished resolving fixed points D5 case} and recalling that mirror symmetry swaps separated and co-separated bow varieties, it becomes clear that if $\mathfrak{C}$ is the standard chamber, then we have $(\bar f_\sharp)^!=\tilde f_{\sharp}$.
\end{remark}


\section{3d Mirror Symmetry of Stable envelopes}

\subsection{Mirror Symmetry statement}
\label{subsection mirror symmetry statement}
In this section, we consider the stable envelopes of an arbitrary bow variety $X$ with $m$ NS5 branes and $n$ D5 branes and of its dual variety $X^!$. Unless specified differently, the choice of the standard chamber (cf. Section \ref{subsection: hamber structure}) is always understood, and hence the symbol $\mathfrak{C}$ is dropped from the notation. For instance, we will write $\StabX{X}{}{f}$ as a shortcut for the stable envelope $\StabX{X}{C}{f}$ of $X$ with standard chamber $\mathfrak{C}=\lbrace a_1<\dots<a_n\rbrace $.

For a fixed point $f\in X^{\At}$, let the sign $\varepsilon_X(f)$ be defined as $(-1)$ to the power of
\[
\sum_{i=1}^n\sum_{j=1}^m\sum_{\substack{k>i\\ l>j}} b_{ij}b_{kl},
\]
where $b_{ij}$ are the entries of the BCT table of $f$ introduced in Section \ref{sec:HW}. Notice that since the BCT tables only depend on $X$ and $f$ up to Hanany-Witten isomorphism, the same is true for $\varepsilon_X(f)$. It is instructive to check that if $X$ is separated, then the sign $\varepsilon(f)$ is $(-1)$ to the power of the number of tie crossings in the tie diagram of $f$. 

We can finally state the main theorem of this paper.

\begin{theorem}
\label{theorem mirror symmetry stabs}
    Mirror symmetry for stable envelopes of bow varieties holds, i.e.
    \begin{equation}
    \label{3d mirror symmetry formula}
    \frac{\StabX{X}{}{f}|_g}{\StabX{X}{}{g}|_g}(a,z,\h)=\varepsilon_X(f)\varepsilon_X(g)\frac{\StabX{X^!}{}{g^!}|_{f^!}}{\StabX{X^!}{}{f^!}|_{f^!}}(z,a,\h^{-1})
    \end{equation}
    for any bow variety $X$ and fixed points $f,g\in X^{\At}$.
\end{theorem}

Since all the elements of \eqref{3d mirror symmetry formula} are invariant under Hanany-Witten isomorphism, our strategy for the proof is to first reduce to the case of separated $X$ and then to argue by induction utilizing the following two lemmas.

\begin{lemma}
\label{lemma induction mirror symmetry NS5}
    Let $X$ be separated and $\Zb$ be an NS5 brane of $X$ with local charge $\w=\w(\Zb)>1$. Set $\w''=1$ and let $\overline X$ be the NS5 resolution of $\Zb$ associated with the splitting $\w=\w'+\w''$. If Theorem \ref{theorem mirror symmetry stabs} holds for $\overline X$, then it also holds for $X$.
\end{lemma}

\begin{lemma}
\label{lemma induction mirror symmetry D5}
    Let $X$ be separated and let $\Ab$ be a D5 brane of $X$ with local charge  $\w=\w(\Ab)>1$. Set $\w'=1$ and let $\widetilde X$ be the D5 resolution of $\Ab$ associated with the splitting $\w=\w'+\w''$. If Theorem \ref{theorem mirror symmetry stabs} holds for $\widetilde X$, then it also holds for $X$.
\end{lemma}

As it will become clear from their proofs, it is reasonable to expect both lemmas above to hold without the assumptions $\w''=1$ or $\w'=1$. However, these assumptions greatly simplify certain computations, so we restrict ourselves to these special cases. We will prove these lemmas in the next section. Before that, let us explain why they imply Theorem \ref{theorem mirror symmetry stabs}.

\begin{proof}[Proof of Theorem \ref{theorem mirror symmetry stabs}] Since both the stable envelopes and the sign $\varepsilon_X(f)$ are invariant under Hanany-Witten isomorphism, we can assume that $X$ is separated. Moreover, we claim that it suffices to assume that all the local charges (both D5 and NS5) of $X$ are strictly positive. Indeed, as shown in Section \ref{section: Getting rid of the ``trivial'' branes}, the ratios of stable envelopes are unaffected by the insertions of charge zero branes. Likewise, inserting a D5 (resp. NS5) brane in the brane diagram of $X$ corresponds to adding a column (resp. a row) of zeros in the BCT tables of its fixed points. Consequently, the signs $\varepsilon_X(f)$ and $\varepsilon_X(g)$ are also unaffected by charge zero 5-branes. This proves the claim.

Overall, we can assume that $X$ is separated and all its local charges are strictly positive. We argue by induction on the number $N(X)$ of 5-branes $\mathcal{B}$ with local charge $\w(\mathcal{B})>1$. The case $N(X)=0$ is the base of our induction. In this case, we have $X\cong T^*\Fl_d\cong X^!$, where $d$ is the dimension of the unique D3 brane of $X$ in between a D5 and an NS5 brane, and Theorem \ref{theorem mirror symmetry stabs} is already proved in \cite{Rimanyi_2019full}. 

Assume now that the theorem holds for all $X$ such that $N(X)\leq M$ for some non-negative $M$ and fix any $X$ such that $N(X) =M+1$. This means that there exists at least one 5-brane $\mathcal{B}$ in $X$ with local charge $\w=\w(\mathcal{B})>1$. If $\mathcal{B}=\Zb$ (resp. $\mathcal{B}=\Ab$), then applying Lemma \ref{lemma induction mirror symmetry NS5} (resp. Lemma \ref{lemma induction mirror symmetry D5}) $\w-1$ times, it follows that the theorem holds for $X$ if it holds for the variety $X^{(\w-1)}$ obtained from $X$ by replacing $\mathcal{B}$ with $\w$ consecutive 5-branes of the same type with local charges all equal to one. But $N(X^{(\w-1)})=N(X)-1=M$, so by the inductive hypothesis the theorem holds for $X^{(\w-1)}$ and hence also for $X$.

\end{proof}

\subsection{Proofs of Lemma \ref{lemma induction mirror symmetry NS5} and Lemma \ref{lemma induction mirror symmetry D5}}

\begin{proof}[Proof of Lemma \ref{lemma induction mirror symmetry NS5}]  We first describe the strategy of the proof.  Let $m$ and $n$ be, respectively, the number of NS5 and D5 branes of the separated bow variety $X$. The strategy for the proof is the following. Consider the NS5 resolution $\overline X$ of the brane $\Zb$ in $X$ and the D5 resolution $\widetilde{(X^!)}$ of the brane $\Ab=\Zb^!$ in $X^!$. Clearly, we have $(\overline X)^!=\widetilde{(X^!)}$. By Proposition \ref{proposition separated NS5 fusion ratios stabs}, we express the left-hand side of \eqref{3d mirror symmetry formula} as a linear combination of stable envelopes of  $\overline X$. Similarly, by Proposition \ref{multiple terms D5 resolution stabs co-separated case}, we express the right-hand side as a linear combination of the stable envelopes of $\widetilde{(X^!)}=(\overline X)^!$. Since we assume that mirror symmetry holds for $\overline X$, to conclude that it also holds for $X$ it suffices to show that the coefficients in these two linear combinations match after the changes of variables prescribed by mirror symmetry. This comparison of the coefficients is done by means of an explicit computation.
	
After this overview of our argument's structure, we now spell it out in detail. For the sake of notation, say that $\Zb$ is the $k$-th NS5 brane in $X$. Then $X$ and $\overline X$ are
\begin{equation*}
	\begin{tikzpicture}[baseline=0,scale=0.4]
		\draw [thick,red] (0.6,0) --(1.4,2); 
		\draw [thick,red](3.6,0) --(4.4,2);  
		\draw [thick,red](8.6,0) -- (9.4,2);  
		\draw [thick,red](13.6,0) -- (14.4,2); 
		\draw [thick,red] (16.6,0) -- (17.4,2);
		\draw [thick,blue](20.4,0) -- (19.6,2); 
		\draw [thick,blue](23.4,0) -- (22.6,2);
		\draw [thick,blue] (26.4,0) -- (25.6,2); 
		\draw [thick,blue](31.4,0) -- (30.6,2);  
		\draw [thick,blue](34.4,0) -- (33.6,2); 
		
		\draw [thick] (1,1) node[left]{$X=$}--(4,1);
		\draw [thick, dashed] (4,1)--(14,1);
		\draw [thick] (14,1)--(26,1);
		\draw [thick, dashed] (26,1)--(31,1);
		\draw [thick] (31,1)--(34,1);

		\node at (1,-.8) {$z_{1}$};
		\node at (4,-.8) {$z_{2}$};
		\node at (9,-.8) {$z_{k}$};
		\node at (14,-.8) {$z_{m-1}$};
		\node at (17,-.8) {$z_{m}$};
		
		\node at (20.7,-.8) {$a_{1}$};
		\node at (23.7,-.8) {$a_{2}$};
		\node at (26.7,-.8) {$a_{3}$};
		\node at (31.7,-.8) {$a_{n-1}$};
		\node at (34.7,-.8) {$a_{n}$};
		
		\node at (9.5,2.8) {$\w$};
		
	\end{tikzpicture}
\end{equation*}
\begin{equation*}
	\begin{tikzpicture}[baseline=0,scale=.4]
		\draw [thick,red] (0.6,0) --(1.4,2); 
		\draw [thick,red](3.6,0) --(4.4,2);  
		\draw [thick,red](8.6,0) -- (9.4,2);  
		\draw [thick,red](11.6,0) -- (12.4,2); 
		\draw [thick,red] (16.6,0) -- (17.4,2);
		\draw [thick,red](19.6,0) -- (20.4,2); 
		\draw [thick,blue](23.4,0) -- (22.6,2);
		\draw [thick,blue] (26.4,0) -- (25.6,2); 
		\draw [thick,blue] (29.4,0) -- (28.6,2); 
		\draw [thick,blue](34.4,0) -- (33.6,2);  
		\draw [thick,blue](37.4,0) -- (36.6,2); 
		
		\draw [thick] (1,1) node[left]{$\overline X=$}--(4,1);
		\draw [thick, dashed] (4,1)--(9,1);
		\draw [thick] (9,1)--(12,1);
		\draw [thick, dashed] (12,1)--(17,1);
		\draw [thick] (17,1)--(29,1);
		\draw [thick, dashed] (29,1)--(34,1);
		\draw [thick] (34,1)--(37,1);

		\node at (1,-.8) {$z_{1}$};
		\node at (4,-.8) {$z_{2}$};
		\node at (9,-.8) {$z'_{k}$};
		\node at (12.3,-.8) {$z''_{k}$};
		\node at (17,-.8) {$z_{m-1}$};
		\node at (20,-.8) {$z_{m}$};
		
		\node at (23.7,-.8) {$a_{1}$};
		\node at (26.7,-.8) {$a_{2}$};
		\node at (29.7,-.8) {$a_{3}$};
		\node at (34.7,-.8) {$a_{n-1}$};
		\node at (37.7,-.8) {$a_{n}$};
		
		\node at (9.5,2.8) {$\w'$};
		\node at (12.5,2.8) {$\w''$};
		
	\end{tikzpicture}.
\end{equation*}
Fix $f\in X^{\At}$ and recall that by assumption $\w''=1$ and hence $\w'=\w-1$. As stated in Lemma \ref{lemma points in lagrangian variety Z resolution} (2), associated with $f$ are $\binom{\w'+\w''}{\w'}=\w$ resolutions $\bar f$. The latter can be both identified with the fixed points in $Y=\ttt{{\fs}$\w-1${\fs}$\w${\bs}\dots{\bs}3{\bs}2{\bs}1{\bs}}$ or with certain fixed points in 
$\overline X$ by means of the embedding $Y^{\hbar }\hookrightarrow \overline X$. For any $i=1,\dots, \w$, let us denote by $f_i$ the resolution of $f$ that, under the identification $Y\cong T^*\Gr{\w-1}{\w}$, corresponds to the hyperplane dual to the $i$-th coordinate plane in $\mathbb{C}^{\w}$\footnote{Writing $\bar f_i$ instead of $f_i$ would be more consistent with our previous notation, but would also make the present proof notationally too heavy.}. In the language of tie diagrams, $f_i$ is the unique tie diagram of  $\ttt{{\fs}$\w-1${\fs}$\w${\bs}\dots{\bs}3{\bs}2{\bs}1{\bs}}$ connecting the rightmost NS5 brane to the $i$-th D5 brane (counting from left to right). Notice that this ordering is consistent with our choice of chamber because we have 
\begin{equation}
	\label{fixed point order proof 3d mirror symmetry}
f_1<f_2<\dots <f_{\w-1}<f_{\w},
\end{equation}
both in $Y$ and in $\overline X$. In particular, the minimal resolution $\bar{f}_\sharp$ introduced in Theorem \ref{main theorem proof NS5 resolution for stabs} coincides with $f_{\w}$ (see also Remark \ref{remark distinguished resolutions fixed points NS5 case}). By Proposition \ref{proposition separated NS5 fusion ratios stabs}, the stable envelopes of $X$ and $\overline X$ are related as follows 
\begin{multline}
	\label{mirror symmetry proof equation 1}
	\frac{\StabX{X}{}{f}\Big|_g}{\StabX{X}{}{g}\Big|_g}(a,z,\h)\\
    =
	\sum_{ i=1}^{\w}  d_{i}(a\h^{-\gamma(g)},\h) 
	\frac{\StabX{\overline{X}}{}{f_{\w}}\Big|_{ g_i}}{\StabX{\overline{X}}{}{g_i}\Big|_{ g_i}}
	(a, z_1,\dots,z'_{i}=z_k, z''_{k}=z_k\h^{1-\w}, \dots, z_m, \h),
\end{multline}
where 
\begin{equation*}
	d_{i}(a,\h)=\frac{\vartheta(N_{ g_i/Y}^-)}{\vartheta(N_{ g_i/Y^{\h}})}(a,\h)
\end{equation*}
is an elliptic class of $Y\cong T^*\Gr{\w-1}{\w}$.

Let us now look at the dual side. As noticed above, we have $(\overline X)^!\cong \widetilde{X^!}$, where $\widetilde{X^!}$ is the D5 resolution of the $i$-th D5 brane in $X^!$: 
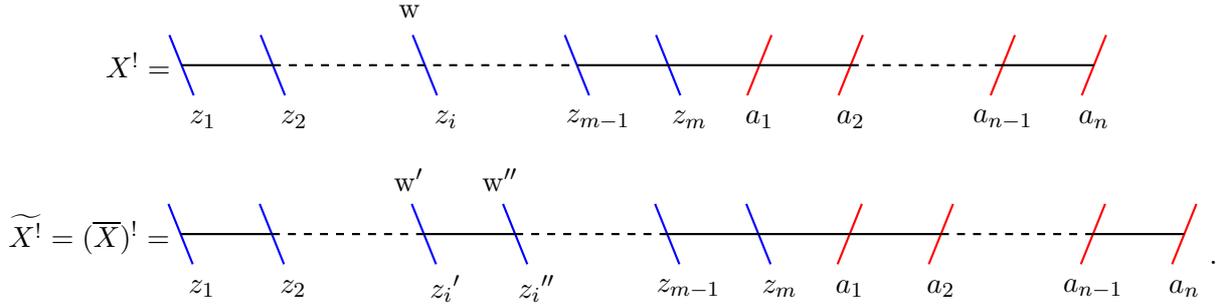
\begin{equation*}
	\begin{tikzpicture}[baseline=0,scale=.4]
		\draw [thick,blue] (1.4,0) --(0.6,2); 
		\draw [thick,blue](4.4,0) --(3.6,2);  
		\draw [thick,blue](9.4,0) -- (8.6,2);  
		\draw [thick,blue](14.4,0) -- (13.6,2); 
		\draw [thick,blue] (17.4,0) -- (16.6,2);
		\draw [thick,red](19.6,0) -- (20.4,2); 
		\draw [thick,red](22.6,0) -- (23.4,2);
		\draw [thick,red](27.6,0) -- (28.4,2);  
		\draw [thick,red](30.6,0) -- (31.4,2); 
		
		\draw [thick] (1,1) node[left]{$X^!=$}--(4,1);
		\draw [thick, dashed] (4,1)--(14,1);
		\draw [thick] (14,1)--(23,1);
		\draw [thick, dashed] (23,1)--(28,1);
		\draw [thick] (28,1)--(31,1);

		\node at (1.7,-.8) {$z_{1}$};
		\node at (4.7,-.8) {$z_{2}$};
		\node at (9.7,-.8) {$z_{i}$};
		\node at (14.7,-.8) {$z_{m-1}$};
		\node at (17.7,-.8) {$z_{m}$};
		
		\node at (20,-.8) {$a_{1}$};
		\node at (23,-.8) {$a_{2}$};
		\node at (28,-.8) {$a_{n-1}$};
		\node at (31,-.8) {$a_{n}$};
		
		\node at (8.5,2.8) {$\w$};
		
	\end{tikzpicture}
\end{equation*}
\begin{equation*}
	\begin{tikzpicture}[baseline=0,scale=.4]
		\draw [thick,blue] (1.4,0) --(0.6,2); 
		\draw [thick,blue](4.4,0) --(3.6,2);  
		\draw [thick,blue](9.4,0) -- (8.6,2);  
		\draw [thick,blue](12.4,0) -- (11.6,2); 
		\draw [thick,blue] (17.4,0) -- (16.6,2);
		\draw [thick,blue](20.4,0) -- (19.6,2); 
		\draw [thick,red](22.6,0) -- (23.4,2);
		\draw [thick,red] (25.6,0) -- (26.4,2); 
		\draw [thick,red](30.6,0) -- (31.4,2);  
		\draw [thick,red](33.6,0) -- (34.4,2); 
		
		\draw [thick] (1,1) node[left]{$\widetilde{X^!}=(\overline X)^!=$}--(4,1);
		\draw [thick, dashed] (4,1)--(9,1);
		\draw [thick] (9,1)--(12,1);
		\draw [thick, dashed] (12,1)--(17,1);
		\draw [thick] (17,1)--(26,1);
		\draw [thick, dashed] (26,1)--(31,1);
		\draw [thick] (31,1)--(34,1);

		\node at (1.7,-.8) {$z_{1}$};
		\node at (4.7,-.8) {$z_{2}$};
		\node at (9.7,-.8) {${z_{i}}'$};
		\node at (12.7,-.8) {${z_{i}}''$};
		\node at (17.7,-.8) {$z_{m-1}$};
		\node at (20.7,-.8) {$z_{m}$};
		
		\node at (23,-.8) {$a_{1}$};
		\node at (26,-.8) {$a_{2}$};
		\node at (31,-.8) {$a_{n-1}$};
		\node at (34,-.8) {$a_{n}$};
		
		\node at (8.5,2.8) {$\w'$};
		\node at (11.5,2.8) {$\w''$};
		
	\end{tikzpicture}.
\end{equation*}
Notice that here the variables $z$ play the role of the equivariant parameters, while $a$ the one of the K\"ahler parameters. Recall that the torus $\At^!$ acting on $X^!$ (the space of the variables $z$) is a proper subtorus of the torus $\widetilde{\At^!}$ acting on $\widetilde{X^!}$. Fix some $f^!\in (X^!)^{\At^!}$. Like in the NS5 case discussed above, by Lemma \ref{lemma Fixed components D5 res} we associate with $f^!$ its resolutions $\widetilde{(f^!)}$. The latter can be identified both with the $\widetilde{\At^!}$-fixed points in $F=\ttt{{\bs}$\w-1${\bs}$\w${\fs}\dots{\fs}3{\fs}2{\fs}1{\fs}}$ or with certain $\widetilde{\At^!}$-fixed points in $\widetilde{X^!}$ by means of the embedding $F\hookrightarrow  \widetilde{X^!}$.

 The first key observation for the proof is that $F=Y^!$, where $Y$ is the bow variety introduced above.  Mirror symmetry of fixed points for the pair $(Y, Y^!=F)$ implies that we have a one-to-one correspondence between NS5 resolutions of $f\in X^{\At}$ and D5 resolutions of $f^!\in (X^!)^{\At^!}$. We label the resolutions of $f$ according to mirror symmetry, i.e. we set $f^!_i:=(f_i)^!$ for $i=1\dots, \w$. The distinguished resolution $\widetilde{(f^!)}_{\sharp}$ from Lemma \ref{restrction Chern roots in A resolution} is identified with $f^!_{\w}$. With this notation, Proposition \ref{multiple terms D5 resolution stabs co-separated case} states that
\begin{multline}
	\label{mirror symmetry proof equation 2}
	\frac{\StabX{X^!}{}{g^!}\Big|_{f^!}}{\StabX{X^!}{}{f^!}\Big|_{f^!}}(z, a,\h^{-1}) \\
	=
	\sum_{i=1}^{\w} c_i(  a\h^{-\gamma(g)}, \h^{-1})\frac{\StabX{\widetilde{X^!}}{}{g^!_i}\Big|_{f^!_{\w}}}{\StabX{\widetilde{X^!}}{}{f^!_{\w}}\Big|_{f^!_{\w}}}(z_1,\dots,z_k'=z_k, z_k''=z_k\h^{1-\w},\dots, z_m, a,\h^{-1}),
\end{multline}
where 
\begin{equation}
    \label{definition c_i coeff in proof mirror symmetry}
    c_i(a,\h^{-1})= \frac{ \left( R_{\widetilde{\mathfrak{C}}/\mathfrak{C},\widetilde{\mathfrak{C}}^\vee / \mathfrak{C}}\right)_{g^!_i, g^!_{\w}}}{\left( R_{\widetilde{\mathfrak{C}}/\mathfrak{C},\widetilde{\mathfrak{C}}^\vee / \mathfrak{C}}\right)_{g^!_{\w},  g^!_{\w}}}(z_k,z_k\h^{1-\w}, a, \h^{-1})
\end{equation}
is a ratio of R-matrix coefficients of the bow variety $F=\ttt{{\bs}$\w-1${\bs}$\w${\fs}\dots{\fs}3{\fs}2{\fs}1{\fs}}$. Notice that since the R-matrix only depends on the ratio $z_k\h^{1-\w}/z_k=\h^{1-\w}$ of the equivariant parameters, the dependence on $z_k$ is only apparent. 

%
%

All the ingredients have finally been set. Applying the assumption of the lemma, i.e. mirror symmetry for stable envelopes of the pair $(\overline X, (\overline X)^!=\widetilde{X^!})$, to each of the summands on the right-hand side of \eqref{mirror symmetry proof equation 2}, we deduce that the right hand side of \eqref{3d mirror symmetry formula}, i.e. the class
\[
\varepsilon_X(f)\varepsilon_X(g)\frac{\StabX{X^!}{}{g^!}\Big|_{f^!}}{\StabX{X^!}{}{f^!}\Big|_{f^!}}(z, a,\h^{-1}),
\]
is equal to 
\begin{equation*}
	\sum_{ i=1}^{\w}  \frac{\varepsilon_X(f)\varepsilon_X(g)}{\varepsilon_{\overline X}(f_{\w})\varepsilon_{\overline X}(g_i)} c_i(a\h^{-\gamma(g)},\h^{-1})
	\frac{\StabX{\overline{X}}{}{f_{\w}}\Big|_{ g_i}}{\StabX{\overline{X}}{}{g_i}\Big|_{ g_i}}
	(a, z_1,\dots,z'_{i}=z_k, z''_{k}=z_k\h^{1-\w}, \dots, z_m, \h).
\end{equation*}
Comparing this last formula with \eqref{mirror symmetry proof equation 1}, we deduce that to complete the proof of the lemma it suffices to show that the coefficients of these two resolutions are mirror dual, i.e. that
\begin{equation}
	\label{equation equality coefficients resolutions}
	\frac{\varepsilon_X(f)\varepsilon_X(g)}{\varepsilon_{\overline X}(f_{\w})\varepsilon_{\overline X}(g_i)} c_i(a,\h^{-1})= d_{i}(a,\h) 
\end{equation}
for all $i=1,\dots,\w$. We check equation \eqref{equation equality coefficients resolutions} via an explicit computation of both sides, performed in the next three lemmas. Namely, in Lemma \ref{lemma 1 proof mirror symmetry stab} we compute $d_{i}(a,\h)$, in Lemma \ref{lemma 2 proof mirror symmetry stab} we compute $c_i(a,\h^{-1})$, and in the last lemma we show that these two classes exactly differ by the sign $\frac{\varepsilon_X(f)\varepsilon_X(g)}{\varepsilon_{\overline X}(f_{\w})\varepsilon_{\overline X}(g_i)}$.
\end{proof}

\begin{lemma}
\label{lemma 1 proof mirror symmetry stab}
	Consider the bow variety $Y=\ttt{{\fs}$\w-1${\fs}$\w${\bs}\dots{\bs}3{\bs}2{\bs}1{\bs}}$. We have
	\[
	d_{i}(a,\h)=\prod_{j=i+1}^{\w}\frac{\vartheta\left(\frac{a_i\h}{a_j}\right) }{\vartheta\left(\frac{a_j}{a_i}\right)}.
	\]
\end{lemma}
\begin{proof}
	By definition 
	\[
	d_{i}(a,\h)=\frac{\vartheta(N_{ g_i/Y}^-)}{\vartheta(N_{ g_i/Y^{\h}})}(a,\h),
	\]
	hence the proof directly follows from the tautological description of the tangent space $TY$ of the bow variety $Y$ \cite[Section 3.2]{rimanyi2020bow}. Alternatively, one can use one of the standard descriptions of the tangent space of $T^*\Gr{\w-1}{\w}\cong Y$.
\end{proof}

\begin{lemma}
\label{lemma 2 proof mirror symmetry stab}
	Consider the bow variety $Y^!=F=\ttt{{\bs}$\w-1${\bs}$\w${\fs}\dots{\fs}3{\fs}2{\fs}1{\fs}}$. As before, denote by $z'$ and $z''$ its equivariant parameters and by $a$ its K\"ahler parameters. Let $R:= R_{\lbrace z'<z''\rbrace ,\lbrace z'>z''\rbrace}$ be its R-matrix. We have
	\[
	\frac{ R_{g^!_i, g^!_{\w}}}{R_{g^!_{\w}, g^!_{\w}}}(z',z'', a,\hbar)=\frac{\text{Stab}^F_{\lbrace z'>z''\rbrace }(g^!_i)\Big|_{g^!_{\w}}}{\text{Stab}^F_{\lbrace z'>z''\rbrace }(g^!_{\w})\Big|_{g^!_{\w}}}(z',z'',a^{-1},\h).
	\]
	Computing the right-hand side, we deduce that 
	\[
	c_{i}(a,\h^{-1})=(-1)^{\w-i} \prod_{j=i+1}^{\w}\frac{\vartheta\left(\frac{a_i\h}{a_j}\right) }{\vartheta\left(\frac{a_j}{a_i}\right)}.
	\]
\end{lemma}
\begin{proof}
Firstly, notice that all the NS5 charges $r_i$ of $F$ are equal to one. Thus, applying Corollary \ref{symmetry R matrix} and keeping in mind that now $a$ is the K\"ahler parameter, we get 
\[
\frac{ R_{g^!_i, g^!_{\w}}}{R_{g^!_{\w}, g^!_{\w}}}(z',z'', a,\hbar)=\frac{ R_{g^!_{\w}, g^!_i}}{R_{g^!_{\w}, g^!_{\w}}}(z',z'', a^{-1},\hbar).
\]
By definition of R-matrix, we have 
\[
\text{Stab}^F_{\lbrace z'>z''\rbrace }(g^!_i)\Big|_{g^!_{\w}}=\sum_{j=1}^{\w} \text{Stab}^F_{\lbrace z'<z''\rbrace }(g^!_j)\Big|_{g^!_{\w}} R_{g^!_j, g^!_i}
\]
for all $ i,j=1,\dots, \w$. Since $g^!_{\w}$ is maximal\footnote{Notice that the mirror dual of  $g^!_{\w}$, namely the fixed point  $g_{\w}\in Y^{\At}$, is instead minimal in the order introduced in \eqref{fixed point order proof 3d mirror symmetry}. This is in agreement with the fact that mirror symmetry inverts the poset structure of the dual fixed loci.} with respect to the partial order determined by $\lbrace z'<z'' \rbrace$, only the diagonal term on the right-hand side survives; hence, we get 
\[
\text{Stab}^F_{\lbrace z'>z''\rbrace }(g^!_i)\Big|_{g^!_{\w}}= \text{Stab}^F_{\lbrace z'<z''\rbrace }(g^!_{\w})\Big|_{g^!_{\w}} R_{g^!_{\w}, g^!_i},
\]
from which the first statement follows.

In order to deduce the second statement from the first one, we first apply Proposition \ref{innocent identity} to get
\[
\frac{ R_{g^!_i, g^!_{\w}}}{R_{g^!_1, g^!_{\w}}}(z',z'', a,\hbar)
=
\frac{\text{Stab}^{F^\vee}_{\lbrace z'<z''\rbrace }(h_i)\Big|_{h_{\w}}}{\text{Stab}^{F^\vee}_{\lbrace z'<z''\rbrace }(h_{\w})\Big|_{h_{\w}}}(z'',z',a^{-1},\h).
\]
Here, $F^\vee$ is the bow variety $\ttt{{\bs}$1${\bs}$\w-1${\fs}\dots{\fs}3{\fs}2{\fs}1{\fs}}$ and $h_i:= (g^!_i)^\vee$.  In terms of tie diagrams, the point $h_i$ corresponds to the unique tie diagram of $\ttt{{\bs}$1${\bs}$\w-1${\fs}\dots{\fs}3{\fs}2{\fs}1{\fs}}$ connecting the leftmost D5 brane to the $i$-th NS5 brane on the right. 

At this point, one could directly compute the right-hand side using a tautological presentation of the stable envelopes \cite{RRTBshuffle}, but we choose a different strategy that simplifies the computation. Notice that $F^\vee$ is dual to the bow variety $(F^\vee)^!= \ttt{{\fs}$1${\fs}$\w${\bs}\dots{\bs}3{\bs}2{\bs}1{\bs}} \cong T^*\mathbb{P}^{\w-1}$ (which should not be confused with $Y\cong \ttt{{\fs}$\w-1${\fs}$\w${\bs}\dots{\bs}3{\bs}2{\bs}1{\bs}}\cong T^*\Gr{\w-1}{\w}$).
Since 3d mirror symmetry for stable envelopes of $T^*\mathbb{P}^{\w-1}$ is known\footnote{Mirror symmetry of $X=T^*\Gr{k}{n}$ and its dual $X^!$ was, before this article, the only known case beyond the self duality of the cotangent bundle of the full flag variety.} \cite{rimanyi20193d}, we can invoke it to deduce that
\[
\frac{ R_{g^!_i, g^!_{\w}}}{R_{g^!_{\w}, g^!_{\w}}}(z',z'', a,\h)
=
\varepsilon_{(F^\vee)^!}(h^!_i)\varepsilon_{(F^\vee)^!}(h^!_{\w}) \frac{\StabX{(F^\vee)^!}{}{h^!_{\w}}\Big|_{h^!_i}}{\StabX{(F^\vee)}{ }{h^!_i}\Big|_{h^!_i}}(a^{-1}, z'',z', \h^{-1}).
\]
As usual, the standard chamber is understood, and $h^!_i:=(h_i)^!$. In particular, the fixed point $h^!_i$  is the $i$-th coordinate plane in $\mathbb{P}^{\w-1}\subset T^*\mathbb{P}^{\w-1}$. Recalling the definition of $c_i(a,\hbar^{-1})$ from equation \eqref{definition c_i coeff in proof mirror symmetry}, we finally have
\begin{equation}
	\label{key formula for lemma comparison coefficients}
c_i(a,\h^{-1})=\frac{ R_{g^!_i, g^!_{\w}}}{R_{g^!_{\w}, g^!_{\w}}}(z,z\h^{1-\w}, a,\h^{-1})=\varepsilon_{(F^\vee)^!}(h^!_i)\varepsilon_{(F^\vee)^!}(h^!_{\w}) \frac{\StabX{(F^\vee)^!}{}{h^!_{\w}}\Big|_{h^!_i}}{\StabX{(F^\vee)}{ }{h^!_i}\Big|_{h^!_i}}(a^{-1}, z\h^{1-\w},z, \h).
\end{equation}
This formula is particularly advantageous because it expresses the left-hand side in terms of the stable envelopes of $T^*\mathbb{P}^{\w-1}$, the easiest stable envelopes to compute. They are explicitly described in Section \ref{sec:StabForPn}. Let us first compute the right-hand side without the change of variables. For this, notice that the tie diagram of $h^!_{i}$ connects the leftmost NS5 brane in $\ttt{{\fs}$1${\fs}$\w${\bs}\dots{\bs}3{\bs}2{\bs}1{\bs}} \cong T^*\mathbb{P}^{\w-1}$ with the $i$-th D5 brane from the left. By Proposition \ref{prop:Pn} and the combinatorics of the fixed point restrictions, also described in Section \ref{sec:StabForPn}, we get
\begingroup
\allowdisplaybreaks
\begin{align*}
	\frac{\StabX{(F^\vee)^!}{}{h^!_{\w}}\Big|_{h^!_i}}{\StabX{(F^\vee)}{ }{h^!_i}\Big|_{h^!_i}}(a_1,\dots, a_{\w}, z_1,z_2, \h)
	&=
	\frac{\prod_{i=1}^{\w-1}\vartheta\left(\frac{a_i\h}{a_k}\right)}{\prod_{i=1}^{k-1} \vartheta\left(\frac{a_i\h}{a_k}\right) \prod_{i=k+1}^{\w} \vartheta\left( \frac{a_k}{a_i} \right) } \frac{\vartheta\left(\frac{a_k}{a_{\w}}\frac{z_1}{z_2}\h^{\w-2}\right)}{\vartheta\left(\frac{z_2}{z_1}\h^{\w-2}\right)}
	\\
	&=
	\frac{
		\prod_{i=1}^{k-1}\vartheta\left(\frac{a_i\h}{a_k}\right)
		\vartheta(h)
		\prod_{i=k+1}^{\w-1}\vartheta\left(\frac{a_i\h}{a_k}\right)
	}
	{\prod_{i=1}^{k-1} \vartheta\left(\frac{a_i\h}{a_k}\right) \prod_{i=k+1}^{\w} \vartheta\left( \frac{a_k}{a_i} \right)
	}
	\frac{\vartheta\left(\frac{a_k}{a_{\w}}\frac{z_2}{z_1}\h^{\w-2}\right)}{\vartheta\left(\frac{z_1}{z_2}\h^{\w-2}\right)}
	\\
	&=\prod_{i=k+1}^{\w}\left(\frac{\vartheta\left(\frac{a_i\h}{a_k}\right)}{\vartheta\left( \frac{a_k}{a_i} \right)}\right)
	\frac{\vartheta(h)}{\vartheta\left(\frac{a_{\w}\h}{a_k}\right)}
	\frac{\vartheta\left(\frac{a_k}{a_{\w}}\frac{z_1}{z_2}\h^{\w-2}\right)}{\vartheta\left(\frac{z_2}{z_1}\h^{\w-2}\right)}.
\end{align*}
Forcing now the substitutions above and recalling that $\vartheta(x^{-1})=-\vartheta(x)$, we get
\begin{align*}
	 \frac{\StabX{(F^\vee)^!}{}{h^!_{\w}}\Big|_{h^!_i}}{\StabX{(F^\vee)}{ }{h^!_i}\Big|_{h^!_i}}(a^{-1}, z\h^{1-\w},z, \h)
	&=
	\prod_{i=k+1}^{\w}\left(\frac{\vartheta\left(\frac{a_k\h}{a_i}\right)}{\vartheta\left( \frac{a_i}{a_k} \right)}\right)
	\frac{\vartheta(h)}{\vartheta\left(\frac{a_k\h}{a_{\w}}\right)}
	\frac{\vartheta\left(\frac{a_{\w}}{a_k\h}\right)}{\vartheta\left(\h^{-1}\right)}
	\\
	&=
	\prod_{i=k+1}^{\w}\frac{\vartheta\left(\frac{a_k\h}{a_i}\right)}{\vartheta\left( \frac{a_i}{a_k} \right)}.
\end{align*}
\endgroup
By this last formula and equation \eqref{key formula for lemma comparison coefficients}, to conclude the proof it suffices to show that 
\[
\varepsilon_{(F^\vee)^!}(h^!_i)\varepsilon_{(F^\vee)^!}(h^!_{\w})=(-1)^{\w-i}.
\]
But this immediately follows from the characterization of $\varepsilon_X(-)$ for a separated brane diagram, given at the beginning of Section \ref{subsection mirror symmetry statement}, and the fact that the tie diagram of $h^!_i$ has exactly $\w-i$ tie crossings.
\end{proof}

\begin{lemma}
\label{lemma 3 proof mirror symmetry stab}
	We have 
	\[
	\frac{\varepsilon_X(f)\varepsilon_X(g)}{\varepsilon_{\overline X}(f_{\w})\varepsilon_{\overline X}(g_i)} = (-1)^{\w-i}.
	\]
    As a consequence, equation \eqref{equation equality coefficients resolutions} holds.
\end{lemma}
\begin{proof}
	Both $X$ and $\overline X$ are separated, so we can compute the left-hand side in terms of tie crossings. We claim that $\varepsilon_{\overline X}(g_i)/\varepsilon_{X}(g) =\varepsilon_{Y}(g_i)$, where $Y=\ttt{{\fs}$\w-1${\fs}$\w${\bs}\dots{\bs}3{\bs}2{\bs}1{\bs}}$. Indeed, the  tie crossings at the numerator differ from the ones at the denominator by those additional crossings that are generated by the resolution of the NS5 brane. But these are effectively the crossings of $g_i$ seen as a fixed point in $Y$.
	
	 Now recall that $g_i$ is the unique fixed point whose tie diagram in 
  \[
\ttt{{\fs}$\w-1${\fs}$\w${\bs}\dots{\bs}3{\bs}2{\bs}1{\bs}}\]
connects the rightmost NS5 brane with the $i$-th D5 brane. This implies that $\varepsilon_{Y}(g_i)=(-1)^{i-1}$. Hence, we get 
	\[
	\frac{\varepsilon_X(f)\varepsilon_X(g)}{\varepsilon_{\overline X}(f_{\w})\varepsilon_{\overline X}(g_i)} =(-1)^{1-i}(-1)^{1-\w}=(-1)^{2-\w-i}=(-1)^{\w-i}.
	\]
	The second claim simply follows by comparing the explicit formulas of Lemmas \ref{lemma 1 proof mirror symmetry stab} and \ref{lemma 2 proof mirror symmetry stab}.
\end{proof}

These computations conclude the proof of Lemma \ref{lemma induction mirror symmetry NS5}. The proof of its sibling statement, Lemma \ref{lemma induction mirror symmetry D5}, is completely analogous, so  we only sketch it.

\begin{proof}[Proof of Lemma \ref{lemma induction mirror symmetry D5}]
Lemma \ref{lemma induction mirror symmetry NS5} was about the comparison of the NS5 resolution for the stable envelopes of X with the D5 resolution for the stable envelopes of $X^!$. Here we do the opposite, i.e. we consider a D5 resolution of type $\w'+\w''=\w$ with  $\w'=1$ for some D5 brane $\Ab$ in X and the homologous NS5 resolution of the dual brane $\Zb=\Ab^!$ in $X^!$. Like before, we consider the associated resolutions of the stable envelopes on the two dual sides. The latter are given by Lemma \ref{multiple terms D5 resolution stabs separated case} on the D5 side and Lemma \ref{proposition co-separated NS5 fusion ratios stabs} on the NS5 side (i.e. the dual side). We then apply the assumption of the lemma, i.e. 3d mirror symmetry of stable envelopes for the pair $(\widetilde X, \overline{X^!}=(\widetilde X)^!)$ on one of the two resolutions (for instance on the D5 side, like before) and compare the resulting formulas. To show that these two formulas are the same, and hence deduce the statement of the lemma, it suffices to compare the coefficients. Like before, the latter can be computed explicitly. As for Lemma \ref{lemma 1 proof mirror symmetry stab}, the computation of the NS5 coefficients is straightforward. On the D5 side, one proceeds like in Lemma \ref{lemma 2 proof mirror symmetry stab} to reduce the coefficients to a ratio of stable envelopes of $T^*\mathbb{P}^{\w-1}$, now described as a co-separated bow variety, and uses the explicit formula in Proposition \ref{prop:Pn} to complete the computation.
\end{proof}



\appendix

\section{Equivariant localization for bow varieties}

\label{Appendix: Equivariant localization for bow varieties}

Torus equivariant localization is an extremely powerful tool for computations in cohomology. The most favorable situation to effectively apply localization techniques is when the natural pullback \footnote{In this article, we use the symbols $f^{\oast}$ and $f_{\oast}$ to denote pullback and pushforward in elliptic cohomology. However, since this appendix is about singular cohomology, we switch to the more standard notation $f^*$ and $f_*$.} $i^*$ associated with the inclusion $i:X^T\hookrightarrow X$ is an isomorphism, i.e. when the cohomology ring of $X$ is determined by the information at the fixed points. This property holds (after permitting denominators) for many important varieties in geometric representation theory, including partial flag varieties and Nakajima quiver varieties \cite{Nakajimaquiver, ginzburg2009lectures, maulik2012quantum}. However, $i^*$ is not injective in general for bow varieties. On some level, this lack of injectivity hints at the fact that the whole cohomology $H_T^*(X)$ ring (or its K-theory and elliptic analogs) of a bow variety is not the correct space to look at from a representation theoretic point of view. Instead, one should look at some appropriate subalgebra on which $i^*$ is an isomorphism.

In this appendix, we show that the subalgebra generated by the stable envelopes satisfies this property. For simplicity, we work in singular cohomology. However, 
all the statements and arguments adapt straightforwardly to K-theory and  elliptic cohomology.

Let $X$ be a $T$-space with finitely many $T$-fixed points. The cohomology $H_T^*(X^T)$ has a natural basis associated with the fixed points $f\in X^T$. With respect to this basis, the composition
\[
\begin{tikzcd}
    H^*_T(X^T)\arrow[r, "i_{*}"] &H^*_T(X)\arrow[r, "i^*"] & H^*_T(X^T)
\end{tikzcd}
\]
is equal to the diagonal matrix multiplying the basis elements by the Euler classes $e(T_f X)\in H_T^*(\pt)$. Since each tangent class $T_f X$ is simply a direct sum of nontrivial $T$-characters, the map $i^* i_{*}$ is injective and hence becomes an isomorphism after localization, i.e. after tensoring with $\text{Frac}(H_T^*(\pt))$. Notice in particular that although $i^*$ might not be injective, it is always injective on the image of $i_{*}$. From now on, we will always work in localized equivariant cohomology, i.e. with the functor 
\[
H_T^*(-)_{loc}:=H_T^*(-)\otimes_{H_T^*(\pt)}\text{Frac}(H_T^*(\pt)),
\]
and all the maps will be considered as morphisms of $\text{Frac}(H_T^*(\pt))$-modules. 

Let now $X$ be a bow variety with $m$ NS5 branes equipped with its standard action of the torus $\Tt=\At\times \Cs_{\hbar}$. The cohomological stable envelope is a morphism of $H_{\Tt}(\pt)$-modules 
\[
\text{Stab}_{\mC}: H_{\Tt}^*(X^{\Tt})=H_{\Tt}^*(X^{\At})\to H_{\Tt}^*(X),
\]
and hence its image generates a $H^*_{\Tt}(\pt)_{loc}$-subalgebra
\[
H^*_{\Tt}(X)^{Stab}_{loc}\subseteq H^*_{\Tt}(X)_{loc}.
\]
Since both the stable envelopes and the equivariant parameters are even, this subalgebra is concentrated in the even part of the cohomology ring and hence is commutative.

\begin{proposition}
    Let $X$ be an arbitrary bow variety. The localized pushforward
    \[
    i_{*}:  H^*_{\Tt}(X^{\Tt})_{loc}\to  H^*_{\Tt}(X)_{loc}
    \]
    is onto $H^*_{\Tt}(X)^{Stab}_{loc}$. As a consequence, the fixed point localization 
    \[
    i^*: H^*_{\Tt}(X)^{Stab}_{loc}\to  H^*_{\Tt}(X^{\T})_{loc}
    \]
    is an isomorphism.
\end{proposition}
    \begin{proof} Applying the Hanany-Witten isomorphism, we may assume that $X$ is separated. Since adding an NS5 brane with maximal charge induces an isomorphism (which is equivalent to adding an NS5 brane with zero local charge in the co-separated setting), we may also assume that $X$ has no D5 branes with charge equal to $0$. By Corollary \ref{corollary any bow can be embedded in a flag variety} we have a closed embedding $j:X\hookrightarrow \widetilde X= T^*\Fl$, where $\Fl$ is a partial flag variety. This embedding is equivariant along the map $\varphi: \Tt\to \widetilde \Tt$ identifying certain equivariant parameters of $\widetilde \At$ up to some shifts in $\h\in \Cs_{\h}$, cf. Proposition \ref{proposition embedding resolution D5 branes}. In particular, $\varphi$ identifies $\At\subset \Tt$ with a subtorus of the maximal torus $\widetilde \At$ acting on $T^*\Fl$. Equivariant formality of partial flag varieties implies that the map
    \[
    \tilde{i}_{*}: H^*_{\widetilde \Tt}(\widetilde X^{\At})_{loc}\to H^*_{\widetilde \Tt}(\widetilde X)_{loc}
    \]
    associated with the inclusion $\tilde i: \widetilde X^{\At} \hookrightarrow \widetilde X$ is an isomorphism. In particular, the subalgebra $ H^*_{\widetilde \Tt}(\widetilde X)^{Stab}_{loc}$ is in the image of $\tilde{i}_{*}$ . In fact, in this case we have $H^*_{\widetilde \Tt}(\widetilde X)_{loc}= H^*_{\widetilde \Tt}(\widetilde X)^{Stab}_{loc}$ since the composition 
    \[
    H^*_{\widetilde \Tt}(\widetilde X^{\widetilde \At})\xrightarrow[]{\text{Stab}_{\mC}} H^*_{\widetilde \Tt}(\widetilde X)\xrightarrow[]{i^*} H^*_{\widetilde \Tt}(\widetilde X^{\widetilde \At})
    \]
    is an isomorphism after localization \cite[Section 3.3.4]{maulik2012quantum} and $H^*_{\widetilde \Tt}(\widetilde X^{\widetilde \At})$ and $H^*_{\widetilde \Tt}(\widetilde X)$ are free $H_{\widetilde \Tt}(\pt)$-modules of the same rank.
    Taking the cohomological limit of Theorem \ref{Fusion of D5 branes for separated brane diagrams}, we get
    \begin{equation}
    \label{cohomological version of D5 resolution os stable envelopes}
        c(\h)e(N^-)  \StabX{X}{C}{f}= j^* \varphi^*\StabX{\widetilde X}{\widetilde C}{\tilde f_\sharp},
    \end{equation}
    where $c(\h)$ is a nonzero monomial in $\h$ and $e(N^-)\in H^*_{\Tt}(\pt)$ is the Euler class of the negative part of the normal bundle $N$ of $j:X\hookrightarrow  \widetilde X$, which is topologically trivial. Since both $j^*$ and $\varphi^*$ are ring homomorphisms, it follows that the composition of the leftmost and bottom arrows in the solid diagram 
    \[
    \begin{tikzcd}
        H^*_{\widetilde \Tt}(\widetilde X^{\At})_{loc}\arrow[d, swap,  "\tilde{i}_{*}"]\arrow[r, dashed] & H^*_{\Tt}( X^{\At})_{loc} \arrow[d,  "{i}_{*}"]
        \\
        H^*_{\widetilde \Tt}(\widetilde X)^{Stab}_{loc} \arrow[r, "j^*\varphi^*"] & H^*_{\Tt}(X)^{Stab}_{loc}
    \end{tikzcd}
    \]
    is onto. Thus, to prove the proposition, it suffices to show that there exists a dashed map that makes the diagram above well-defined (a priori, the image of $i_{*}$ might not be contained in $H^*_{\Tt}(X)^{Stab}_{loc}$) and commutative. We claim that this map is given by $e(N^+)e(N^-)j_{\At}^*\varphi^*$, where $j_{\At} :X^{\At}\hookrightarrow \widetilde X^{\At} $ is the restriction of $j$ on the fixed locus and $N^\pm$ are the positive and negative parts of the normal bundle $N$.
    
    Fix some $\alpha\in  H^*_{\widetilde \Tt}(\widetilde X^{\At})_{loc}$.  It suffices to assume that $\alpha$ is supported on some fixed component $F\subset \widetilde X^{\At}$. 
    Two cases need to be distinguished. If $F\cap X=\emptyset$, then $j_{\At}^*\varphi^*(\alpha)=0$. On the other hand, by the exact sequence of the pair $(\widetilde X, \widetilde X\setminus F)$, it follows that $\tilde{i}_{*}(\alpha)$ restricts to zero on $ \widetilde X\setminus F\supset X$, and hence $j^*\varphi^*\tilde{i}_{*}(\alpha)=0$, so, in this case, the diagram commutes. If instead $F\cap X\neq \emptyset$, then by Lemma \ref{lemma Fixed components D5 res} $F\cap X=f$, where $f$ is some $\At$-fixed point in $X$.
    Taking into account tubular neighborhoods\footnote{To embed $N$ in $\widetilde X $ as an equivariant tubular neighborhood of $X$, we might need to pass to the maximal compact subgroup inside $\At$, which does not affect the cohomology theory.}, we get the commutative diagram
    \[
    \begin{tikzcd}
        f\arrow[rr, bend left, "j_{\At}"]\arrow[d, "i"] \arrow[r, hookrightarrow]&  N^0_f \arrow[r, hookrightarrow] \arrow[d, hookrightarrow] & F\arrow[d, "\tilde i"]\\
        X \arrow[rr, bend right, "j"]\arrow[r, hookrightarrow]&  N \arrow[r, hookrightarrow] & \widetilde X
    \end{tikzcd}.
    \]
    Here, $N_f^0$ is the normal bundle of $f$ in $F$.
    Since the right square is Cartesian and satisfies $\dim(\widetilde X)-\dim(F)=\dim(N)-\dim(N^0_f)$, the descending compatibility of pushforward and pullback implies that
    \[
    j^*\varphi^*\tilde{i}_{*}(\alpha)=e(N^+)e(N^-)j_{\At}^*\varphi^*(\alpha),
    \]
    as claimed. 
    \end{proof}
    \begin{remark}
        \label{Remark stabs are in image localized pushforward}
        The proof of the previous proposition is only based on formal properties of the cohomology theory and on Theorem \ref{Fusion of D5 branes for separated brane diagrams}. As a consequence, this statement can be directly generalized to K-theory and elliptic cohomology. In particular, it implies that the elliptic stable envelopes, their products, and their linear combinations (whenever the latter are defined since a priori different stable envelopes live in different line bundles) are in the image of the localized elliptic pushforward $i_{\oast}$.
    \end{remark}

    \begin{remark}
        Beyond $H^*_{\Tt}(X)^{Stab}_{loc}$, another interesting commutative subalgebra of $H^*_{\Tt}(X)_{loc}$ is $H^*_{\Tt}(X)^{taut}_{loc}$, the subalgebra generated by the Chern classes of the tautological bundles on $X$. Since the cohomology of any partial flag variety is generated by tautological classes, which are compatible with pullback, equation \eqref{cohomological version of D5 resolution os stable envelopes} implies that the stable envelopes of any bow variety are contained in $H^*_{\Tt}(X)^{taut}_{loc}$. As a result, we have the following chain of inclusions:
        \[
            H^*_{\Tt}(X)^{Stab}_{loc} \subseteq H^*_{\Tt}(X)^{taut}_{loc} \subseteq H^*_{\Tt}(X)_{loc}.
        \]
        In the case of quiver varieties, it is known that these inclusions are actually  equalities\footnote{The second inclusion is actually an isomorphism in the integral theory, i.e. without localizing.}, see \cite{maulik2012quantum, McGerty_2017} and \cite[App. A.4]{schiffmann2022cohomological}. For bow varieties in general this property is not expected to hold. 
    \end{remark}

\printbibliography

\end{document}